\documentclass[11pt]{amsart}

\usepackage{amsmath, amssymb, amscd, cancel, graphicx, soul}
\usepackage{mathdots}

\usepackage[all]{xy}
\usepackage{multirow}
\usepackage{longtable}
\usepackage{array}

\newtheorem{thm}{Theorem}[section]
\newtheorem{conj}[thm]{Conjecture}
\newtheorem{cor}[thm]{Corollary}
\newtheorem{lem}[thm]{Lemma}
\newtheorem{prop}[thm]{Proposition}

\newtheorem{defin}[thm]{Definition}

\def\E{{\mathcal E}}

\def\T{{\mathcal T}}
\def\Z{{\mathbb Z}}

\def\L{{\langle}}
\def\R{{\rangle}}

\def\del{{\partial}}
\def\disc{{\textup{disc}}}

\def\mod{{\textup{mod} \;}}
\def\rk{{\mathrm{rk}}}
\def\supp{{\textup{supp}}}

\newcommand{\into}{\hookrightarrow}

\begin{document}

\title[The lens space realization problem]%
{The lens space realization problem}

\author[Joshua Evan Greene]{Joshua Evan Greene}

\address{Department of Mathematics, Columbia University\\ New York, NY 10027}

\thanks{Partially supported by an NSF Post-doctoral Fellowship.}

\email{josh@math.columbia.edu}

\maketitle

\begin{abstract}

We determine the lens spaces that arise by integer Dehn surgery along a knot in the three-sphere.  Specifically, if surgery along a knot produces a lens space, then there exists an equivalent surgery along a Berge knot with the same knot Floer homology groups.  This leads to sharp information about the genus of such a knot.  The arguments rely on tools from Floer homology and lattice theory, and are primarily combinatorial in nature.

\end{abstract}

{\it\hfill Dedicated to the memory of Professor Michael Moody}


\section{Introduction}  What are all the ways to produce the simplest closed 3-manifolds by the simplest 3-dimensional topological operation?  From the cut-and-paste point of view, the simplest 3-manifolds are the {\em lens spaces} $L(p,q)$, these being the spaces (besides $S^3$ and $S^1 \times S^2$) that result from identifying two solid tori along their boundaries, and the simplest operation is {\em Dehn surgery} along a knot $K \subset S^3$.   With these meanings in place, the opening question goes back forty years to Moser \cite{moser:torus}, and its definitive answer remains unknown.

By definition, a {\em lens space knot} is a knot $K \subset S^3$ that admits a lens space surgery.  Moser observed that all torus knots are lens space knots and classified their lens space surgeries.  Subsequently, Bailey-Rolfsen \cite{baileyrolfsen} and Fintushel-Stern \cite{fs:lenssurgery} gave more examples of lens space knots.  The production of examples culminated in an elegant construction due to Berge that at once subsumed all the previous ones and generated many more classes \cite{berge:lens}.  Berge's examples are the knots that lie on a Heegaard surface $\Sigma$ of genus two for $S^3$ and that represent a primitive element in the fundamental group of each handlebody.  For this reason, such knots are called {\em doubly primitive}.   Berge observed that performing surgery along such a knot $K$, with (integer) framing specified by a push-off of $K$ on $\Sigma$, produces a lens space.  Furthermore, he enumerated several different types of doubly primitive knots.  By definition, the {\em Berge knots} are the doubly primitive knots that Berge specifically enumerated in \cite{berge:lens}.  They are reproduced in Subsection \ref{ss: list} (more precisely, the {\em dual} Berge knots are reported there).

The most prominent question concerning lens space surgeries is the {\em Berge conjecture}.

\begin{conj}[Problem 1.78, \cite{kirby:problems}]\label{conj: berge}

If integer surgery along a knot $K \subset S^3$ produces a lens space, then it arises from Berge's construction.

\end{conj}

Complementing Conjecture \ref{conj: berge} is the cyclic surgery theorem of Culler-Gordon-Luecke-Shalen \cite{cgls:cyclic}, which implies that if a lens space knot $K$ is not a torus knot, then the surgery coefficient is an integer.  Therefore, an affirmative answer to the Berge conjecture would settle Moser's original question.  We henceforth restrict attention to {\em integer} slope surgeries as a result.  Reflecting $K$ if necessary, we may further assume that the slope is positive.  Thus, in what follows, we attach to every lens space knot $K$ a positive integer $p$ for which $p$-surgery along $K$ produces a lens space, and denote the surgered manifold by $K_p$.

Using monopole Floer homology, Kronheimer-Mrowka-Ozsv\'ath-Szab\'o related the knot genus and the surgery slope via the inequality \begin{equation}\label{e: kmos} 2g(K)-1 \leq p \end{equation} \cite[Corollary 8.5]{kmos}.  Their argument utilizes the fact that the Floer homology of a lens space is as simple as possible: $\rk \; \widehat{HF}(Y) = |H_1(Y;\Z)|$.  A space with this property is called an {\em L-space}, and a knot with a positive L-space surgery is an {\em L-space knot}.  Their proof adapts to the setting of Heegaard Floer homology as well \cite{os:absgr}, the framework in place for the remainder of this paper.  Ozsv\'ath-Szab\'o established a significant constraint on the knot Floer homology groups $\widehat{HFK}(K)$ and hence the Alexander polynomial $\Delta_K$ \cite[Theorem 1.2 and Corollary 1.3]{os:lens}.  Utilizing this result, Ni proved that $K$ is fibered \cite[Corollary 1.3]{ni:fibered}.

As indicated by Berge, it is often preferable to take the perspective of surgery along a knot in a lens space.  Corresponding to a lens space knot $K \subset S^3$ is a dual knot $K' \subset K_p$, the core of the surgery solid torus.  Reversing the surgery, it follows that $K'$ has a positive integer surgery producing $S^3$.  Following custom, we refer to the dual of a Berge knot as a Berge knot as well, and stress the ambient manifold to prevent confusion.  As demonstrated by Berge \cite[Theorem 2]{berge:lens}, the dual to a doubly primitive knot takes a particularly pleasant form: it is an example of a {\em simple knot}, of which there is a unique one in each homology class in $L(p,q)$.  Thus, each Berge knot in a lens space is specified by its homology class, and this is what we report in Subsection \ref{ss: list}.  This point of view is taken up by Baker-Grigsby-Hedden \cite{bgh:lens} and J. Rasmussen \cite{r:Lspace}, who have proposed programs to settle Conjecture \ref{conj: berge} by studying knots in lens spaces with simple knot Floer homology.


\subsection{Results.}\label{ss: results}

\noindent  A derivative of the Berge conjecture is the {\em realization problem}, which asks for those lens spaces that arise by integer surgery along a knot in $S^3$.  Closely related is the question of whether the Berge knots account for all the doubly primitive knots.  Furthermore, the Berge conjecture raises the issue of tightly bounding the knot genus $g(K)$ from above in terms of the surgery slope $p$.  The present work answers these three questions.  

\begin{thm}\label{t: main}

Suppose that positive integer surgery along a knot $K \subset L(p,q)$ produces $S^3$.  Then $K$ lies in the same homology class as a Berge knot $B \subset L(p,q)$.

\end{thm}

\noindent The resolution of the realization problem follows at once.  As explained in Section \ref{s: completion}, the same result holds with $S^3$ replaced by any L-space homology sphere with $d$-invariant $0$.  As a corollary, we obtain the following result.

\begin{thm}\label{c: main}

Suppose that $K \subset S^3$, $p$ is a positive integer, and $K_p$ is a lens space.  Then there exists a Berge knot $B \subset S^3$ such that $B_p \cong K_p$ and $\widehat{HFK}(B) \cong \widehat{HFK}(K)$.  Furthermore, every doubly primitive knot in $S^3$ is a Berge knot.

\end{thm}

\noindent Based on well-known properties of the knot Floer homology groups, it follows that $K$ and $B$ have the same Alexander polynomial, genus, and four-ball genus.  Furthermore, the argument used to establish Theorem \ref{t: main} leads to a tight upper bound on the knot genus $g(K)$ in relation to the surgery slope.

\begin{thm}\label{t: berge bound}

Suppose that $K \subset S^3$, $p$ is a positive integer, and $K_p$ is a lens space.  Then
\begin{equation}\label{e: berge bound}
2g(K)-1 \leq p - 2 \sqrt{(4p+1)/5},
\end{equation}
unless $K$ is the right-hand trefoil and $p = 5$.  Moreover, this bound is attained by the type VIII Berge knots specified by the pairs $(p,k) = (5n^2+5n+1, 5n^2-1)$.

\end{thm}

\noindent Theorem \ref{t: berge bound} was announced without proof in \cite[Theorem 1.2]{greene:cabling} (cf. \cite{saito:lens}).  As indicated in \cite{greene:cabling}, for $p \gg 0$, Theorem \ref{t: berge bound} significantly improves on the bound $2g(K)-1 \leq p-9$ conjectured by Goda-Teragaito \cite{gt:lenssurgery} for a hyperbolic knot $K$, and can be used to show that the conjectured bound holds for all but at most two values $p \in \{ 14, 19 \}$.  In addition, one step involved in both approaches to the Berge conjecture outlined in \cite{bgh:lens,r:Lspace} is to argue the non-existence of a non-trivial knot $K$ for which $K_{2g(K)-1}$ is a lens space.  This fact follows immediately from Theorem \ref{t: berge bound}.


\subsection{Berge knots in lens spaces.}\label{ss: list}

J. Rasmussen concisely tabulated the Berge knots $B \subset L(p,q)$ \cite[Section 6.2]{r:Lspace}.  
To describe those with a {\em positive} $S^3$ surgery, select a positive integer $k$ and produce a positive integer $p$ in terms of it according to the table below.  The value $k \pmod p$ represents the homology class of $B$ in $H_1(L(p,q)) \cong \Z / p \Z$, $q \equiv -k^2 \pmod p$, as described at the end of Section \ref{s: topology}.  We reproduce the tabulation here.

\smallskip

\noindent {\bf Berge Type I$_{\pm}$:} $p = ik \pm 1, \quad \gcd(i,k) = 1$;

\noindent {\bf Berge Type II$_{\pm}$:} $p = ik \pm 1, \quad \gcd(i,k) = 2$, $i,k \geq 4$;

\noindent {\bf Berge Type III:} $ \begin{cases} (a)_{\pm} \quad p \equiv \pm(2k-1)d \pmod{k^2}, \quad d \; | \; k+1, \; {k+1 \over d} \text{ odd}; \\ (b)_{\pm} \quad p \equiv \pm(2k+1)d \pmod{k^2}, \quad d \; | \; k-1, \; {k-1 \over d} \text{ odd};  \end{cases}$

\noindent {\bf Berge Type IV:} $ \begin{cases} (a)_{\pm} \quad p \equiv \pm(k-1)d \pmod{k^2}, \quad d \; | \; 2k+1; \\ (b)_{\pm} \quad p \equiv \pm(k+1)d \pmod{k^2}, \quad d \; | \; 2k-1;  \end{cases}$

\noindent {\bf Berge Type V:} $ \begin{cases} (a)_{\pm} \quad p \equiv \pm(k+1)d  \pmod{k^2}, \quad d \; | \; k+1, d \text{ odd}; \\ (b)_{\pm} \quad p \equiv \pm(k-1)d \pmod{k^2}, \quad d \; | \; k-1, d \text{ odd};  \end{cases}$

\noindent 

\noindent {\bf Berge Type VII:} $k^2+k+1 \equiv 0 \pmod{p}$;

\noindent {\bf Berge Type VIII:} $k^2-k-1 \equiv 0 \pmod{p}$;

\noindent {\bf Berge Type IX:} $p = {1 \over 11}(2k^2 + k + 1), k \equiv 2 \pmod{11}$;

\noindent {\bf Berge Type X:} $p = {1 \over 11}(2k^2 + k + 1), k \equiv 3 \pmod{11}$.

\smallskip

\noindent As indicated by J. Rasmussen, type VI occurs as a special case of type V, and types XI and XII result from allowing negative values for $k$ in IX and X, respectively.


\subsection{Overview and organization.}\label{ss: methods}  We now provide a detailed overview of the general strategy we undertake to establish the main results.  We hope that this account will satisfy the interests of most readers and clarify the intricate combinatorial arguments that occupy the main body of the text.

Our approach draws inspiration from a remarkable pair of papers by Lisca \cite{lisca:lens1,lisca:lens2}, in which he classified the sums of lens spaces that bound a smooth, rational homology ball.  Lisca began with the observation that the lens space $L(p,q)$ naturally bounds a smooth, negative definite plumbing 4-manifold $X(p,q)$ (Section \ref{s: topology}).  If $L(p,q)$ bounds a rational ball $W$, then the 4-manifold $Z := X(p,q) \cup -W$ is a smooth, closed, negative definite 4-manifold with $b_2(Z) = b_2(X) =: n$.  According to Donaldson's celebrated ``Theorem A", the intersection pairing on $H_2(Z;\Z)$ is isomorphic to {\em minus} the standard Euclidean integer lattice $-\Z^n$ \cite{d:thma}.  As a result, it follows that the intersection pairing on $X(p,q)$, which we henceforth denote by $-\Lambda(p,q)$, embeds as a full-rank sublattice of $-\Z^n$.  Lisca solved the combinatorial problem of determining the pairs $(p,q)$ for which there exists an embedding $\Lambda(p,q) \into \Z^n$, subject to a certain additional constraint on the pair $(p,q)$.  By consulting an earlier tabulation of Casson-Gordon \cite[p. 188]{cg:cobordism}, he observed that the embedding exists iff $\pm L(p,q)$ belongs to a family of lens spaces already known to bound a special type of rational ball.  The classification of lens spaces that bound rational balls follows at once.  Pushing this technique further, Lisca obtained the classification result for sums of lens spaces as well.

In our situation, we seek the pairs $(p,q)$ for which $L(p,q)$ arises as positive integer surgery along a knot $K \subset S^3$.  Thus, suppose that $K_p \cong L(p,q)$, and form a smooth 4-manifold $W_p(K)$ by attaching a $p$-framed 2-handle to $D^4$ along $K \subset \del D^4$.  This space has boundary $K_p$, so we obtain a smooth, closed, negative definite 4-manifold by setting $Z = X(p,q) \cup - W_p(K)$, where $b_2(Z) = n+1$.  By Donaldson's Theorem, it follows that $\Lambda(p,q)$ embeds as a codimension one sublattice of $\Z^{n+1}$.  However, this restriction is too weak: it is easy to produce pairs $(p,q)$ that fulfill this condition, while $L(p,q)$ does not arise as a knot surgery (for example, $L(33,2)$).

Thankfully, we have another tool to work with: the {\em correction terms} in Heegaard Floer homology (\cite[Section 2]{greene:cabling}).  Ozsv\'ath-Szab\'o defined these invariants and subsequently used them to phrase a necessary condition on the pair $(p,q)$ in order for $L(p,q)$ to arise as a positive integer surgery \cite[Corollary 7.5]{os:absgr}.  Using a computer, they showed that this condition is actually sufficient for $p \leq 1500$: every pair that fulfills it appears on Berge's list \cite[Proposition 1.13]{os:lens}.  Later, J. Rasmussen extended this result to all $p \leq 100,000$ \cite[end of Section 6]{r:Lspace}.  Following their work, it stood to reason that the Ozsv\'ath-Szab\'o condition is both necessary and sufficient for {\em all} $(p,q)$.  However, it remained unclear how to manipulate the correction terms effectively towards this end.

The key idea here is to use the correction terms in tandem with Donaldson's Theorem.  The result is an enhanced lattice embedding condition (cf. \cite{greene:3braids, greene:cabling, gj:slicepretzels}).  In order to state it,  we first require a combinatorial definition.

\begin{defin}\label{d: change}

A vector $\sigma = (\sigma_0, \dots, \sigma_n) \in \Z^{n+1}$ with $1 = \sigma_0 \leq \sigma_1 \leq \cdots \leq \sigma_n$ is a {\em changemaker} if for all $0 \leq k \leq \sigma_0 + \cdots + \sigma_n$, there exists a subset $A \subset \{0,\dots,n\}$ such that $\sum_{i \in A} \sigma_i = k$.  Equivalently, 
$\sigma_i \leq \sigma_0 + \cdots + \sigma_{i-1} + 1$ for all $1 \leq i \leq n$.

\end{defin}

\noindent If we imagine the $\sigma_i$ as values of coins, then Definition \ref{d: change} asserts a necessary and sufficient condition under which one can make exact change from the coins in any amount up to their total value.  The reader may find it amusing to establish this condition; its proof appears in both \cite{brown:changemaker} and \cite[Lemma 3.2]{greene:cabling}.  Note that Definition \ref{d: change} differs slightly from the one used in \cite{greene:cabling}, since here we require that $\sigma_0 = 1$.

Our lattice embedding condition now reads as follows.  Again, we phrase it from the perspective of surgery along a knot in a lens space.

\begin{thm}\label{t: main technical}

Suppose that positive integer surgery along a knot $K \subset L(p,q)$ produces $S^3$.  Then $\Lambda(p,q)$ embeds as the orthogonal complement to a changemaker $\sigma \in \Z^{n+1}$, $n = b_2(X)$.

\end{thm}

\noindent Our strategy is now apparent: determine the list of pairs $(p,q)$ which pass this refined embedding obstruction, and check that it coincides with Berge's list.  Indeed, this is the case.

\begin{thm}\label{t: linear changemakers}

At least one of the pairs $(p,q)$, $(p,q')$, where $q q' \equiv 1 \pmod 1$, appears on Berge's list iff $\Lambda(p,q)$ embeds as the orthogonal complement to a changemaker in $\Z^{n+1}$.  

\end{thm}

\noindent Furthermore, when $\Lambda(p,q)$ embeds, we recover a value $k \pmod p$ that represents the homology class of a Berge knot $K \subset L(p,q)$ (Proposition \ref{p: homology}).  Theorem \ref{t: main} follows easily from this result.

To give a sense of the proof of Theorem \ref{t: linear changemakers}, we first reflect on the lattice embedding problem that Lisca solved.  He made use of the fact that $\Lambda(p,q)$ admits a special basis; in our language, it is a {\em linear lattice} with a distinguished {\em vertex basis} (Subsection \ref{ss: linear lattice}).  He showed that any embedding of a linear lattice as a full-rank sublattice of $\Z^n$ (subject to the extra constraint he posited) can be built from one of a few small embeddings by repeatedly applying a basic operation called {\em expansion}.    Following this result, the identification of the relevant pairs $(p,q)$ follows from a manipulation of continued fractions.  The precise details of Lisca's argument are involved, but ultimately elementary and combinatorial in nature.

One is tempted to carry out a similar approach to Theorem \ref{t: linear changemakers}.  Thus, one might first attempt to address the problem of embedding $\Lambda(p,q)$ as a codimension one sublattice of $\Z^n$, and then analyze which of these are complementary to a changemaker.  However, getting started in this direction is difficult, since Lisca's techniques do not directly apply.

More profitable, it turns out, is to turn this approach on its head.  Thus, we begin with a study of the lattices of the form $(\sigma)^\perp \subset \Z^n$ for some changemaker $\sigma$; by definition, these are the  {\em changemaker lattices}.  A changemaker lattice is best presented in terms of its {\em standard basis}  (Subsection \ref{ss: changemaker}).  The question then becomes: when is a changemaker lattice isomorphic to a linear lattice?  That is, how do we recognize whether there exists a change of basis from its standard basis to a vertex basis?

The key notion in this regard is that of an {\em irreducible} element in a lattice $L$.  By definition, an element $x \in L$ is {\em reducible} if $x = y+z$, where $y,z \in L$ are non-zero and $\L y,z \R \geq 0$; it is irreducible otherwise. Here $\L\, ,\, \R$ denotes the pairing on $L$.  As we show, the standard basis elements of a changemaker lattice are irreducible (Lemma \ref{l: change irred}), as are the vertex basis elements of a linear lattice.  Furthermore, the irreducible elements in a linear lattice take a very specific form (Corollary \ref{c: linear irred}).  This leads to a variety of useful Lemmas, collected in Subsection \ref{ss: int graph}.  For example, if a changemaker lattice is isomorphic to a linear lattice, then its standard basis does not contain three elements, each of norm $\geq 3$, such that any two pair together non-trivially (Lemma \ref{l: no triangle}).

Thus, we proceed as follows.  First, choose a standard basis $S \subset \Z^n$ for a changemaker lattice $L$, and suppose that $L$ is isomorphic to a linear lattice.  Then apply the combinatorial criteria of Subsection \ref{ss: int graph} to deduce the specific form that $S$ must take.  Standard basis elements come in three distinct flavors -- {\em gappy}, {\em tight}, and {\em just right} (Definition \ref{d: tight, gappy, just right}) -- and our case analysis decomposes according to whether $S$ contains no gappy or tight vectors (Section \ref{s: just right}), a gappy vector but not a tight one (Section \ref{s: gappy, no tight}), or a tight vector (Section \ref{s: tight}).  In addition, Section \ref{s: decomposable} addresses the case in which $L$ is isomorphic to a (direct) sum of linear lattices.  This case turns out the easiest to address, and the subsequent Sections \ref{s: just right} - \ref{s: tight} rely on it, while increasing in order of complexity.

The net result of Sections \ref{s: decomposable} - \ref{s: tight} is a collection of several structural Propositions that enumerate the possible standard bases for a changemaker lattice isomorphic to a linear lattice or a sum thereof. Section \ref{s: cont fracs} takes up the problem of converting these standard bases into vertex bases, extracting the relevant pairs $(p,q)$ for each family of linear lattices, as well as the value $k \pmod p$ of Proposition \ref{p: homology}.  Here, as in Lisca's work, we make some involved calculations with continued fractions.  Table \ref{table: A} gives an overview of the correspondence between the structural Propositions and the Berge types.  Lastly, Section \ref{s: completion} collects the results of the earlier Sections to prove the Theorems stated above.

The remaining introductory Sections discuss various related topics.


\subsection{Related progress.}\label{ss: partial}

A number of authors have recently addressed both the realization problem and the classification of doubly primitive knots in $S^3$.  S. Rasmussen established Theorem \ref{t: main} under the constraint that $k^2 < p$ \cite[Theorem 1.0.3]{sr:thesis}.  This condition is satisfied precisely by  Berge types I-V.  Tange established Theorem \ref{t: main} under a different constraint relating the values $k$ and $p$ \cite[Theorem 6]{tange:realization}.  The two constraints are not complementary, however, so the full statement of Theorem \ref{t: main} does not follow on combination of these results.  In a different direction, Berge showed by direct topological methods that all doubly primitive knots are Berge knots; equivalently, a simple knot in a lens space has an $S^3$ surgery iff it is a (dual) Berge knot \cite{berge:ppknots}.


\subsection{Comparing Berge's and Lisca's lists.}\label{ss: berge lisca}

Lisca's list of lens spaces that bound rational balls bears a striking resemblance to the list of Berge knots of type I-V (Subsection \ref{ss: list}).  J. Rasmussen has explained this commonality by way of the knots $K$ in the solid torus $S^1 \times D^2$ that possess integer $S^1 \times D^2$ surgeries. The classification of these knots is due to Berge and Gabai \cite{berge:solidtorus,gabai:solidtorus}.  Given such a knot, we obtain a knot $K' \subset S^1 \times S^2$ via the standard embedding $S^1 \times D^2 \subset S^1 \times S^2$.  Performing the induced surgery along $K'$ produces a lens space $L(p,q)$, which we can effect by attaching a 2-handle along $K' \subset \del (S^1 \times D^3)$.  The resulting 4-manifold is a rational ball built from a single 0-, 1-, and 2-handle, and it has boundary $L(p,q)$.  As observed by J. Rasmussen, Lisca's theorem shows that every lens space that bounds a rational ball must bound one built in this way.  On the other hand, we obtain a knot $K'' \subset S^3$ from $K$ via the standard embedding $S^1 \times D^2 \subset S^3$.  Performing the induced surgery along $K''$ produces a lens space $L(r,s)$ and a dual knot representing some homology class $k \pmod r$ .  The Berge knots of type I-V arise in this way.  The pair $(p,q)$ comes from setting $p = k^2$ and $q = r$; in this way, we reconstruct Lisca's list (but not his result!) from Berge's.

Analogous to the Berge conjecture, Lisca's theorem raises the following conjecture.

\begin{conj}

If a knot in $S^1 \times S^2$ admits an integer lens space surgery, then it arises from a knot in $S^1 \times D^2$ with an integer $S^1 \times D^2$ surgery.

\end{conj}


\subsection{4-manifolds with small $b_2$.}  

Which lens spaces bound a smooth 4-manifold built from a single 0- and 2-handle?  Theorem \ref{t: main} can be read as answer to this question.  Which lens spaces bound a smooth, simply-connected 4-manifold $W$ with $b_2(W) = 1$?  Are there examples beyond those coming from Theorem \ref{t: main}?  The answers to these questions are unknown.  By contrast, the situation in the topological category is much simpler: a lens space $L(p,q)$ bounds a topological, simply-connected 4-manifold with $b_2 = b^+_2=1$ iff $-q$ is a square $(\mod p)$.

Similarly, we ask: which lens spaces bound a smooth rational homology ball?  Which bound one built from a single 0-, 1-, and 2-handle?  As addressed in Subsection \ref{ss: berge lisca}, the answers to these two questions are the same.  Furthermore, Lisca showed that a two-bridge link is smoothly slice if and only if its branched double-cover (a lens space) bounds a smooth rational homology ball.

Which lens spaces bound a topological rational homology ball?  The answer to this question is unknown.  For that matter, it is unknown which two-bridge links $L$ are topologically slice.  Note that if $L$ is topologically slice, then the lens space that arises as its branched double-cover bounds a topological rational homology ball.  However, the converse is unknown: is it the case that a lens space bounds a topological rational homology ball iff the corresponding two-bridge link is topologically slice?  Are the answers to these questions the same as in the smooth category?


\subsection{The Poincar\'e sphere.}\label{ss: poincare}

Tange constructed several families of simple knots in lens spaces with integer surgeries producing the Poincar\'e sphere $P^3$ \cite[Section 5]{tange:poincare}.  J. Rasmussen verified that Tange's knots account for all such simple knots in $L(p,q)$ with $|p| \leq 100,000$ \cite[end of Section 6]{r:Lspace}.  Furthermore, he observed that in the homology class of each type VII Berge knot, there exists a $(1,1)$-knot $T_L$ as constructed by Hedden \cite[Figure 3]{hedden:berge}, and it admits an integer $P^3$-surgery for values $p \leq 39$ \cite[end of Section 5]{r:Lspace}.  Combining conjectures of Hedden \cite[Conjecture 1.7]{hedden:berge} and J. Rasmussen \cite[Conjecture 1]{r:Lspace}, it would follow that Tange's knots and the knots $T_L$ homologous to type VII Berge knots are precisely the knots in lens spaces with an integer $P^3$-surgery.  Conjecture \ref{conj: poincare} below is the analogue to the realization problem in this setting.

\begin{conj}\label{conj: poincare}

Suppose that integer surgery along a knot $K \subset L(p,q)$ produces $P^3$.\footnote{Or, more generally, any L-space homology sphere with $d$-invariant $-2$.}  Then either $2g(K) - 1 < p$, and $K$ lies in the same homology class as a Tange knot, or else $2g(K)-1 = p$, and $K$ lies in the same homology class as a Berge knot of type VII.

\end{conj}

\noindent Tange has obtained partial progress on Conjecture \ref{conj: poincare} \cite{tange:realization}.  The methodology developed here to establish Theorem \ref{t: main} suggests a similar approach to Conjecture \ref{conj: poincare}, making use of an unpublished variant on Donaldson's theorem due to Fr{\o}yshov \cite[Proposition 2 and the remark thereafter]{froyshov:poincare}.  Lastly, we remark that the determination of {\em non}-integral $P^3$-surgeries along knots in lens spaces seems tractable, although it falls outside the scope of the cyclic surgery theorem.


\section*{Acknowledgments} Thanks to John Baldwin for sharing the meal of paneer bhurji that kicked off this project, and to him, Cameron Gordon, Matt Hedden, John Luecke, and Jake Rasmussen for helpful conversations.  Paolo Lisca's papers \cite{lisca:lens1,lisca:lens2} and a lecture by Dusa McDuff on her joint work with Felix Schlenk \cite{ms:ellipsoids} were especially influential along the way.  The bulk of this paper was written at the Mathematical Sciences Research Institute in Spring 2010.  Thanks to everyone connected with that institution for providing an ideal working environment.


\section{Topological Preliminaries}\label{s: topology}

Given relatively prime integers $p > q > 0$, the lens space $L(p,q)$ is the oriented manifold obtained from $- p/q$ Dehn surgery along the unknot.  It bounds a plumbing manifold $X(p,q)$, which has the following familiar description.  Expand $p/q$ in a Hirzebruch-Jung continued fraction
\[ p/q = [a_1,a_2,\dots,a_n]^- = a_1 - \cfrac{1}{a_2 - \cfrac{1}{\ddots \\ - \cfrac{1}{a_n}}} \; ,\] 
with each $a_i$ an integer $\geq 2$.  Form the disk bundle $X_i$ of Euler number $-a_i$ over $S^2$, plumb together $X_i$ and $X_{i+1}$ for $i = 1, \dots, n-1$, and let $X(p,q)$ denote the result.  The manifold $X(p,q)$ is {\em sharp} \cite[Section 2]{greene:cabling}.  It also admits a Kirby diagram given by the framed chain link $\mathbb{L} = L_1 \cup \cdots \cup L_n \subset S^3$, in which each $L_i$ is a planar unknot framed by coefficient $-a_i$, oriented so that consecutive components link once positively. To describe the intersection pairing on $X(p,q)$, we state a definition.

\begin{defin}\label{d: linear lattice}

The {\em linear lattice} $\Lambda(p,q)$ is the lattice freely generated by elements $x_1,\dots,x_n$ with inner product given by \[ \langle x_i, x_j \rangle =  \begin{cases} a_i, & \text{ if } i=j; \\ -1, & \text{ if } |i-j|=1; \\ 0, & \text{ if } |i-j| > 1. \end{cases}\]

\end{defin}
\noindent  A more detailed account about lattices (in particular, the justification for calling $\Lambda(p,q)$ a lattice) appears in Section \ref{s: lattices}.  It follows at once that the inner product space $H_2(X(p,q),Q_X)$ equals {\em minus} $\Lambda(p,q)$; here and throughout, we take homology groups with integer coefficients.  We note that $p/q' = [a_n,\dots,a_1]^-$, where $0 < q' < p$ and $q q' \equiv 1 \pmod p$ (Lemma \ref{l: cont frac basics}(4)).  Thus, we obtain $\Lambda(p,q) \cong \Lambda(p,q')$ on the algebraic side, and $L(p,q) \cong L(p,q')$ on the topological side (cf. Proposition \ref{p: gerstein}).

Now suppose that positive integer surgery along a knot $K \subset L(p,q)$ produces $S^3$.  Let $W$ denote the associated 2-handle cobordism from $L(p,q)$ to $S^3$, capped off with a $4$-handle.  Orienting $K$ produces a canonical generator $[\Sigma] \in H_2(-W)$ defined by the condition that $\L [C],[\Sigma] \R = +1$, where $C$ denotes the core of the 2-handle attachment.  Form the closed, oriented, smooth, negative definite 4-manifold $Z = X(p,q) \cup - W$.  By \cite[Theorem 3.3]{greene:cabling}, it follows that $\Lambda(p,q)$ embeds in the orthogonal complement $(\sigma)^\perp \subset \Z^{n+1}$, where the changemaker $\sigma$ corresponds to the class $[\Sigma]$.  {\em A priori} $\sigma$ could begin with a string of zeroes as in \cite{greene:cabling}, but Theorem \ref{t: main technical} rules this out, and moreover shows that $\Lambda(p,q) \cong (\sigma)^\perp$.  We establish Theorem \ref{t: main technical} once we develop a bit more about lattices (cf. Subsection \ref{ss: changemaker}), and we make use of it in the remainder of this section.

We now focus on the issue of recovering the homology class $[K] \in H_1(L(p,q))$ from this embedding.  Regard $\mathbb{L}$ as a surgery diagram for $L(p,q)$, and let $\mu_i, \lambda_i \subset \del (nd(L_i))$ denote a meridian, Seifert-framed longitude pair for $L_i$, oriented so that $\mu_i \cdot \lambda_i = +1$.  Let $T_i$ denote the $i^{th}$ surgery solid torus.  The boundary of $T_n$ is a Heegaard torus for $L(p,q)$; denote by $a$ the core of $T_n$ and by $b$ the core of the complementary solid torus $T_n'$.  We compute the self-linking number of $b$ as $-q'/p \pmod 1$ (cf. \cite[Section 2]{r:Lspace}, bearing in mind the opposite orientation convention in place there).  Thus, if $[K] = \pm k [B]$, then the self-linking number of $K$ is $- k^2 q'/p \pmod 1$.  The condition that $K$ has a positive integer homology sphere surgery amounts to the condition that $-k^2 q' \equiv 1 \pmod p$ (ibid.), from which we derive $q \equiv -k^2 \pmod p$.

Define
\begin{equation}\label{e: x}
x := \sum_{i=1}^n p_{i-1} x_i \in \Lambda(p,q),
\end{equation}
where the values $p_i$ are inductively defined by $p_{-1} = 0$, $p_0 = 1$, and $p_i = a_i p_{i-1} - p_{i-2}$ (cf. Definition \ref{d: cont frac} and Lemma \ref{l: cont frac basics}(1)).  We identify the elements $x_i$ and $x$ with their images under the embedding $\Lambda(p,q) \oplus (\sigma) \into \Z^{n+1}$.

\begin{prop}\label{p: homology}

Suppose that positive integer surgery along the oriented knot $K \subset L(p,q)$ produces $S^3$, let $\Lambda(p,q) \oplus (\sigma) \into \Z^{n+1}$ denote the corresponding embedding, and set $k = \L e_0, x \R$.  Then
\[ [K] = k \, [b] \in H_1(L(p,q)).\]
\end{prop}

\begin{proof}

(I) We first express the homology class of a knot $\kappa \subset L(p,q)$ from the three-dimensional point of view. To that end, we construct a compressing disk $D \subset T_n'$ which is related to the class $x$.  Let $P_i$ denote the planar surface in $S^3 - L$ with $[\del P_i] = [\lambda_i] -  [\mu_{i-1}] - [\mu_{i+1}]$ (taking $\mu_0, \mu_{n+1} = \varnothing$).  Choose $p_{i-1}$ disjoint copies of $P_i$, and form the oriented cut-and-paste $P$ of all these surfaces.  We calculate
\begin{eqnarray*}
[\del P] &=& \sum_{i=1}^n p_{i-1} [\del P_i] = \sum_{i=1}^n p_{i-1} ([\lambda_i] - [\mu_{i-1}] - [\mu_{i+1}]) \\
&=& \sum_{i=1}^{n-1} (p_{i-1}[\lambda_i] - (p_{i-2} + p_i)[\mu_i]) \; + \; (p_{n-1}[\lambda_n] - p_{n-2}[\mu_n]) \\
&=& \sum_{i=1}^{n-1} p_{i-1}[\lambda_i - a_i \mu_i] \; + \; (p_{n-1}[\lambda_n] - p_{n-2}[\mu_n]).
\end{eqnarray*}
Let $D_i$ denote a compressing disk for $T_i$.  Since $[\del D_i] = [\lambda_i - a_i \mu_i]$, it follows that we can form the union of $P$ with $p_{i-1}$ copies of $-D_i$ for $i = 1,\dots, n-1$ to produce a properly embedded, oriented surface $D \subset T_n'$.  The boundary $\del D$ represents $p_{n-1}[\lambda_n] - p_{n-2}[\mu_n] \in H_1(\del (nd (L_n)))$, and since $\gcd(p_{n-1},p_{n-2}) = 1$, it follows that $\del D$ has a single component.  Furthermore, a simple calculation shows that $\chi(D) = 1$.  Therefore, $D$ provides the desired compressing disk.

Since $b \cdot D = 1$, we calculate
\begin{equation}\label{e: kappa class}
[\kappa] = (\kappa \cdot D) [b] \in H_1(L(p,q))
\end{equation}
for an oriented knot $\kappa \subset L(p,q)$ supported in $T_n'$ and transverse to $D$.

\noindent (II) Now we pull in the four-dimensional point of view.  Given $\alpha \in H_2(X, \del X)$, represent the class $\del_* \alpha$ by a knot $\kappa \subset L(p,q)$, isotop it into the complement of the surgery tori $T_1 \cup \cdots \cup T_n$, and regard it as knot in $\del D^4$.  Choose a Seifert surface for it and push its interior slightly into $D^4$, producing a surface $F$.

Consider the Kirby diagram of $X$.  We can represent the class $x$ by the sphere $\mathcal{S}$ obtained by pushing the interior of $P$ slightly into $D^4$, producing a surface $P'$, and capping off $\del P'$ with $p_{i-1}$ copies of the core of handle attachment along $L_i$, for $i= 1,\dots,n$.  It is clear that
\begin{equation}\label{e: int nos}
\kappa \cdot D = \kappa \cdot P = F \cdot P' = F \cdot \mathcal{S} = \L [F] , x \R.
\end{equation}

Since $\del_* \alpha = \del_* [F] = [\kappa]$, it follows that $\alpha -[F]$ represents an absolute class in $H_2(X)$.  Since the pairing $H_2(X) \otimes H_2(X) \to \Z$ takes values in $|H_1(\del X)| \cdot \Z$, it follows that $\L \alpha , x \R \equiv \L [F] , x \R \pmod{p}$.  Comparing with \eqref{e: kappa class} and \eqref{e: int nos}, we obtain
\begin{equation}\label{e: alpha class}
\del_*\alpha = \L \alpha , x \R [b].
\end{equation}

\noindent (III)  At last we use the 2-handle cobordism $W$ and the closed manifold $Z$.  Given $\beta \in H_2(-W, - \del W)$, write $\beta = n [C]$, where $C$ denotes the core of the handle attachment along $K$.  Since $\del_*[C] = [K]$, it follows that
\begin{equation}\label{e: beta class}
\del_*\beta = \L \beta , [\Sigma] \R [K].
\end{equation}

Finally, consider the commutative diagram
\[ \xymatrix @R=.5pc{ 
 &  H_2(Z,-W) \ar[r]^\sim_{exc.} & H_2(X,\del X) \ar[dr]^{\del_*} &  \cr
H_2(Z) \ar[ur] \ar [dr] & & & H_1(L(p,q))\cr
&  H_2(Z,X) \ar[r]^(.42)\sim_(.42){exc.} & H_2(-W, -\del W) \ar[ur]^{\del_*} & 
} \]
Proceeding along the top, the image of a class $\gamma \in H_2(Z)$ in $H_1(L(p,q))$ is given by $\L \gamma , x \R [b]$ according to \eqref{e: alpha class}.  Similarly, proceeding along the bottom, its image in $H_1(L(p,q))$ is given by $\L \gamma , \sigma \R [K]$ according to \eqref{e: beta class}, switching to the use of $\sigma$ for $[\Sigma]$.  Thus, taking $\gamma = e_0$, we have 
\[ k [b] = \L e_0 , x \R [b] = \L e_0 , \sigma \R [K] = [K],\]
using the fact that $\sigma_0 = 1$.  This completes the proof of the Proposition.

\end{proof}

Thus, for an {\em unoriented} knot $K \subset L(p,q)$, we obtain a pair of values $\pm k \pmod p$ that specify a pair of homology classes in $H_1(L(p,q))$, one for each orientation on $K$.  Note that had we used the reversed basis $\{x_n,\dots,x_1\}$, we would have expressed $[K]$ as a multiple $k'[a] \in H_1(L(p,q))$.  Since $[a] = q [b]$, we obtain $k' \equiv k q' \equiv k^{-1} \pmod p$, which is consistent with $q' \equiv - (k')^2 \pmod p$ and $L(p,q') \cong L(p,q)$.  Thus, given a value $k \pmod p$, we represent equivalent (unoriented) knots by choosing any of the values $\{ \pm k, \pm k^{-1} \} \pmod p$.  For the latter Berge types listed in Subsection \ref{ss: list}, we use a judicious choice of $k$.  For example, Berge types IX and X involve a concise quadratic expression for $p$ in terms of $k$, but there does not exist such a nice expression for it in terms of the least positive residue of $-k$ or $k^{-1} \pmod p$.


\section{Lattices}\label{s: lattices}


\subsection{Generalities.}\label{ss: generalities}

A {\em lattice} $L$ consists of a finitely-generated free abelian group equipped with a positive-definite, symmetric bilinear pairing $\L \; , \; \R : L \times L \to \mathbb{R}$.  It is {\em integral} if the image of its pairing lies in $\Z$.  In this case, its {\em dual lattice} is the lattice \[ L^* := \{ x \in L \otimes \mathbb{R} \; | \; \L x, y \R \in \Z \; \text{for all} \; y \in L \},\] and its {\em discriminant} $\disc(L)$ is the index $[L^*:L]$.  All lattices will be assumed integral henceforth. 

Given a vector $v \in L$, its {\em norm} is the value $| v | := \L v,v \R$.  It is {\em reducible} if $v = x+y$ for some non-zero $x,y \in L$ with $\L x,y \R \geq 0$, and {\em irreducible} otherwise.
It is {\em breakable} if $v = x + y$ for some $x,y \in L$ with $|x|,|y| \geq 3$ and $\L x,y \R = -1$, and {\em unbreakable} otherwise.  A lattice $L$ is {\em decomposable} if it is an orthogonal direct sum $L = L_1 \oplus L_2$ with $L_1, L_2 \ne (0)$, and {\em indecomposable} otherwise.

Observe that any lattice $L$ has a basis $S = \{v_1,\dots,v_n\}$ of irreducible vectors, gotten by first selecting a non-zero vector $v_1$ of minimal norm, and then inductively selecting $v_i$ as a vector of minimal norm not contained in the span of $v_1,\dots,v_{i-1}$.  Given such a basis $S$, we define its {\em pairing graph} 
\[ \widehat{G}(S) = (S,E), \quad E = \{ (v_i,v_j) \; | \; i \ne j \text{ and } \langle v_i, v_j \rangle \ne 0 \}.\footnote{More on graph notation in Subsection \ref{ss: graph lattices}.} \]
Let $G_k$ denote a connected component of $\widehat{G}(S)$ and $L_k \subset L$ the sublattice spanned by $V(G_k)$.  If $L_k = L' \oplus L''$, then each vector in $V(G_k)$ must belong to one of $L'$ or $L''$ by indecomposability.  Since $G_k$ is connected, it follows that they must all belong to the same summand, whence $L_k$ is indecomposable.  A basic result, due to Eichler, asserts that this decomposition $L \cong \bigoplus_k L_k$ into indecomposable summands is unique up to reordering of its factors \cite[Theorem II.6.4]{milnor+husemoller}.


\subsection{Graph lattices.}\label{ss: graph lattices}

Let $G = (V,E)$ denote a finite, loopless, undirected graph.  Write $v \sim w$ to denote $(v,w) \in E$.  A {\em subgraph} of $G$ takes the form $H = (V',E')$, where $V' \subset V$ and $E' \subset  \{ (v,w) \in E \; | \: v, w \in V' \}$; it is {\em induced} if ``$=$" holds in place of ``$\subset$", in which case we write $H = G|V'$.  For a pair of disjoint subsets $T,T' \subset V$, write $E(T,T')$ for the set of edges between $T$ and $T'$, $e(T,T')$ for its cardinality, and set $d(T) = e(T,V-T)$.  In particular, the {\em degree} of a vertex $v \in V$ is the value $d(v)$.

Form the abelian group $\overline{\Gamma}(G)$ freely generated by classes $[v], \; v \in V$, and define a symmetric, bilinear pairing by \[ \langle [v], [w] \rangle =  \begin{cases} d(v), & \text{ if } v=w; \\ -e(v,w), & \text{ if } v \ne w. \end{cases}\]   Let \[[T] := \sum_{v \in T} \; [v]\] and note that 
\[\langle [T], [T'] \rangle = e(T \cap T', V - (T \cup T')) - e(T - T', T' - T).\]
In particular, $\langle [T], [T] \rangle = d(T)$, and $\langle [T], [T'] \rangle = -e(T,T')$ for disjoint $T, T'$.

Given $x \in \overline{\Gamma}(G)$, write $x = \sum_{v \in V} x_v [v]$, and observe that $|x| = \sum_{e \in E} (x_v - x_w)^2$, where $v$ and $w$ denote the endpoints of the edge $e$.  It follows that $|x| \geq 0$, so the pairing on $\overline{\Gamma}(G)$ is positive semi-definite.  Let $V_1,\dots,V_k$ denote the vertex sets of the connected components of $G$.  It easy to see that $|x| = 0$ iff $x$ belongs to the span of $[V_1],\dots,[V_k]$, and moreover that these elements generate $Z(G) := \{ x \in \overline{\Gamma}(G) \; | \; \L x,y \R = 0 \text{ for all } y \in \overline{\Gamma}(G) \}$.  It follows that the quotient $\Gamma(G) := \overline{\Gamma}(G) / Z(G)$ is a lattice.

\begin{defin}\label{d: graph lattice}

The {\em graph lattice} associated to $G$ is the lattice $\Gamma(G)$.

\end{defin}

Now assume that $G$ is connected.  For a choice of {\em root} $r \in V$, every element in $\overline{\Gamma}(G)$ is equivalent $(\mod Z(G))$ to a unique element in the subspace of $\overline{\Gamma}(G)$ spanned by the set $\{ [v] \; | \; v \in V - r \}$.  In what follows, we keep a choice of root fixed, and identify $\Gamma(G)$ with this subspace.  We reserve the notation $[T]$ for $T \subset V - r$.

\begin{defin}\label{d: vertex basis}

The set $\{ [v] \; | \; v \in V - r \}$ constitutes a {\em vertex basis} for $\Gamma(G)$.

\end{defin}

\begin{prop}\label{p: graph irred}

The irreducible elements of $\Gamma(G)$ take the form $\pm [T]$, where $T$ and  $V - T$ induce connected subgraphs of $G$.

\end{prop}

\begin{proof}

Suppose that $0 \ne x = \sum_{v \in V - r} c_v [v] \in \Gamma(G)$ is irreducible.  Replacing $x$ by $-x$ if necessary, we may assume that $c := \max_v c_v \geq 1$.  Let $T = \{ v \; | \; c_v = c \}$; then
\begin{eqnarray*}
\langle [T], x - [T] \rangle &=& \langle [T], (c-1)[T] \rangle + \langle [T], \sum_{v \in V - T} c_v [v] \rangle \\ &=&  (c-1) \cdot d(T) -\sum_{v \in V - T} c_v \cdot e(v,T) \\ &=& \sum_{v \in V - T} (c-1-c_v) \cdot e(v,T) \geq 0.\end{eqnarray*}
  Since $x$ is irreducible, it follows that $x = [T]$. 

Next, we argue that $[T]$ is irreducible if and only if the induced subgraphs $G|T$ and $G|(V-T)$ are connected.  Write $y = \sum_{v \in V} y_v [v] \in \overline{\Gamma}(G)$.  Then
\begin{equation}\label{e: y}
\langle y, y-[T] \rangle = \sum_C \sum_{(u,v) \in E(C)} (y_u - y_v)^2 + \sum_{(u,v) \in E(T,V-T)} (y_u - y_v)(y_u - y_v -1),
\end{equation}
where $C$ ranges over the connected components of $G|T$ and $G|(V-T)$.  Each summand appearing in \eqref{e: y} is non-negative.  It follows that \eqref{e: y} vanishes identically if and only if (a) $y_u$ is constant on each component $C$ and (b) if a component $C_1 \subset G|T$ has an edge $(u,v)$ to a component $C_2 \subset G|(V-T)$, then $y_u = y_v$ or $y_v + 1$.  Now pass to the quotient $\Gamma(G)$.  This has the effect of setting $y_r = 0$ in \eqref{e: y}.  If $G|(V-T)$ were disconnected, then we may choose a component $C$ such that $r \notin V(C)$ and set $y_u = -1$ for all $u \in V(C)$ and $0$ otherwise.  Then $y$ and $[T] -y $ are non-zero, orthogonal, and sum to $[T]$, so $[T]$ is reducible.  Similarly, if $G|T$ were disconnected, then we could choose an arbitrary component $C$ and set $y_u = 1$ if $u \in V(C)$ and $0$ otherwise, and conclude once more that $[T]$ is reducible.  Otherwise, both $G|T$ and $G|(V-T)$ are connected, and $y$ vanishes on $G|(V-T)$ and equals $0$ or $1$ on $G|T$.  Thus, $y = 0$ or $[T]$, and it follows that $[T]$ is irreducible.

\end{proof}

\begin{prop}\label{p: graph break}

Suppose that $G$ does not contain a cut-edge, and suppose that $[T] = y + z$ with $\L y, z \R = -1$.  Then either 

\begin{enumerate}

\item $G|T$ contains a cut-edge $e$, $V(G|T \, - \, e) = T_1 \cup T_2$, and $\{y,z\} = \{[T_1], [T_2]\}$; or

\item $G(V-T)$ contains a cut-edge $e$, $V(G|(V-T) - e) = T_1 \cup T_2$, $r \in T_2$, and $\{y,z\}= \{ [T_1 \cup T], - [T_1] \}$.

\end{enumerate}

\end{prop}

\begin{proof}

Reconsider \eqref{e: y}.  In the case at hand, the inner product is $1$.  Each term $(y_u - y_v)(y_u - y_v - 1)$ is either 0 or $\geq 2$, so it must be the case that each such term vanishes and there exists a unique edge $e \in E(T) \cup E(V-T)$, $e=(u,v)$, for which $(y_u - y_v)^2 =1$ and all other terms vanish.  In particular, it follows that $e$ is a cut-edge in either (a) $G|T$ or (b) $G|(V-T)$.

In case (a), write $T_1$ and $T_2$ for the vertex sets of the components of $G|T \, - \, e$.  Then $y$ is constant on $T_1$, $T_2$, and $V-T$; furthermore, it vanishes on $V-T$ and its values on $T_1$ and $T_2$ differ by one.  Since $e$ is not a cut-edge in $G$, it follows that $E(V-T,T_1), E(V-T,T_2) \ne \varnothing$, so the values on $T_1$ and $T_2$ differ from the value on $V-T$ by at most one.  It follows that these values are $0$ and $1$ in some order.  This results in (1).

In case (b), write $T_1$ and $T_2$ for the vertex sets of the components of $G|(V-T) - e$ with $r \in T_2$.  Now $y$ is constant on $T_1, T_2$, and $T$; furthermore, it vanishes on $T_2$ and its values on $T_1$ and $T_2$ differ by one.  Hence the value on $T_2$ is $1$ or $-1$.  Since $e$ is not a cut-edge in $G$, it follows that $E(T,T_1), E(T,T_2) \ne \varnothing$, so the value on $T$ is $0$ or $1$ more than the values on $T_1$ and $T_2$.  Thus, either the value on $T_1 \cup T$ is 1, or the value on $T_1$ is $-1$ and the value on $T$ is 0.  This results in (2).

\end{proof}


\subsection{Linear lattices.}\label{ss: linear lattice}

Observe that a sum of linear lattices $L = \bigoplus_k L_k$ occurs as a special case of a graph lattice.  Indeed, construct a graph $G$ whose vertex set consists of one vertex for each generator $x_i$ of $L_k$ (Definition \ref{d: linear lattice}), as well as one additional vertex $r$.  For a pair of generators $x_i, x_j$, declare  $(x_i,x_j) \in E$ if and only if $\L x_i, x_j \R = -1$, and define as many parallel edges between $r$ and $x_i$ as necessary so that $d(x_i) = a_i$.  It is clear that $\Gamma(G) \cong L$, and this justifies the term linear {\em lattice}.  Furthermore, the $x_i$ comprise a vertex basis for $L$.

Given a linear lattice $L$ and a subset of consecutive integers $\{i,\dots,j\} \subset \{1,\dots,n\}$, we obtain an {\em interval} $\{x_i,\dots,x_j\}$.  Two distinct intervals $T = \{x_i,\dots,x_j\}$ and $T' = \{x_k,\dots,x_l\}$ {\em share a common endpoint} if $i = k$ or $j = l$ and are {\em distant} if $k > j+1$ or $i > l+1$.  If $T$ and $T'$ share a common endpoint and $T \subset T'$, then write $T \prec T'$.  If $i = k+1$ or $l = j+1$, then $T$ and $T'$ are {\em consecutive} and write $T \dagger T'$.  They {\em abut} if they are either consecutive or share a common endpoint.  Write $T \pitchfork T'$ if $T \cap T' \ne \varnothing$ and $T$ and $T'$ do not share a common endpoint.  Observe that if $T \pitchfork T'$, then the symmetric difference $(T - T') \cup (T' - T)$ is the union of a pair of distant intervals.

\begin{cor}\label{c: linear irred}

Let $L = \bigoplus_k L_k$ denote a sum of linear lattices.

\begin{enumerate}

\item The irreducible vectors in $L$ take the form $\pm[T]$, where $T$ is an interval in some $L_k$;

\item each $L_k$ is indecomposable;

\item if $T \pitchfork T'$, then $[T-T'] \pm [T'-T]$ is reducible;

\item $[T]$ is unbreakable iff $T$ contains at most one vertex of degree $\geq 3$.

\end{enumerate}

\end{cor}

\begin{proof}

(1) follows from Proposition \ref{p: graph irred}, noting that for $T \subset V - \{r\}$, if $T$ and $V-T$ induce connected subgraphs of the graph $G$ corresponding to $L$, then $T$ is an interval in some $L_k$.  (2) follows since the elements of the vertex basis for $L_k$ are irreducible and their pairing graph is connected.  For item (3), write $(T-T') \cup (T'-T) = T_1 \cup T_2$ as a union of distant intervals.  Then $[T-T']\pm[T'-T] = \epsilon_1 [T_1] + \epsilon_2 [T_2]$ for suitable signs $\epsilon_1,\epsilon_2 \in \{ \pm 1 \}$, and $\L \epsilon_1 [T_1], \epsilon_2 [T_2] \R = 0$.  For (4), we establish the contrapositive in two steps.

\noindent ($\implies$) If an interval $T$ contains a pair of vertices $x_i,x_j$ of degree $\geq 3$, then it breaks into consecutive intervals $T = T_i \cup T_j$ with $x_i \in T_i$ and $x_j \in T_j$. It follows that $[T]$ is breakable, since $[T] = [T_i]+[T_j]$ with $\L [T_i],[T_j] \R = -1$ and $d(T_i), d(T_j) \geq 3$.

\noindent ($\impliedby$) If $[T] = y+z$ is breakable, then Proposition \ref{p: graph break} applies.  Observe that case (2) does not hold since breakability entails $|z| \geq 3$, while a cut-edge in $G(V-T)$ separates it into $T_1 \cup T_2$ with $d(T_1) = 2$ and $r \in T_2$.  Thus, case (1) holds, and it follows that $d(T_1),d(T_2) \geq 3$, so both $T_1$ and $T_2$ contain a vertex of degree $\geq 3$, which shows that $T$ contains at least two such.

\end{proof}

Next we turn to the question of when two linear lattices are isomorphic.  Let $I$ denote the set of irreducible elements in $L$, and given $y \in I$, let \[I(y) = \{ z \in I \: | \; \L y,z \R = -1, y + z \in I \}.\]  To unpack the meaning of this definition, suppose that $y \in I$, $|y| \geq 3$, and write $y = \epsilon_y [T_y]$.  If $z \in I(y)$ with $z = \epsilon_z [T_z]$, then either $\epsilon_y = \epsilon_z$ and $T_y \dagger T_z$, or else $|z|=2$, $\epsilon_y = - \epsilon_z$, and $T_z \prec T_y$.  Now suppose that $z \in I(y)$ with $|z| = 3$.  Choose elements $x_i \in T_y$ and $x_j \in T_z$ of norm $\geq 3$ so that the open interval $(x_i,x_j)$ contains no vertex of degree $\geq 3$.  If $w \in I(y) \cap (-I(z))$ with $w = \epsilon_w [T_w]$, then $T_w \subset (x_i,x_j)$, and either $T_w \prec T_y$ and $\epsilon_w = - \epsilon_y$, or else $T_w \prec T_z$ and $\epsilon_w = \epsilon_z$.  It follows that $|I(y) \cap (-I(z))| = |(x_i,x_j)| = |i-j|-1$.

The following is the main result of \cite{gerstein}.

\begin{prop}[Gerstein]\label{p: gerstein}

If $\Lambda(p,q) \cong \Lambda(p',q')$, then $p = p'$, and $q = q'$ or $q q' \equiv 1 (\textup{mod } p)$.

\end{prop}

\begin{proof}

Let $L$ denote a linear lattice with standard basis $S = \{ x_1,\dots,x_n \}$.  The Proposition follows once we show that $L$ uniquely determines the sequence of norms ${\bf x} = (|x_1|,\dots,|x_n|)$ up to reversal, noting that if $p/q = [a_1,\dots,a_n]^-$ and $p/q' = [a_n,\dots,a_1]^-$, then $q q' \equiv 1 \pmod p$ (Lemma \ref{l: cont frac basics}(4)).

Suppose that $I$ contains an element of norm $\geq 3$.  In this case, select $y_1 \in I$ with minimal norm $\geq 3$ subject to the condition that there does not exist a pair of orthogonal elements in $I(y_1)$.  It follows that $y_1 = \epsilon [T_1]$, where $T_1$ contains exactly one element $x_{j_1} \in S$ of norm $\geq 3$, and $j_1$ is the smallest or largest index of an element in $S$ with norm $\geq 3$.  Inductively select $y_i \in I(y_{i-1})$ with minimal norm $\geq 3$ subject to the condition that $\L y_i, y_j \R = 0$ for all $j < i$, until it is no longer possible to do so, terminating in some element $y_k$.  It follows that $y_i = \epsilon [T_i]$ for all $i$, where $\epsilon \in \{ \pm 1 \}$ is independent of $i$; each $T_i$ contains a unique $x_{j_i} \in S$ of norm $\geq 3$; $T_i \dagger T_{i+1}$ for $i < k$; and each $x_j \in S$ of norm $\geq 3$ occurs as some $x_{j_i}$.  Therefore, up to reversal, the (possibly empty) sequence $(|y_1|,\dots,|y_k|) = (|x_{j_1}|,\dots,|x_{j_k}|)$ coincides with ${\bf x}$ with every occurrence of 2 omitted.

To recover ${\bf x}$ completely, assume for notational convenience that $j_1 < \cdots < j_k$.  Set $n_i: = | I(y_i) \cap (-I(y_{i+1})) |$ for $i=1,\dots,k-1$, so that $n_i = j_{i+1}-j_i -1$.  If $k \geq 2$, then set $n_0 = | I(y_1) - (-I(y_2))$, $n_k = |I(y_k) - (-I(y_{k-1}))|$, and observe that $n_0 = j_1-1$ and $n_k = n-j_k$.  If $k = 1$, then decompose $I(y) = I_0 \cup I_1$, where $\L z_i, z'_i \R \ne 0$ for all $z_i,z'_i \in I_i$, $i = 0,1$.  In this case, set $n_i = |I_i|$, and observe that $\{n_0,n_1\} = \{ j_1 - 1 , n-j_1 \}$.  Lastly, if $k = 0$, then set $n_0 = n$.  Letting $2^{[t]}$ denote the sequence of 2's of length $t$, it follows that ${\bf x} = (2^{[n_0]},|y_1|,2^{[n_1]},\dots,2^{[n_{k-1}]},|y_k|,2^{[n_k]})$.  

Since the elements $y_1,\dots,y_k$ and the values $n_0,\dots,n_k$ depend solely on $L$ for their definition, it follows that ${\bf x}$ is determined uniquely up to reversal, and the Proposition follows.

\end{proof}

The following Definition and Lemma anticipate our discussion of the intersection graph in Subsection \ref{ss: int graph} (esp. Lemma \ref{l: cycle}).

\begin{defin}\label{d: abut graph}

Given a collection of intervals $\T =\{ T_1,\dots,T_k \}$ whose classes are linearly independent, define a graph \[ G(\T) = (\T, \E), \quad \E = \{ (T_i,T_j) \; | \; T_i \text{ abuts } T_j \}.\]

\end{defin}

\begin{lem}\label{l: triangle}

Given a cycle $C \subset G(\T)$, the intervals in $V(C)$ abut pairwise at a common end.  That is, there exists an index $j$ such that each $T_i \in V(C)$ has left endpoint $x_{j+1}$ or right endpoint $x_j$.  In particular, $V(C)$ induces a complete subgraph of $G(\T)$.

\end{lem}

\begin{proof}

Relabeling as necessary, write $V(C) = \{T_1,\dots,T_k\}$, where $(T_i,T_{i+1}) \in E(C)$ for $i = 1,\dots, k$, subscripts $(\mod k)$.  We proceed by induction on the number of edges $n \geq k$ in the subgraph induced on $V(C)$.

When $n = k$, $C$ is an induced cycle.  In this case, if some of three of the intervals abut at a common end, then they span a cycle, $k=3$, and we are done.  If not, then define a sign $\epsilon_i = \pm 1$ by the rule that $\epsilon_i = 1$ if and only if $T_i  \dagger T_{i-1}$ and $T_i$ lies to the right of $T_{i-1}$, or if $T_i$ and $T_{i-1}$ share a common left endpoint.  Fix a vertex $x_j$, suppose that $x_j \in T_i$ for some $i$, and choose the next index $l \pmod{k}$ for which $x_j \in T_l$.  Observe, crucially, that $\epsilon_i = - \epsilon_l$.  It follows that $\L x_j, \sum_{i=1}^k \epsilon_i [T_i] \R = 0$.  As $x_j$ was arbitrary, we obtain the linear dependence $\sum_{i=1}^k \epsilon_i [T_i] = 0$, a contradiction.  It follows that if $n = k$, then $k=3$ and the three intervals abut at a common end.  

Now suppose that $n > k$.  Thus, there exists an edge $(T_i,T_j) \in E(C)$ for some pair of non-consecutive indices $i, j \pmod{k}$. Split $C$ into two cycles $C_1$ and $C_2$ along $(T_i,T_j)$.  By induction, every interval in $V(C_1)$ and $V(C_2)$ abuts at the same end as $T_i$ and $T_j$, so the same follows at once for $V(C)$.

\end{proof}


\subsection{Changemaker lattices.}\label{ss: changemaker}

Fix an orthonormal basis $\{e_0,\dots,e_n\}$ for $\Z^{n+1}$.

\begin{defin}\label{d: changemaker lattice}

A {\em changemaker lattice} is any lattice isomorphic to $(\sigma)^\perp \subset \Z^{n+1}$ for some changemaker $\sigma$ (Definition \ref{d: change}).

\end{defin}

\begin{lem}\label{l: change disc}

Suppose that $L = (\sigma)^\perp \subset \Z^{n+1}$ is a changemaker lattice.  Then $\disc(L) = |\sigma|$.

\end{lem}

\begin{proof}

(cf. \cite[proof of Lemma II.1.6]{milnor+husemoller})  Consider the map
\[ \varphi: \Z^{n+1} \to \Z / | \sigma | \Z, \quad \varphi(x) = \L x, \sigma \R \pmod{| \sigma |}.\]
As $\sigma_0 = 1$, the map $\varphi$ is onto, so $K : = \ker(\varphi)$ has discriminant $[\Z^{n+1}:K]^2 = | \sigma |^2$.  On the other hand, $K = L \oplus (\sigma)$, so $\disc(L) = \disc(K) / | \sigma | = | \sigma |$.

\end{proof}

\begin{proof}[Proof of Theorem \ref{t: main technical}]

We invoke \cite[Theorem 3.3]{greene:cabling} with $X = X(p,q)$.  It follows that $\Lambda(p,q)$ embeds as a full-rank sublattice of $(\sigma)^\perp \subset \Z^{n+1}$, where $|\sigma| = p$ and $\sigma$ is  a changemaker according to the convention of \cite{greene:cabling}.  It stands to verify that $\sigma_0 = 1$, and furthermore that $\Lambda(p,q)$ actually equals $(\sigma)^\perp$ on the nose.  First, if $\sigma_0 = 0$, then $\Lambda(p,q)$ would have a direct summand isomorphic to $(e_0) \cong \Z$, in contradiction to its indecomposability.  Hence $\sigma_0 = 1$. Second, $\disc(\Lambda(p,q)) = p = |\sigma| = \disc((\sigma)^\perp)$, using Lemma \ref{l: change disc} at the last step.  Since $\rk \; \Lambda(p,q) = \rk \; (\sigma)^\perp$, the two lattices coincide.

\end{proof}

We construct a basis for a changemaker lattice $L$ as follows.  Fix an index $1 \leq j \leq n$, and suppose that $\sigma_j = 1 + \sum_{i=0}^{j-1} \sigma_i$.  In this case, set $v_j = -e_j + 2e_0 + \sum_{i=1}^{j-1} e_i \in L$.  Otherwise, $\sigma_j \leq \sum_{i=0}^{j-1} \sigma_i$.  It follows that there exists a subset $A \subset \{ 0,\dots,j-1\}$ such that $\sigma_j = \sum_{i \in A} \sigma_i$.  Amongst all such subsets, choose the one maximal with respect to the total order $<$ on subsets of $\{ 0, 1, \dots, n \}$ defined by declaring $A' < A$ if the largest element in $(A \cup A') \setminus (A \cap A')$ lies in $A$; equivalently, $\sum_{i \in A'} 2^i < \sum_{i \in A} 2^i$.  Then set $v_j = -e_j+\sum_{i \in A} e_i \in L$.  If $v = -e_j + \sum_{i \in A'} e_i$ for some $A' < A$, then write $v \ll v_j$.    

The vectors $v_1,\dots,v_n$ are clearly linearly independent.  The fact that they span $L$ is straightforward to verify, too: given $w \in L$, add suitable multiples of $v_n,\dots,v_1$ to $w$ in turn to produce a sequence of vectors with {\em support} decreasing to $\varnothing$.  Recall that the support of a vector $v \in \Z^{n+1}$ is the set $\supp(v) = \{ i \; | \; \L v, e_i \R \ne 0 \}$.  For future reference, we also define
\[ \supp^+(v) = \{ i \; | \; \L v, e_i \R > 0 \} \quad \text{and} \quad \supp^-(v) = \{ i \; | \; \L v, e_i \R < 0 \}. \]

\begin{defin}\label{d: tight, gappy, just right}

The set $S = \{v_1,\dots,v_n\}$ constitutes the {\em standard basis} for $L$.  A vector $v_j \in S$ is 

\begin{itemize}

\item {\em tight}, if $v_j = -e_j + 2e_0 + \sum_{i=1}^{j-1} e_i$;

\smallskip

\item {\em gappy}, if $v_j = -e_j+\sum_{i \in A} e_i$ and $A$ does not consist of consecutive integers; and

\smallskip

\item {\em just right}, if $v_j = -e_j+\sum_{i \in A} e_i$ and $A$ consists of consecutive integers.

\end{itemize}
A {\em gappy index} for a gappy vector $v_j$ is an index $k \in A$ such that  $k+1 \notin A \cup \{j\}$.
\end{defin}

\noindent  Thus, every element of $S$ belongs to exactly one of these three types.

We record a few basic observations before proceeding to some more substantial facts about changemaker lattices.  Write $v_{jk} = \L v_j, e_k \R$.

\begin{lem}\label{l: basic}

The following hold:

\begin{enumerate}

\item $v_{jj} = -1$ for all $j$, and $v_{j,j-1}=1$ unless $j=1$ and $v_{1,0} = 2$;

\item for any pair $v_i, v_j$, we have $\L v_i, v_j \R \geq -1$;

\item if $k$ is a gappy index for some $v_j$, then $|v_{k+1}| \geq 3$;

\item given $z = \sum_{i=0}^n z_i e_i \in L$ with $|z| \geq 3$, $\supp^-(z)=\{j\}$, and $z_j = -1$, it follows that $j = \max(\supp(z))$.

\end{enumerate}

\end{lem}

\begin{proof}
(1) is clear, using the maximality of $A$ for the second part. (2) is also clear.  (3) follows from maximality, as otherwise $v_j \ll v_j - v_{k+1}$. For (4), suppose not, and select $k > j$ for which $z_k > 0$.  We obtain the contradiction $0 = \L z, \sigma \R > \sigma_k - \sigma_j \geq 0$, where the inequality is strict because $|z| \geq 3$.

\end{proof}

\begin{lem}\label{l: change irred}

The standard basis elements of a changemaker lattice are irreducible.

\end{lem}

\begin{proof}

Choose a standard basis element $v_j \in S$ and suppose that $v_j = x +y$ for $x,y \in L$ with $\langle x,y \rangle \geq 0$.  In order to prove that $v_j$ is irreducible, it stands to show that one of $x$ and $y$ equals $0$.  Write $x = \sum_{i=0}^n x_i e_i$ and $y = \sum_{i=0}^n y_i e_i$.

{\em Case 1.} $v_j$ is not tight.  In this case, $|v_{ji}| \leq 1$ for all $i$.  We claim that $x_i y_i = 0$ for all $i$.  For suppose not. Since $\langle x, y \rangle \geq 0$ there exists an index $i$ so that $x_i y_i > 0$.  Then $|v_{ji}| = |x_i + y_i| \geq 2$, a contradiction.  Since all but one coordinate of $v_j$ is non-negative, it follows that one of $x$ and $y$ has all its coordinates non-negative.  But the only such element in $L$ is $0$.  It follows that $v_j$ is irreducible.  Notice that this same argument applies to any vector of the form $-e_j + \sum_{i \in A} e_i$.

{\em Case 2.} $v_j$ is tight.  We repeat the previous argument up to the point of locating an index $i$ such that $x_i y_i > 0$.  Now, however, we conclude that $x_0 = y_0 =1$, and $x_i y_i \leq 0$ for all other indices $i$.  In particular, there is at most one index $k$ for which $x_k y_k = -1$, and $x_i y_i =0$ for $i \ne 0,k$.  If there were no such $k$, then we conclude as above that one of $x$ and $y$ has all its coordinates non-negative, and $v_j$ is irreducible as before.  Otherwise, we may assume that $x_k = 1$.  Then $\supp^-(y) = \{k\}$ and $\supp^-(x) = \{j\}$. Since $x_i + y_i = v_{ji} \ne 0$ for all $i \leq j$, it follows that $k > j$.  But then $|x| \geq 3$ and $k = \max(\supp(x)) > j$, in contradiction to Lemma \ref{l: basic}(4).  Again it follows that $v_j$ is irreducible.

\end{proof}

We collect a few more useful cases of irreducibility (cf. Lemma \ref{l: tight}).

\begin{lem}\label{l: more irred}

Suppose that $v_t \in L$ is tight.

\begin{enumerate}

\item If $v_j$ is tight, $j \ne t$, then $v_j - v_t$ is irreducible.

\item If $v_j = -e_j+e_{j-1}+e_t$, $j > t$, then $v_j + v_t$ is irreducible.

\item If $v_{t+1} = -e_{t+1}+e_t+\cdots+e_0$, then $v_{t+1} - v_t$ is irreducible. 

\end{enumerate}

\end{lem}

\begin{proof}

In each case, we assume that the vector in question is expressed as a sum of non-zero vectors $x$ and $y$ with $\L x,y \R \geq 0$.  Recall that both $x$ and $y$ have entries of both signs.

(1) Assume without loss of generality that $j > t$.  Thus, $v_j - v_t = -e_j + 2e_t + \sum_{i=t+1}^{j-1} e_i$, where the summation could be empty.  As in the proof of Lemma \ref{l: change irred},  it quickly follows that $x_t = y_t = 1$, $x_k y_k = -1$ for some $k$, and otherwise $x_i y_i = 0$.  Without loss of generality, $x_k = 1$.  Thus, $\supp^-(x) = \{j\}$ and $\supp^-(y) = \{ k \}$.  By Lemma \ref{l: basic}(4), it follows that $j = \max(\supp(x)) > k$.  As $x_k+y_k = 0$, it follows that $k < t$.  Now $0 = \L y, \sigma \R \geq \sigma_t - \sigma_k > 0$, a contradiction. Therefore, $v_j - v_t$ is irreducible.

(2) It follows that $x_0 = y_0 = 1$, $x_k = -y_k = \pm 1$ for some value $k \geq t$, and otherwise $x_i y_i = 0$.  Without loss of generality, say $x_k = 1$.  Thus, $\supp^-(x) = \{j\}$ and $\supp^-(y) = \{ k \}$.  By Lemma \ref{l: basic}(4), $j = \max (\supp(x))$.  In particular, it follows that $k < j$.  Another application of Lemma \ref{l: basic}(4) implies that $k = \max(\supp(y))$.  It follows that $y = -e_k + \sum_{i \in A} e_i$ for some $A \subset \{0,\dots,t-1\}$. But then $0 = \L y, \sigma \R = -\sigma_k + \sum_{i \in A} \sigma_i < -\sigma_k + 1 + \sum_{i=0}^{t-1} \sigma_i = -\sigma_k + \sigma_t \leq 0$, a contradiction.  Therefore, $v_j + v_t$ is irreducible.

(3) We have $v_{t+1} - v_t = -e_{t+1} + 2e_t - e_0$.    It follows that $x_t = y_t = 1$.  If $x_i y_i = 0$ for every other index $i$, then $\{ x,y \} = \{ -e_{t+1} + e_t, e_t - e_0 \}$, but $\L e_t - e_0, \sigma \R = \sigma_t - \sigma_0 > 0$, a contradiction.  It follows that there is a unique index $k$ for which $x_k y_k = -1$, and otherwise $x_i y_i = 0$.  Without loss of generality, say $x_k = 1$.  Hence $y = e_t - \sum_{i \in A} e_i$ for some $A \subset \{0,k,t+1\}$. But $A$ cannot contain an index $> t$, for then $0 = \L y, \sigma \R \leq \sigma_t - \sigma_{t+1} < 0$, nor can it just contain indices $< t$, for then $0 = \L y, \sigma \R \geq \sigma_t - \sum_{i=0}^{t-1} \sigma_i = 1$.  Therefore, $v_{t+1} - v_t$ is irreducible.

\end{proof}

\begin{lem}\label{l: change unbreakable}

If $v_j \in S$ is not tight, then it is unbreakable.

\end{lem}

\begin{proof}

Suppose that $v_j$ were breakable and choose $x$ and $y$ accordingly.  From the conditions that  $\langle x, y \rangle = -1$ and $v_{j0 } \ne 2$ it follows that $x_k y_k = -1$ for a single index $k$, and otherwise $x_i y_i = 0$.  Without loss of generality, say $x_k = -1$.   Then $\supp^-(x) = \{ k \}$ and $\supp^-(y) = \{ j \}$.  By Lemma \ref{l: basic}(4), it follows that $k = \max (\supp(x))$ and $j = \max( \supp(y))$.  In particular, it follows that $k < j$, and that $y_i = v_{ji}$ for all $i > j$.  On the other hand, $j \in \supp^+(y) - \supp^+(v_i)$.  It follows  that $v_j \ll y$, a contradiction.  Hence $v_i$ is unbreakable, as claimed.

\end{proof}


\section{Comparing linear lattices and changemaker lattices}\label{s: prep work}

In this section we collect some preparatory results concerning when a changemaker lattice is isomorphic to a sum of one or more linear lattices.  Thus, for the entirety of this section, let $L$ denote a changemaker lattice with standard basis $S = \{ v_1, \dots, v_n \}$, and suppose that $L$ is isomorphic to a linear lattice or a sum thereof.  By Corollary \ref{c: linear irred} and Lemma \ref{l: change irred}, it follows that $v_i = \epsilon_i [T_i]$ for some sign $\epsilon_i = \pm 1$ and interval $T_i$. Let $\T = \{T_1,\dots,T_n\}$. If $v_i$ is not tight, then Corollary \ref{c: linear irred} and Lemma \ref{l: change unbreakable} imply that $T_i$ contains at most one vertex of degree $\geq 3$.  If $[T_i]$ is unbreakable and $d(T_i) \geq 3$, then let $z_i$ denote its unique vertex of degree $\geq 3$.


\subsection{Standard basis elements and intervals.}  

Tight vectors, especially breakable ones, play an involved role in the analysis (Section \ref{s: tight}).  We begin with some basic observations about them.

\begin{lem}\label{l: tight 2}

Suppose that $v_t$ is tight, $j \ne t$, and $|v_j| \geq 3$.  Then $\L v_t, v_j \R$ equals

\begin{enumerate}

\item $|v_j| - 1$, iff $T_j \prec T_t$;

\item $|v_j| -2$, iff $z_j \in T_t$ and $T_j \pitchfork T_t$, or $|v_j| = 3, T_j \dagger T_t$, and $\epsilon_j \ne \epsilon_t$;

\item $\epsilon \in \{\pm 1\}$, iff $T_j \dagger T_t$ and $\epsilon_j \epsilon_t \ne\epsilon$, or $|v_j| = 3$, $z_j \in T_t$, $T_j \pitchfork T_t$, and $\epsilon_j \epsilon_t = \epsilon$; or

\item $0$, iff $z_j \notin T_t$ and either $T_j$ and $T_t$ are distant or $T_j \pitchfork T_t$.

\end{enumerate}
If $|v_j| = 2$, then $|\L v_t, v_j \R| \leq 1$, with equality iff $T_t$ and $T_j$ abut.

\end{lem}

\begin{proof}[Proof sketch.]

Observe that $-1 \leq \L v_i, v_j \R \leq |v_j|-1$ for any pair of distinct $i,j$.  Assuming that $|v_j| \geq 3$, the result follows by using the fact that $T_j$ is unbreakable, and conditioning on how $T_j$ meets $T_t$ and whether or not $d(T_j) > 3$.

\end{proof}

\begin{lem}\label{l: tight}

Suppose that $v_t \in S$ is tight.

\begin{enumerate}

\item No other standard basis vector is tight.

\item If $v_j = -e_j+e_{j-1}+e_t$, $j > t+1$, then $T_t \dagger T_j$.

\item If $v_{t+1} = -e_{t+1}+e_t+\cdots+e_0$, then $t=1$ and $T_1 \dagger T_2$. 

\end{enumerate}

\end{lem}

\begin{proof}

We apply each case of Lemma \ref{l: more irred} in turn.

(1) Suppose that there were another index $j$ for which $v_j$ is tight.  Without loss of generality, we may assume that $j> t$.  Then $\langle v_t, v_j \rangle = | v_t | - 2 \geq 3.$  It follows that $\epsilon_j = \epsilon_t$ and $T_j \pitchfork T_t$.  Thus, $[T_j - T_t] - [T_t - T_j]$ is reducible, but it also equals $\epsilon_j(v_j - v_t)$, which is irreducible according to Lemma \ref{l: more irred}(1).  This yields the desired contradiction.

(2) We have $\L v_t, v_j \R = -1$ and $|v_j | = 3$, so either the desired conclusion holds, or else $z_j \in T_t$, $T_t \pitchfork T_j$, and $\epsilon_t \epsilon_j = -1$.  If the latter possibility held, then $[T_j - T_t] - [T_t - T_j]$ is reducible, but it also equals $\epsilon_j v_j - \epsilon_t v_t = \epsilon_j (v_j + v_t)$, which is irreducible according to Lemma \ref{l: more irred}(2).  It follows that $T_t \dagger T_j$.

(3) We have $\L v_t, v_{t+1} \R = |v_{t+1}|-2$, so either the desired conclusion holds, or else $z_{t+1} \in T_t$, $T_t \pitchfork T_{t+1}$, and $\epsilon_t = \epsilon_{t+1}$. If the latter possibility held, then again $[T_{t+1} - T_t]-[T_t-T_{t+1}]$ is reducible, but it also equals $\epsilon_t (v_{t+1}-v_t)$, which is irreducible according to Lemma \ref{l: more irred}(3).  It follows that $t = 1$ and $T_1 \dagger T_2$.

\end{proof}

Lemmas \ref{l: change unbreakable} and \ref{l: tight}(1) immediately imply the following result.

\begin{cor}\label{c: one breakable}

A standard basis $S$ contains at most one unbreakable vector, and it is tight. \qed

\end{cor}

The following important Lemma provides essential information about when two standard basis elements can pair non-trivially together: unless one is breakable or has norm 2, then they correspond to consecutive intervals.

\begin{lem}\label{l: cool}

Given a pair of unbreakable vectors $v_i, v_j \in S$ with $| v_i |, | v_j | \geq 3$, we have $| \L v_i, v_j \R | \leq 1$, with equality if and only if $T_i \dagger T_j$ and $\epsilon_i \epsilon_j = - \langle v_i, v_j \rangle$.

\end{lem}

\begin{proof}

The Lemma follows easily once we establish that $\L [T_i], [T_j] \R \leq 0$.  Thus, we assume that $\L [T_i], [T_j] \R \geq 1$ and derive a contradiction.  Since these classes are unbreakable and they pair positively, it follows that $z_i = z_j$.  In particular, $d:= |v_i| = d(T_i) = d(T_j) =|v_j| $.   Now, either $T_i \pitchfork T_j$, in which case $\L [T_i], [T_j] \R = d-2$, or else $T_i$ and $T_j$ share a common endpoint, in which case $\L [T_i], [T_j] \R = d-1$.

Let us first treat the case in which $i = t$ and $v_t$ is tight.  By Corollary \ref{c: one breakable}, it follows that $v_j$ is not tight.  Thus, $d = t+4$, and $\supp(v_j)$ contains at least three values $> t$.  If $v_{jt}=1$, then $\L v_t, v_j \R \leq d-3$, while if $v_{jt} = 0$, then $\supp(v_j)$ contains at least four values $>t$, and again $\L v_j, v_t \R \leq d-3$.  As $\L v_t, v_j \R \geq -1$ and $d \geq 5$, we have $| \L v_t, v_j \R | \leq d-3$, whereas $|\L [T_t], [T_j] \R| \geq d-2$.  This yields the desired contradiction to the assumption that $\L [T_t], [T_j] \R \geq 1$ in this case.

Thus, we may assume that neither $v_i$ nor $v_j$ is tight, and without loss of generality that $j > i$.  Suppose that $\epsilon: = \epsilon_j = \epsilon_i$.  Thus $\langle v_i, v_j \rangle = \langle [T_i], [T_j] \rangle \geq 1$.  If $v_{ji} = 1$, then $\langle v_j, v_i \rangle \leq d -2$, with equality possible iff $v_{jk} = 1$ whenever $v_{ik} = 1$.  But then $| v_j | > | v_i |$, a contradiction.  Hence $v_{ji} = 0$.  If $\langle v_j, v_i \rangle = d-1$, then again $v_{jk} = 1$ whenever $v_{ik} = 1$.  But then $v_j - v_i = -e_j + e_i \gg v_j$, a contradiction.

Still assuming that $\epsilon_j = \epsilon_i$, we are left to consider the case that $v_{ji} = 0$ and $\langle v_j, v_i \rangle = d - 2$, and therefore $T_i \pitchfork T_j$.  In this case, 
$\supp(v_j) - \supp(v_i) = \{j,k\}$ and $\supp(v_i) - \supp(v_j) = \{i,l\}$ for some indices $k, l$.  If $\sigma_j = \sigma_k$, then $v_j = -e_j +e_k$, in contradiction to $|v_j| \geq 3$.  If $\sigma_j = \sigma_i$, then either $i < k$, in which case we derive the contradiction $v_j = -e_j + e_k$ again, or else $k < i$, in which case we derive the contradiction $-e_j + e_i \gg v_j$.  Therefore, $\sigma_j \ne \sigma_i, \sigma_k$.  It easily follows that $v_j - v_i = -e_j + e_k + e_i - e_l$ is irreducible.   On the other hand, $v_j - v_i = \epsilon ( [T_j] - [T_i] ) = \epsilon[T_j - T_i] - \epsilon[T_i - T_j]$ is reducible, a contradiction.

It follows that $\epsilon: = \epsilon_j = - \epsilon_i$.  Hence $\langle [T_i], [T_j] \rangle = - \langle v_i, v_j \rangle \leq 1$.  In case of equality, we have $d = 3$ and $T_i \pitchfork T_j$. So on the one hand, $v_j = -e_j + e_i + e_p$ and $v_i = -e_i + e_q + e_s$ for distinct indices $i,j,p,q,s$, whence $v_j + v_i = -e_j + e_p + e_q + e_s$ is irreducible (cf. the proof of Lemma \ref{l: change irred}, Case 1).  On the other hand, it equals $\epsilon([T_j]-[T_i]) = \epsilon[T_j - T_i] - \epsilon[T_i - T_j]$, which is reducible, a contradiction.

In total, $\L [T_i], [T_j] \R \leq 0$ in every case, and the Lemma follows.

\end{proof}

\begin{cor}\label{c: distinct z_i}

If $T_i$ and $T_j$ are distinct unbreakable intervals with $d(T_i),d(T_j) \geq 3$, then $z_i \ne z_j$.

\qed

\end{cor}


\subsection{The intersection graph.}\label{ss: int graph}

This subsection defines the key notion of the intersection graph, and establishes the most important properties about it that are necessary to carry out the combinatorial analysis of Sections \ref{s: decomposable} - \ref{s: tight}.

Recall that from the standard basis $S$ we obtain a collection of intervals $\T$.  Let $\overline{S} \subset S$ denote the subset of unbreakable elements of $S$; thus, $\overline{S} =  S - v_t$ if $S$ contains a breakable element $v_t$,  and $\overline{S} = S$ otherwise.

\begin{defin}[Compare Definition \ref{d: abut graph}]\label{d: int graph}

The {\em intersection graph} is the graph
\[ G(S) = (S,E), \quad E = \{ (v_i,v_j) \; | \; T_i \text{ abuts } T_j \}.\]
Write $G(S')$ to denote the subgraph induced by a subset $S' \subset S$.  If $(v_i, v_j) \in E$ with $i < j$, then $v_i$ is a {\em smaller neighbor} of $v_j$.

\end{defin}

Observe that $G(S)$ is a subgraph of the pairing graph $\widehat{G}(S)$ (Subsection \ref{ss: generalities}), and Lemma \ref{l: cool} implies that they coincide unless $S$ contains a breakable element $v_t$.  Furthermore, if $v_t \in S$ is breakable, then Lemma \ref{l: tight 2} implies that $(v_t,v_j) \in E$ iff $\L v_t, v_j \R \in \{|v_j|-1,1,-1\}$, except in the special case that $|v_j| = 3$, $z_j \in T_t$, $T_j \pitchfork T_t$, and $\epsilon_j \epsilon_t = \L v_t, v_j \R$.  Therefore, $G(S)$ is determined by the pairings of vectors in $S$ except in this special case, which fortunately arises just once in our analysis (Proposition \ref{p: breakable 3}).

We now collect several fundamental properties about the intersection graph $G(S)$.

\begin{defin}\label{d: claw}

The {\em claw} $(i;j,k,l)$ is the graph $Y = (V,E)$ with \[ V = \{i,j,k,l\} \quad \text{and} \quad E = \{ (i,j), (i,k), (i,l) \}.\]  A graph $G$ is {\em claw-free} if it does not contain an induced subgraph isomorphic to $Y$.  

\end{defin}

\noindent Equivalently, if three vertices in $G$ neighbor a fourth, then some two of them neighbor.

\begin{lem}\label{l: no claw}

$G(S)$ is claw-free.

\end{lem}

\begin{proof}

If $T_i$ abuts three intervals $T_j, T_k, T_l$, then it abuts some two at the same end, and then those two abut.

\end{proof}

\begin{defin}\label{d: heavy}

A {\em heavy triple} $(v_i,v_j,v_k)$ consists of distinct vectors of norm $\geq 3$ contained in the same component of $G(\overline{S})$, none of which separates the other two in $G(\overline{S})$.  In particular, if $(v_i,v_j,v_k)$ spans a triangle, then it spans a {\em heavy triangle}.

\end{defin}

\begin{lem}\label{l: no triangle}

$G(\overline{S})$ does not contain a heavy triple.

\end{lem}

\begin{proof}

Since $v_i,v_j,v_k$ belong to the same component of $G(\overline{S})$, the intervals $T_i,T_j,T_k$ are subsets of some path $P \subset G - r$.  Assume without loss of generality that $z_i$ lies between $z_j$ and $z_k$ on $P$.  Every unbreakable interval in $P$ that avoids $z_i$ lies to one side of it, and $T_j$ and $T_k$ lie to opposite sides by assumption.  As each element in $\overline{S}$ is unbreakable and $T_i$ is the unique interval containing $z_i$, it follows that $v_j$ and $v_k$ lie in separate components of $G(\overline{S}) -v_i$.

\end{proof}

\begin{lem}\label{l: cycle}

Every cycle in $G(\overline{S})$ has length $3$ and contains a unique vector $v_i$ of norm $2$. Furthermore, if $(v_i,v_j,v_k)$ is a cycle with $i < j < k$ and $v_k$ is not gappy, then $|v_l| = 2$ for all $l \leq i$ if $\overline{S} = S$, or for all $t < l \leq i$ otherwise.

\end{lem}

\begin{proof}

Choose a cycle $C \subset G(\overline{S})$.  By Lemma \ref{l: triangle}, it follows that $V(C)$ induces a complete subgraph.  Thus, $V(C)$ cannot contain three vectors of norm $\geq 3$, for then they would span a heavy triangle, in contradiction to Lemma \ref{l: no triangle}.  Note also that the vectors of norm $2$ in $S$ induce a union of paths in $G(S)$.  Therefore, $V(C)$ cannot contain more than two such vectors.  If it did contain two, then they must take the form $v_{i+1} = -e_{i+1}+e_i$ and $v_i = -e_i + e_{i-1}$.  Choose any other $v_j \in V(C)$.  Then either $v_{j,i} = 0$ and $v_{j,i\pm1} = 0$, or else $v_{j,i} = 0$ and $v_{j,i\pm1} = 1$.  In the first case, $v_j \ll v_j - v_{i+1}$, and in the second, $v_j \ll v_j - v_i$, both of which entail contradictions.  It follows that $V(C)$ contains at most two vectors of norm $\geq 3$ and at most one vector of norm $2$, hence exactly that many of each.  This establishes the first part of the Lemma.

For the second part, it follows at once that $\min(\supp(v_k)) = i$ and $|v_i|=2$.  If $|v_l| \geq 3$ for some largest value $l < i$ and $l \ne t$, then $(v_l,v_{l+1},\dots,v_i)$ induces a path in $G(\overline{S})$, and then $(v_l,v_j,v_k)$ forms a heavy triple. The second part now follows as well.

\end{proof}

\begin{cor}\label{c: cycle}

If $C \subset S$ spans a cycle in $G(S)$, then it induces a complete subgraph and $|V(C)| \leq 4$, with equality iff $C$ contains a breakable vector $v_t$. \qed

\end{cor}

\begin{defin}\label{d: triangle sign}

If $(v_i,v_j,v_k)$ spans a triangle in $G(S)$, then it is {\em positive} or {\em negative} according to the sign of $\L v_i,v_j \R \cdot \L v_j,v_k \R \cdot \L v_k,v_i \R$.

\end{defin}

\begin{lem}\label{l: signs}

If $(v_i,v_j,v_k)$ spans a triangle in $G(S)$ and some pair of $T_i, T_j, T_k$ are consecutive, then the triangle is positive.

\end{lem}

\begin{proof}

Observe that $\L v_i,v_j \R \cdot \L v_j,v_k \R \cdot \L v_k,v_i \R = (\epsilon_i \epsilon_j \epsilon_k)^2 \cdot \L [T_i],[T_j] \R \cdot \L [T_j],[T_k] \R \cdot \L [T_k],[T_i] \R$.  Two pairs of $T_i,T_j,T_k$ are consecutive and the other pair shares a common endpoint, so the right-hand side of this equation is positive.

\end{proof}

Most of the case analysis to follow in Sections \ref{s: decomposable} - \ref{s: tight} involves arguing that elements of $S$ must take a specific form, for otherwise we would obtain a contradiction to one of the preceding Lemmas.  In such cases, we typically just state something to the effect of ``$(v_i,v_j,v_k)$ forms a negative triple" without the obvious conclusion ``a contradiction", to spare the use of this phrase several dozen times.

We conclude with one last basic observation.

\begin{lem}\label{l: index s}

Suppose that $v_s \in S$ has norm $\geq 3$ with $s$ chosen smallest, and that it is not tight.  Then $v_s$ is just right, and $|v_s| \in \{s, s+1\}$.

\end{lem}

\begin{proof}

Recall that if $v_g$ is gappy, then $|v_{k+1}| \geq 3$ for a gappy index $k$.  By minimality of $s$, it follows that $v_s$ is just right.  If $|v_s| < s$, then $v_s = -e_s+e_{s-1}+\cdots + e_k$ for some $2 \leq k \leq s-2$, and $(v_k;v_{k-1},v_{k+1},v_s)$ induces a claw, a contradiction.

\end{proof}


\section{A decomposable lattice}\label{s: decomposable}

The goal of this section is to classify the changemaker lattices isomorphic to a sum of more than one linear lattice (Proposition \ref{p: decomposable structure}).  We begin with a basic result.

\begin{lem}\label{l: two summands}

A changemaker lattice has at most two indecomposable summands.  If it has two indecomposable summands, then there exists an index $s > 1$ for which $v_s = -e_s + \sum_{i=0}^{s-1} e_i$, $|v_i| = 2$ for all $1 \leq i < s$, and $v_s$ and $v_1$ belong to separate summands.

\end{lem}

\begin{proof}

For the first statement, it suffices to show that $\widehat{G}(S)$ has at most two connected components (cf. Subsection \ref{ss: generalities}).  Thus, suppose that $\widehat{G}(S)$ has more than one component, fix a component $C$ that does not contain $v_1$, and choose $s > 1$ smallest such that $v_s \in V(C)$.  Thus, $v_s \not\sim v_i$ for all $1 \leq i < s$.  Let $k = \min(\supp(v_s))$.  Then $k = 0$, since otherwise $v_s \sim v_k$.  Furthermore, $v_s$ is not gappy, for if $l$ is a gappy index, then $v_s \sim v_{l+1}$.  Therefore, $v_s = -e_s + \sum_{i=0}^{s-1} e_i$.  If $|v_i| \geq 3$ for some $i < s$, then $v_s \sim v_i$.  It follows that $|v_i| = 2$ for all $i < s$.  As $s$ is uniquely determined, it follows that $C$ is as well, so $\widehat{G}(S)$ contains exactly two components.  The statement of the Lemma now follows.

\end{proof}

For the remainder of the section, suppose that $L$ is a changemaker lattice isomorphic to a sum of two linear lattices. 

\begin{lem}\label{l: just right}

All elements of $S$ are just right.

\end{lem}

\noindent Thus, $G(S) = \widehat{G}(S)$.

\begin{proof}

Suppose that $v_t \in S$ were tight.  Then $\langle v_t, v_1 \rangle = 1$ and $\langle v_t, v_s \rangle \geq 1$ would imply that $v_1$ and $v_s$ belong to the same component of $\widehat{G}(S)$, a contradiction.

Next, suppose that $v_g \in S$ were gappy with $g$ chosen minimal.  Note that $| v_g | \geq 3$.  Let $k$ denote the minimal gappy index for $v_g$.  Since $v_{k+1}$ is not gappy, it follows that $v_{k+1,k-1} \ne 0$.  Since $\langle v_g, v_{k+1} \rangle \leq 1$ by Lemma \ref{l: cool}, it follows that $v_{g,k-1} = 0$, and now minimality of $k$ implies that $k = \min(\supp(v_g))$.  This implies that $v_g \sim v_k$.  We cannot have $v_{k+1} \sim v_k$, since this would force $| v_k | \geq 3$, and then $(v_k,v_{k+1},v_j)$ forms a heavy triangle, in contradiction to Lemma \ref{l: no triangle}.  Hence $v_k \not\sim v_{k+1}$.  As $G(S)$ does not contain a cycle of length $> 3$, it follows that $v_k$ and $v_{k+1}$ belong to separate components of $G(S_{g-1})$.  Since $G(S_{g-1})$ has at most two components, it follows that $G(S_g)$ is connected, hence $G(S)$ is as well, a contradiction.

\end{proof}

\begin{lem}\label{l: v_m of degree 2}

Suppose that $v_m$ has multiple smaller neighbors.  Then  $m > s+1$, and

\begin{itemize}

\item $v_m = -e_m + e_{m-1} + \cdots + e_{s-1}$,

\item $v_{s+1} = -e_{s+1}+e_s+e_{s-1}$, 

\item $v_s = -e_s + e_{s-1}+ \cdots + e_0$, and 

\item $|v_k| = 2$ for all other $k < m$.

\end{itemize}

\end{lem}

\begin{proof}

Suppose that $v_m \sim v_i, v_j$ with $i < j < m$.  As in the proof of Lemma \ref{l: just right}, the vectors $v_i$ and $v_j$ cannot belong to separate components of $G(S_{m-1})$, for then $G(S)$ would be connected.  Furthermore, $v_i \sim v_j$, since otherwise $G(S_m)$ would contain a cycle of length $> 3$.  By Lemma \ref{l: cycle}, it follows that $|v_l| = 2$ for all $l \leq i$.  Hence $s \geq i+1$.  From $v_j \sim v_i$, it follows that $j \geq s+1$ and $v_j = -e_j + e_{j-1} + \cdots + e_i$.  As $\L v_j, v_m \R \leq 1$, it follows that $j = i+2$.  Thus, $s = i+1$.  As $i$ and $j$ are uniquely determined, it follows that $v_m \not\sim v_k$ for all $s+2 < k < m$, so $|v_k| = 2$ for all such $k$.

\end{proof}

The following definition is essential to describe the way in which we build families of standard bases.  The terminology borrows from \cite[Definition 3.4]{lisca:lens1}, although its meaning differs somewhat.

\begin{defin}\label{d: expansion}

We call $S_m$ an {\em expansion} of $S_{m-1}$ if $v_m = -e_m+ e_{m-1} + \cdots + e_k$ for some $k$, $|v_i| = 2$ for all $k+ 1 < i < m$, and $| v_{k+1} | \geq 3$ in case $m > k+1$.

\end{defin}

\begin{lem}\label{l: v_m of degree 1}

Suppose that $v_m$ has a single smaller neighbor.  Then either 

\begin{enumerate}

\item $s = 2$, $|v_k|=2$ for $s < k < m$, and $v_m = -e_m + e_{m-1} + \cdots + e_0$;

\item $|v_k| = 2$ for $s < k < m$ and $v_m = -e_m + e_{m-1} + \cdots + e_s$; or

\item $S_m$ is an expansion of $S_{m-1}$.

\end{enumerate}

\end{lem}

\begin{proof}

Suppose first that $v_{m0} = 1$.  It follows by assumption on $v_m$ that $|v_k|=2$ for all $1 \leq k < n$ except for a single $v_i$, for which $|v_i | = 3$.  On the other hand, $| v_s | = s+1$.  It follows that $s = 2$, and (1) holds.

Thus, we may assume that $v_m = -e_m + e_{m-1} + \cdots + e_i$ for some $i > 0$.  Now $v_m \sim v_i$, so $| v_k | = 2$ for all $i+1 < k < m$.  Suppose that $i+1 < m$ and $| v_{i+1} | = 2$.  If $v_i \sim v_l$ with $l < i$, then $(v_i;  v_l, v_{i+1}, v_m)$ induces a claw, in contradiction to Lemma \ref{l: no claw}.  It follows that $i = s$ and (2) holds.

The remaining cases that $m = i+1$, or that $m > i+1$ and $|v_{i+1}| > 2$, both result in (3).

\end{proof}

\begin{defin}

Let $A_{s,m}$ denote the family of standard bases enumerated in Lemma \ref{l: v_m of degree 2} and $B_m$, $C_{s,m}$ the families enumerated in Lemma \ref{l: v_m of degree 1}(1) and (2), respectively.

\end{defin}

Combining Lemmas  \ref{l: v_m of degree 2} and  \ref{l: v_m of degree 1} and induction, we obtain the following structural result.

\begin{prop}\label{p: decomposable structure}

Suppose that a changemaker lattice is isomorphic to a sum of more than one linear lattice.  Then its standard basis is built by a sequence of (possibly zero) expansions to $A_{s,m}$, $B_m$, $C_{s,m}$, or $\varnothing$, for some $m > s \geq 2$. \qed

\end{prop}

In fact, we have established somewhat more: if a changemaker lattice is isomorphic to a linear lattice or a sum thereof, and $G(S_{n'})$ is disconnected for some $n' \leq n$, then $S_{n'}$ takes the form appearing in Proposition \ref{p: decomposable structure}.  We will utilize Proposition \ref{p: decomposable structure} in this stronger form on several occasions in Sections \ref{s: just right} - \ref{s: tight}.

We obtain vertex bases for the families in Proposition \ref{p: decomposable structure} as follows:

\begin{enumerate}

\item[$A_{s,m}$:]  $\{v_{m-1},\dots,v_{s+1},v_{s-1},\dots,v_1,-(v_m+v_1+\cdots+v_{s-1})\} \cup \{v_s\}$;

\item[$B_m$:] $\{ v_{m-1} \dots,v_2,-v_m\} \cup \{ v_1 \}$;

\item[$C_{s,m}$:] $\{ v_1, \dots, v_{s-1} \} \cup \{ v_{m-1}, \dots, v_s, v_m \}$.

\end{enumerate}

\noindent For expansion on $\varnothing$, the standard basis is, up to reordering, a vertex basis (cf. Proposition \ref{p: expansion}).


\section{All vectors just right}\label{s: just right}

Before proceeding further, we briefly comment on the purpose of this and the next two sections, and establish some notation.  Just as Proposition \ref{p: decomposable structure} describes the structure of the standard basis for a changemaker lattice isomorphic to a sum of more than one linear lattice, our goal in Sections \ref{s: just right} - \ref{s: tight} is to produce a comprehensive collection of {\em structural Propositions} that do the same thing for the case of a single linear lattice.  In each Proposition we enumerate a specific family of standard bases, and in Section \ref{s: cont fracs} we verify that each basis does, in fact, span a linear lattice by converting it into a vertex basis.  

{\em A posteriori}, each standard basis $S = \{v_1,\dots,v_n\}$ contains at most one tight vector and two gappy vectors.  We always denote the tight vector by $v_t$.  We denote the gappy vector with the smaller index by $v_g$, which always takes the form $e_k + e_j + e_{j+1} + \cdots + e_{g-1}-e_g$ with $k < j+1$.  When there are two gappy vectors (\ref{p: breakable 1}(1), \ref{p: breakable 2}(2), \ref{p: breakable 3}(1,2)), we specifically notate the one with the larger index.  We write $s = \min \{ i \; | \; |v_i| > 2 \}$ when there is no tight vector, and $s =\min  \{ i > t \; | \; |v_i| > 2 \}$ when there is one.  Otherwise, every standard basis element $v_i$ is just right, so is completely determined by $i$ and its norm, which we report iff $|v_i| \geq 3$.  In \ref{p: gappy structure}(2) and \ref{p: breakable 1}-\ref{p: breakable 3} we report some families of standard bases {\em up to truncation}.  Thus in \ref{p: gappy structure}(2), we may truncate by taking $n = g$ and disregarding $v_i$ for $i \geq g+1$.

\noindent {\em Example.} The first structural Proposition \ref{p: just right 1}(1) reports the family of standard bases parametrized by $s \geq 2$, where 
\begin{itemize}
\item $v_i = e_{i-1}-e_i$ for $i=1,\dots,s-1$;
\item  $v_s = e_0 + \cdots + e_{s-1} - e_s$;
\item $v_{s+1} = e_{s-1} + e_s - e_{s+1}$;
\item $v_{s+2} = e_s + e_{s+1} - e_{s+2}$; and
\item $v_n = v_{s+3} = e_{s-1} + e_s + e_{s+1} + e_{s+2} - e_{s+3}$.
\end{itemize}

\medskip

For the remainder of this section, assume that $L$ is a changemaker lattice isomorphic to a linear lattice, and that every element of $S$ is just right (hence also unbreakable).


\subsection{$G(S)$ contains a triangle.}  A {\em sun} is a graph consisting of a triangle $\Delta$ on vertices $\{a_1,a_2,a_3\}$ together with three vertex-disjoint paths $P_1,P_2,P_3$ such that $a_i$ is an endpoint of $P_i$, $i = 1,2,3$.  The other endpoints of the $P_i$ are the {\em extremal vertices} of the graph.

\begin{lem}\label{l: sun}

If $G(S)$ contains a triangle $\Delta$, then $G(S)$ is a sun and $V(\Delta) = \{v_i,v_{i+2},v_m\}$ for some $i+2 < m$.  Furthermore, $|v_l| = 2$ for all $v_l$ along the path containing $v_i$.

\end{lem}

\begin{proof}

Choose a triangle $\Delta \subset G(S)$ with $V(\Delta) = \{ v_i,v_j,v_m \}$, $i < j < m$.  By Lemma \ref{l: cycle}, $|v_l| = 2$ for all $l \leq i$ and $|v_j|,|v_m| \geq 3$.  Since $v_j$ and $v_m $ are just right, we have $v_j = -e_j+\cdots+e_i$ and $v_m = -e_m+ \cdots + e_i$.  Since $\L v_m, v_j \R \leq 1$, it follows that $j = i+2$.

Suppose by way of contradiction that $\Delta' \subset G(S)$ were another triangle with $V(\Delta') = \{ v_{i'},v_{j'},v_{m'} \}$, $i' < j' < m'$.  Then $|v_l| = 2$ for all $l \leq i'$ and $j' = i'+2$, so $i' \in \{i-1,i,i+1\}$.  If $i' = i$, then $\L v_{m'},v_m \R \geq 2$, which cannot occur.  If $i' = i + 1$, then $(v_{i+3},v_{i+1},v_i)$ is a path in $G(S)$, which implies that $(v_{i+2},v_{i+3},v_m)$ forms a heavy triple, a contradiction.  By symmetry, $i = i' + 1$ cannot occur either.

Consequently, $G(S)$ contains a unique triangle $\Delta$.  Furthermore, $G(S)$ is claw-free by Lemma \ref{l: no claw}.  It follows that $G(S)$ is a sun.  If some vector $v_l$ on the path containing $v_i$ had norm $\geq 3$, then $(v_l,v_k,v_m)$ forms a heavy triple, so this does not occur.

\end{proof}

\begin{prop}\label{p: just right 1}

Suppose that every element in $S$ is just right, and that $G(S)$ contains a triangle.  Then either

\begin{enumerate}

\item $n = s+3$, $|v_s| = s+1$, $|v_{s+1}|=|v_{s+2}|=3$, and $|v_{s+3}|=5$;

\item $n = s+3$, $|v_s| = s+1$, $|v_{s+1}|=3$, and $|v_{s+2}| = |v_{s+3}|=4$; or

\item $|v_s| = s = 3$, $|v_m| = m$ for some $m > 3$, $|v_i| = 2$ for all $i < m$, $i \ne 3$, and $S$ is built from $S_m$ by a sequence of expansions.

\end{enumerate}

\end{prop}

\begin{proof}

We apply Lemma \ref{l: sun}, keeping the notation therein.

\noindent (I) {\em Suppose that $|v_{i+1}| > 2$.}

\noindent In this case, $s = i+1$ and $v_s$ has no smaller neighbor, so $|v_s| = s+1$.  Since $G(S)$ is connected, $v_s$ has some neighbor $v_j = -e_j+\cdots+e_l$.  Note that $v_s \not\sim v_{s+1},v_m$, so  $j \ne s+1,m$.  If $l < s$, then in fact $l < s-1$ and $(v_{s+1},v_j,v_m)$ forms a heavy triangle, so it follows that $l = s$.  Since $1 \geq \L v_j,v_m \R = \min\{m,j\} - (s+1) \geq 1$, it follows that $\min\{m,j\} = s+2$.

\noindent (I.1) {\em Suppose that $j = s+2$.}

\noindent The subgraph $H$ of $G(S)$ induced on $\{v_1,\dots,v_{s+2},v_m\}$ is a sun with extremal vertices $v_1,v_{s},v_{s+1}$.  We claim that $G(S) = H$.  For if not, then there exists some vector $v_l = -e_l + \cdots + e_k$ with a single edge to $H$, meeting it in an extremal vertex.  This forces $k \leq s+1$, but then $\L v_l, v_m \R \geq 1$, a contradiction.  It follows that $n=m = s+3$, and case (1) results.

\noindent (I.2) {\em Suppose that $m = s+2$.}

\noindent The subgraph $H$ induced on $\{v_1,\dots,v_{s+2},v_j\}$ is a sun with extremal vertices $v_1,v_s,v_{s+1}$ like before.  The argument just given (with $v_j$ in place of $v_m$) applies to show that $G(S) = H$.  It follows that $n = j = s+3$, and case (2) results.

\noindent (II) {\em Suppose that $|v_{i+1}| = 2$.}

\noindent In this case, $s = i+2 = 3$ and $|v_3| = 3$.  If $|v_j| \geq 3$ for some $3 < j < m$ chosen smallest, then the subgraph induced on $V' = \{v_1,\dots,v_{j-1},v_m\}$ is a sun in which $v_j$ has multiple neighbors, which cannot occur in the sun $G(S)$.  Thus, $|v_j| = 2$ for all $3 < j < m$.  Next, choose any $v_j = -e_j+\cdots+e_k$ with $j > m$.  Then $G(S_{j-1})$ is a sun, so $v_j$ has exactly one smaller neighbor.  It easily follows that $k \geq 2$, $v_j \sim v_k$, and $v_k$ has some smaller neighbor $v_l$.  If $|v_j| \geq 3$, then $|v_{k+1}| \geq 3$, since otherwise $(v_k;v_l,v_{k+1},v_j)$ induces a claw.  It follows that $G(S_j)$ is an expansion of $G(S_{j-1})$.  By induction, $G(S)$ is a sequence of expansions applied to $G(S_m)$, and case (3) results.

\end{proof}


\subsection{$G(S)$ does not contain a triangle.}

In this case, $G(S)$ is a path.

\subsubsection{Some vertex has multiple smaller neighbors.}   Suppose that $v_m \in S$ has multiple smaller neighbors.  Since $G(S)$ is a path, it follows that $G(S_{m-1})$ consists of a union of two paths and that $v_m$ is adjacent precisely to one endpoint of each.  Therefore, $m$ is the minimal index for which $G(S_m)$ is connected, which establishes that $m$ is unique.

\begin{lem}\label{l: little}

Suppose that every element in $S$ is just right, $G(S)$ does not contain a triangle, $v_m \in S$ has multiple smaller neighbors, and $v_{m-1}$ is not an endpoint of $G(S)$.  Then $m = n$.

\end{lem}

\begin{proof}

Suppose by way of contradiction that $n > m$, and consider $v_{m+1}$.  Its unique smaller neighbor $v_j$ is an endpoint of the path $G(S_m)$, and $j < m-1$ by hypothesis.  Therefore, $|v_{m+1}| \geq 4$, but this implies that $\L v_{m+1}, v_m \R \geq 1$, a contradiction.

\end{proof}

\begin{prop}\label{p: just right 2}

Suppose that every element in $S$ is just right, $G(S)$ does not contain a triangle, and some vector $v_m \in S$ has multiple smaller neighbors.  Then either

\begin{enumerate}

\item $m=n=4$, $|v_2| = |v_3| = 3$, and $|v_4| = 5$;

\item $m=n=4$, $|v_2|=3$, and $|v_3|=|v_4|=4$;

\item $m=n = s+3$, $|v_s| = s+1$, $|v_{s+2}|=3$, and $|v_{s+3}|=5$;

\item $m=n = s+3$, $|v_s| = s+1$, and $|v_{s+2}| = |v_{s+3}| = 4$; or 

\item $s=3$, $|v_3|=4$, $|v_m|= m > 3$, and $|v_{m+1}| =3$ in case $n > m$.

\end{enumerate}

\end{prop}

\begin{proof}

Write $v_m = -e_m + \cdots + e_k$.

\noindent (I)  {\em Suppose that $k=0$.}

\noindent Then $\L v_m, v_s \R = s-1 \geq 1$ forces $s = 2$.  Further, $v_m$ is not adjacent to $v_1$, but as $v_m$ has a neighbor in the component of $G(S_{m-1})$ containing it, $v_1$ must have some neighbor $v_i = -e_i + \cdots + e_1$ with $i \geq 3$.  Then $\L v_m, v_i \R \geq i-2 \geq 1$, so $i=3$.  Hence $v_m$ neighbors $v_2$ and $v_3$ but no other $v_j$, $j < m$.  It follows that $|v_j| = 2$ for $3<j<m$.  If $m > 4$, then $(v_3;v_1,v_4,v_m)$ induces a claw.  Hence $m=4$.  By Lemma \ref{l: little}, it follows that $m = n$, and case (1) results.

\noindent (II)  {\em Suppose that $k > 0$.}  Hence $v_m \sim v_k, v_j$ for some $j \geq k+2$.  

\noindent (II.1) {\em Suppose that $j > k+2$.}

\noindent  In this case, $|v_j|=3$ and $|v_{j-1}|=2$, so $v_{j-2}$ neighbors $v_j$ and $v_{j-1}$ and no other vector.  It follows that $j-2 \in \{1,s\}$.  But $1=j-2 > k >0$ cannot occur, so $j-2=s$.  Since $0=\L v_m,v_s \R = s-k-1$, it follows that $s = k+1$.  Moreover, since $v_j$ is an endpoint of its path in $G(S_{m-1})$, it follows that $m=j+1$.  By Lemma \ref{l: little}, it follows that $m = n$, and case (3) results.

\noindent (II.2) {\em Suppose that $j = k+2$.}  Since $v_j$ and $v_k$ belong to different components of $G(S_{m-1})$, it follows that $|v_k|=2$ and $|v_{k+2}| \geq 4$.  

\noindent (II.2.i) {\em Suppose that $v_{k+2}$ has no smaller neighbor.}

\noindent In this case, $|v_{k+1}|=2$ and $k+2 = s$.  As $v_k \sim v_{k+1}$ and $v_k$ is an endpoint of its path in $G(S_{m-1})$, it follows that $v_k$ has no smaller neighbor, hence $k = 1$.  It follows that $s=3$, $|v_3| = 4$, $|v_m| = m$, and $|v_i| = 2$ for all other $i < m$.  If $n = m$, then we land in case (5).  If $n > m$, then consider $v_{m+1}$.  Its unique smaller neighbor $v_j$ is an endpoint of $G(S_m)$.  It follows that $j = m-1$ and $|v_{m+1}|=3$.  By a similar argument, it follows that $|v_i|=2$ for all $i > m+1$.  Therefore, case (5) results.

\noindent (II.2.ii) {\em Suppose that $v_{k+2}$ has a smaller neighbor.}  Thus, it has no larger neighbor in $G(S_m)$ besides $v_m$, so $m = k+3$.  

\noindent (II.2.ii$'$) {\em Suppose that $|v_{k+1}| \geq 3$.}

\noindent It follows that $v_{k+1} \sim v_{k+2}$, and $v_{k+2}$ cannot have any other smaller neighbor.  Hence $\min(\supp(v_{k+2}))=0$, $s = k+1$, and $1 \geq |\L v_{k+1},v_k \R| = k$, so $k =1$.    By Lemma \ref{l: little}, it follows that $m = n$, and case (2) results.

\noindent (II.2.ii$''$) {\em Suppose that $|v_{k+1}|=2$.}

\noindent In this case, $v_k \sim v_{k+1}$, so $v_k$ has no smaller neighbor.  Hence $k \in \{1,s\}$.  However, if $k = 1$, then $v_{k+2}$ would not have a smaller neighbor.  Hence $k = s$, and as $\L v_{k+2}, v_s \R = 0$ but $v_{k+2}$ has a smaller neighbor, it follows that $|v_{k+2}| = 4$.  By Lemma \ref{l: little}, it follows that $m = n$, and case (4) results.

\end{proof}

\subsubsection{No vector has multiple smaller neighbors.}

\begin{prop}\label{p: just right 3}

Suppose that every element in $S$ is just right, $G(S)$ does not contain a triangle, and no vector in $S$ has multiple smaller neighbors.  Then either

\begin{enumerate}

\item $|v_i| = 2$ for all $i$;

\item $s = 3, |v_3|=3,|v_4|=5$; or

\item $|v_s| = s$ and $S$ is built from $S_s$ by a sequence of expansions.

\end{enumerate}

\end{prop}

\begin{proof}

It must be the case that $G(S_j)$ is connected for all $j$, since $G(S)$ is connected, and if $G(S_{m-1})$ were disconnected for some $m > 0$, then $v_m$ would have multiple smaller neighbors.   Thus, unless case (1) occurs, it follows that $|v_s| = s \geq 3$.

\noindent (I) {\em Suppose that there exists an index $m > s$ for which $\min(\supp(v_m))=0$.}

\noindent In this case, $\L v_m, v_s \R = s-2 \geq 1$.  It follows that $s = 3$ and $v_3$ is the unique smaller neighbor of $v_m$.  If $m > 4$, then $(v_3;v_2,v_4,v_m)$ induces a claw.  Hence $m = 4$.  If $n > 4$, then choose $j$ maximal for which $|v_5|=\cdots=|v_j|=2$.  Thus, $G(S_j)$ has endpoints $v_j$ and $v_2$.  If $n > j$, then $v_{j+1}$ must neighbor one of $v_j$ and $v_2$.  However, $v_{j+1} \sim v_2$ implies that $\L v_{j+1}, v_4 \R \geq 1$, while $v_{j+1} \sim v_j$ implies that $|v_{j+1}|=2$.  Both result in contradictions, so it follows that $n = j$, and case (2) results.

\noindent (II) {\em Suppose that $\min(\supp(v_m)) > 0$ for all $m > s$.}

\noindent Consider $v_m = -e_m + \cdots + e_k$ with $m > s$ and $k > 0$.  Thus, $v_k$ is the unique smaller neighbor of $v_m$, so it is an endpoint of $G(S_{m-1})$, and $|v_i| = 2$ for all $k+1 < i < m$.  If $m > k+1$ and $|v_{k+1}| = 2$, then $v_k \sim v_{k+1}$.  As $v_k$ is an endpoint of $G(S_{m-1})$, it follows that $v_k$ has no smaller neighbor.  But then $k = 1$ and $|v_i| = 2$ for all $i < m$, in contradiction to the assumption that $m > s$.  It follows that $|v_{k+1}| \geq 3$, or else $m = k+1$ and $|v_m| = 2$.  Hence $S_m$ is an expansion on $S_{m-1}$.  By induction on $m$, it follows that case (3) results.

\end{proof}


\section{A gappy vector, but no tight vector}\label{s: gappy, no tight}

In this section, assume that $L$ is a changemaker lattice isomorphic to a linear lattice, $S$ does not contain a tight vector, and it does contain a gappy vector $v_g$.  Again, every vector in $S$ is unbreakable, and $G(S) = \widehat{G}(S)$.  For use in Section \ref{s: tight}, the following Lemma allows the possibility that $S$ contains a tight, unbreakable vector.

\begin{lem}\label{l: one gappy}

Suppose that $v_g \in S$ is gappy, and that $S$ contains no breakable vector.  Then $v_g$ is the unique gappy vector, $v_g = -e_g + e_{g-1} + \cdots + e_j + e_k$ for some $k +1 < j < g$, and $v_k$ and $v_{k+1}$ belong to distinct components of $G(S_{g-1})$.

\end{lem}

\begin{proof}

Choose $v_g$ with $g$ minimal, and choose a minimal gappy index $k$ for $v_g$.  Then $|v_{k+1}| \geq 3$,  and since $v_{k+1}$ is not gappy, it follows that $v_{k+1,k-1} = 1$.  Thus, $v_{g,k-1}=0$, since otherwise $\L v_g, v_{k+1} \R \geq 2$.  It follows that $v_g \sim v_{k+1},v_k$.  If $v_k \sim v_{k+1}$, then either $|v_k| \geq 3$, or else $k = 1$ and $v_2$ is tight.  In the first case, the triangle $(v_k,v_{k+1},v_g)$ is heavy, and in the second case, it is negative.  Hence $|v_k| = 2$ and $v_k \not\sim v_{k+1}$.  If $v_k$ and $v_{k+1}$ were in the same component of $G(S_{g-1})$, then a shortest path between them, together with $v_g$, would span a cycle of length $> 3$ in $G(S)$.  It follows that $G(S_{g-1})$ has two components, and $v_k$ and $v_{k+1}$ belong to separate components.

Suppose by way of contradiction that $l > k$ were another gappy index.  Then $|v_{l+1}| \geq 3$, so $v_{k+1} \not\sim v_{l+1}$, since otherwise $(v_{k+1},v_{l+1},v_g)$ forms a heavy triangle.  Furthermore, $v_{l+1,k} = 0$, since otherwise $\L v_g, v_{l+1} \R \geq 2$.  It follows that $v_{l+1} \not\sim v_k$, too.  But then $(v_g; v_k,v_{k+1},v_{l+1})$ induces a claw.  Hence no other index $l$ exists, and $v_g$ takes the stated form.

Lastly, suppose by way of contradiction that $v_h$ were another gappy vector, with $h > g$ chosen smallest.  Note that $G(S_{h-1})$ is connected.  It follows that $v_h$ has at most two smaller neighbors and that they are adjacent, since otherwise there would exist a cycle of length $> 3$ in $G(S)$.  Choose a minimal gappy index $k'$ for $v_h$ and let $l = \min(\supp(v_h))$.  Then $|v_{k'+1}| \geq 3$, and since $\L v_g, v_{k'+1} \R \leq 1$, it follows that $v_{k'+1,i} = 0$ for $i = l,\dots,k'-1$.  Thus, $v_{k'+1,i}=1$ for some $i < l$, whence $l > 0$.  Thus, $v_h \sim v_{k'+1}, v_l$, so $v_{k'+1} \sim v_l$.  However, $|v_l|=2$, so it follows that $v_{k'+1,l-1} = 1$; but then $v_{k'+1} \ll v_{k'+1} - v_l$, a contradiction.  It follows that $v_g$ is the unique gappy vector, as claimed.

\end{proof}

By Lemma \ref{l: one gappy}, it follows that $G(S_{g-1})$ is disconnected, so $S_{g-1}$ must take one of the forms described by Proposition  \ref{p: decomposable structure}.  Lemmas \ref{l: S_g nothing} and \ref{l: S_g not nothing} condition on these possible forms to determine the structure of $S_g$.

\begin{lem}\label{l: S_g nothing}

Suppose that $S_{g-1}$ is built from $\varnothing$ by a sequence of expansions.  Then $S_g$ takes one of the following forms:

\begin{enumerate}

\item  $k = s-1$, $j = s+1$, $|v_s| = s+1$, and $|v_{s+2}|=4$;

\item $j = k+2$, $|v_i| = 2$ for all $i >k+1$, and otherwise $S_{g-1}$ is arbitrary;

\item $k = s$, $j = s+2$, $|v_s| = s+1$, $|v_{s+1}|=3$, and $|v_{s+2}|=3$;

\item $k = 1$, $s = 2$, $|v_2| = 3$, $|v_j|=j$, and $|v_{j+1}|=3$; or

\item $k = 1$, $s = 2$, $|v_2|=3$, and $|v_j| = j$.

\end{enumerate}

\end{lem}

\begin{proof}

\noindent (I)  {\em Suppose that $v_g \sim v_j$.}

\noindent Thus, $v_{jk} = 0$.  If $v_j \not\sim v_{k+1}$, then $(v_g;v_k,v_{k+1},v_j)$ induces a claw.  Hence $v_j \sim v_{k+1}$.  If $|v_j| \geq 3$, then $(v_{k+1},v_j,v_g)$ forms a heavy triangle.  Hence $|v_j|=2$ and $j = k+2$. 

\noindent (I.1)  {\em Suppose that $v_l \sim v_k$ with $k < l < g$.}

\noindent Thus, $v_l = -e_l+\cdots+e_k$.  Then $l > k+2$ and $\L v_g, v_l \R \leq 1$, which implies that $l = k+3$ and $|v_{k+3}|=4$.  If $|v_i| \geq 3$ for some largest value $i \leq k$, then $(v_i,v_{k+3},v_g)$ forms a heavy triple.  Hence $|v_i| = 2$ for all $i \leq k$, and $s = k+1$.  If $v_l \sim v_{s+1}$ for some $s+1 < l < g$, then $(v_g;v_{s-1},v_s,v_l)$ induces a claw.  It follows that $|v_l| = 2$ for all $s+1 < l < g$, and case (1) results.

\noindent (I.2)  {\em Suppose that $v_l \not\sim v_k$ for all $k+1 < l < g$.}

\noindent It follows that $|v_l|=2$ for all such $l$, and case (2) results.

\noindent (II)  {\em Suppose that $v_g \not\sim v_j$.}

\noindent Thus, $v_{jk} = 1$.  Furthermore, the assumption on $S_{g-1}$ implies that $k = \min(\supp(v_j))$ and that $|v_i| = 2$ for all $k+1<i<j$.  If $v_k$ has a smaller neighbor $v_l$, then $(v_k;v_l,v_j,v_g)$ induces a claw.  It follows that $k \in \{1,s\}$.

\noindent (II.1)  {\em Suppose that $k = s$.}

\noindent Since $|v_s| \geq 3$, it follows that $|v_{s+1}|=3$.  If $j > s+2$, then $|v_{s+2}|=2$ and $(v_{s+1};v_{s-1},v_{s+2},v_g)$ induces a claw.  Hence $j = s+2$.  If $v_{s+1} \sim v_l$ for some $s+2 < l < g$, then either $l = s+3$, in which case $(v_s,v_l,v_g)$ forms a heavy triangle, or else $l > s+3$, in which case $\L v_g,v_l \R \geq 2$.  It follows that $|v_l|=2$ for all $s+2 < l < g$, and case (3) results.

\noindent (II.2)  {\em Suppose that $k = 1$.} It follows that $s = 2$.

\noindent (II.2$'$)  {\em Suppose that $v_l \sim v_{j-1}$ for some $j < l < g$.}

\noindent If $v_g \sim v_l$, then $(v_{j-1},v_l,v_g)$ forms a heavy triple.  Hence $v_g \not\sim v_l$, so $l = j+1$ and $|v_{j+1}|=3$.  If $v_i \sim v_j$ for some $j < i < g$, then $(v_j,v_i,v_g)$ forms a heavy triple.  Hence $|v_i| = 2$ for all $j+1 < i < g$, and case (4) results.

\noindent (II.2$''$)  {\em Suppose that $v_l \not\sim v_{j-1}$ for all $j < l < g$.}

\noindent It follows that $|v_l|=2$ for all $j < l < g$, and case (5) results.

\end{proof}

\begin{lem}\label{l: S_g not nothing}

Suppose that $S_{g-1}$ is not built from $\varnothing$ by a sequence of expansions. Then $S_g$ takes one of the following forms:

\begin{enumerate}

\item $k=s$, $j=s+2$, $|v_s| = s+1$, $|v_{s+1}|=3$, and $|v_{s+2}| = 4$;

\item $k=1$, $j=3$, $s = 2$, $|v_2| = 3$, and $|v_3| =4$; or

\item $k = s-1$, $j = s+1$, $|v_s| = s+1$, and $|v_{s+2}|=3$.

\end{enumerate}

\end{lem}

\begin{proof}

Since $S_{g-1}$ is not obtained from $\varnothing$ by a sequence of expansions, Proposition \ref{p: decomposable structure} implies that $S_{g-1}$ is built by applying a sequence of expansions to $A_{s,m}, B_m$, or $C_{s,m}$, for some $m > s \geq 2$.  We consider these three possibilities in turn.

\noindent (I) $S_m = A_{s,m}$.

\noindent In this case, $v_s$ is a singleton in $G(S_{g-1})$.  It follows that $s \in \{ k,k+1 \}$.  If $s = k+1$, then since $(v_{s-1},v_s,v_m)$ spans a triangle and $v_{s-1} \sim v_g$, it follows that $(v_s,v_m,v_g)$ forms a heavy triple.  Therefore, $s = k$.  If $m \ne j$, then $(v_{s+1},v_m,v_g)$ forms a heavy triangle.  Therefore, $m = j$.  If $m > s+2$, then $|v_{s+1}|=2$, and $(v_{s+1};v_{s-1},v_{s+2},v_g)$ induces a claw.  Therefore, $m = s+2$.  It follows that $S_{g-1}$ is built from $A_{s,s+2}$ by a sequence of expansions.  If $v_l \sim v_{k+1}$ for some $s+2<l\leq g-1$, then either $l = s+3$ and $(v_{s+1},v_{s-1},v_{s+3},v_g)$ induces a claw, or else $l > s+3$ and $(v_{s+1},v_l,v_g)$ forms a heavy triple.  Consequently, no such $l$ exists, and therefore $|v_i| = 2$ for all $s+2 < i \leq g-1$.  This results in case (1).

\noindent (II)  $S_m = B_m$.

\noindent In this case, $v_1$ is a singleton in $G(S_{g-1})$, so $k = 1$.  If $m \ne j$, then $(v_2,v_m,v_g)$ forms a heavy triangle.  Therefore, $m = j$.  If $m > 3$, then $(v_2;v_3,v_m,v_g)$ induces a claw.  Hence $m = 3$.  It follows that $S_{g-1}$ is built from $B_4$ by a sequence of expansions.  If $v_l \sim v_2$ for some $3 < l \leq g-1$, then either $(v_3,v_l,v_g)$ forms a heavy triangle, or $(v_3;v_4,v_l,v_g)$ induces a claw.  It follows that $|v_i|=2$ for all $4 < i \leq g-1$.  This results in case (2).

\noindent (III) $S_m = C_{s,m}$.

\noindent In this case, $(v_1,\dots,v_{s-1})$ spans a component of $G(S_{g-1})$.  It follows that $k = s-1$.  If $v_m \sim v_g$, then $(v_s,v_m,v_g)$ forms a heavy triangle.  Hence $v_m \not\sim v_g$.  If $j > s+1$, then $(v_s;v_{s+1},v_m,v_g)$ induces a claw.  Hence $j = s+1$.  Since $v_m \not\sim v_g$, it follows that $m = s+2$.  Therefore, $S_{g-1}$ is built from $C_{s,s+2}$ by a sequence of expansions.  If $v_l \sim v_{s+1}$ for some $s+2 < l \leq g-1$, then $l = s+3$ since $\L v_g, v_l \R \leq 1$, and then $(v_g;v_{s-1},v_s,v_{s+3})$ induces a claw.  It follows that $|v_l|=2$ for all $s+2 < l < g$.  This results in case (3).

\end{proof}

\begin{lem}\label{l: gappy triangle}

Suppose that there exists $v_m \in S$ with multiple smaller neighbors, and $m > g$.  Then $m = g+1$, $g = s+2$, $|v_s|=s+1$, $v_{s+2} = -e_{s+2} + e_{s+1} + e_{s-1}$, $|v_{s+3}| = 5$, and $|v_i|=2$ for $i =1,\dots,s-1,s+1$.

\end{lem}

\begin{proof}

Since $G(S_{m-1})$ is connected and $G(S_m)$ does not contain a cycle of length $>3$, it follows that $v_m$ has precisely two smaller neighbors $v_a, v_b$ with $a < b$, and $(v_a,v_b,v_m)$ spans a triangle.   By Lemma \ref{l: cycle}, it follows that $v_m = -e_m+ \cdots + e_a$, $|v_l|=2$ for all $l \leq a$, and $|v_b|, |v_m| \geq 3$.  Furthermore, $|v_l|=2$ for all $l < m$, $l \ne s, b$.  As $|v_s|,|v_g| \geq 3,$ it follows that $s = a+1$ and $b = g$;  since $|v_{k+1}| \geq 3$, it follows that $k = s$; and since $\L v_m, v_g \R \leq 1$, it follows that $v_g = -e_g + e_{g-1} +e_{s-1}$ for some $g \geq s+2$.  If $g < m-1$, then $(v_g;v_{g-1},v_{g+1},v_m)$ induces a claw, and if $g > s+2$, then $(v_g;v_s,v_{g-1},v_m)$ induces a claw.  It follows that $m = g+1$, $g = s+2$, and $S_m$ takes the stated form.

\end{proof}

\begin{prop}\label{p: gappy structure}

Suppose that $S$ contains a gappy vector $v_g$ but no tight vector.  Then $S$ takes one of the following forms:

\begin{enumerate}

\item $n = g$ and $S$ is as in Lemma \ref{l: S_g nothing}(2);


\item $n \geq g$, and up to truncation, $|v_{g+1}|=3$, $|v_i|=2$ for all $g+1 < i \leq n$, and $S_g$ is as in Lemmas \ref{l: S_g nothing} or \ref{l: S_g not nothing}, except for Lemma \ref{l: S_g nothing}(2);


\item $S_{g+1}$ is as in Lemma \ref{l: gappy triangle} and $|v_i| = 2$ for all $g+1 < i \leq n$.


\end{enumerate}

\end{prop}

\begin{proof}

If $n = g$ then the result is immediate.  Thus, suppose that $n > g$, and select any $g < m \leq n$.  If $v_m$ has multiple smaller neighbors, then $m = g+1$ and $S_{g+1}$ takes the form stated in Lemma \ref{l: gappy triangle}.  Assuming this is not the case, $v_m$ has a unique smaller neighbor.  If $l := \min(\supp(v_m)) = 0$, then $v_m \sim v_s, v_g$, a contradiction.  Hence $l > 0$, and $v_l$ is the unique smaller neighbor of $v_m$.  Observe that $l \ne g$, since then $(v_g;v_k,v_{k+1},v_m)$ induces a claw.  It follows that $v_g$ has no larger neighbor.  If $|v_{l+1}| = 2$, then $s < g \leq l$, so $v_l$ has a smaller neighbor $v_i$, and then $(v_l;v_i,v_{l+1},v_g)$ induces a claw.  It follows that $|v_{l+1}| \geq 3$.

Consequently, if $m > g$ is chosen minimal with $|v_m| \geq 3$, then $m = g+1$ and either $S_{g+1}$ is as in Lemma \ref{l: gappy triangle}, or else $|v_{g+1}| = 3$.  Furthermore, there does not exist any $m' > m$ with $|v_{m'}| \geq 3$, since then $\min(\supp(v_{m'})) = g$ and $v_{m'} \sim v_g$, which does not occur.  Therefore, $|v_i| = 2$ for all $g+1 < i \leq n$.  Finally, $S_g$ cannot take the form stated in Lemma \ref{l: S_g nothing}(2), for then $(v_{k+1},v_g,v_{g+1})$ forms a heavy triple.  The statement of the Proposition now follows.

\end{proof}


\section{A tight vector}\label{s: tight}

Suppose that $S$ contains a tight vector $v_t$.  By Lemma \ref{l: tight}(1), the index $t$ is unique.  The arguments in this Section reach slightly beyond the criteria laid out in Subsection \ref{ss: int graph} that sufficed to carry out the analysis in Sections \ref{s: decomposable} - \ref{s: gappy, no tight}.  Nevertheless, the basic ideas are the same as before.


\subsection{All vectors unbreakable.}\label{ss: tight unbreakable}

Propositions \ref{p: berge viii} and \ref{p: tight unbreakable} describe the structure of a standard basis that contains a tight, unbreakable element.  However, we do not make any assumption on $v_t$  just yet, as these results will apply in Subsection \ref{ss: tight breakable}.

\begin{lem}

$S_{t-1}$ is built from $\varnothing$ by a sequence of expansions.

\end{lem}

\begin{proof}

If $|v_i|=2$ for all $i < t$, then the result is immediate, so suppose that $|v_s| \geq 3$ with $s < t$ chosen smallest.  Thus, $|v_s|=s$ or $s+1$.  Let us rule out the first possibility.  If $|v_s| = s$, then $s \geq 3$ and $\L v_t, v_s \R = s-2 \geq 1$.  Hence either $T_s \pitchfork T_t$, or else $s = 3$ and $T_s \dagger T_t$.  In the first case, $(v_1;v_2,v_s,v_t)$ induces claw, and in the second case, $(v_1,v_s,v_t)$ forms a negative triangle.  Therefore, $|v_s| = s+1$.

It follows that $\L v_t, v_s \R = |v_s|-1 \geq 2$, so that $T_s \prec T_t$.  As $\L v_1, v_s \R = 0$, it follows that $T_1$ and $T_s$ abut $T_t$ at opposite ends.  We claim that $v_1$ and $v_s$ belong to separate components of $G(S_{t-1})$.  For suppose the contrary, and choose a shortest path between them.  Together with $v_t$ they span a cycle of length $\geq 4$ in $G(S)$ that is missing the edge $(v_1,v_s)$, contradicting Corollary \ref{c: cycle}.

Therefore, $G(S_{t-1})$ is disconnected.  It follows by Proposition \ref{p: decomposable structure} that $S_{t-1}$ is built from $A_{s,m}, B_m, C_{s,m}$, or $\varnothing$ by a sequence of expansions.  Let us rule out the first three possibilities in turn.

\noindent (a) $A_{s,m}.$  Since $|v_m| \geq 4$ and $\L v_t, v_m \R = |v_m| -2$, it follows that $T_m \pitchfork T_t$.  Since $v_{s+1} \sim v_m$, it follows that $T_{s+1} \dagger T_m$, whence $T_{s+1} \; \cancel{\dagger} \; T_t$ since otherwise $T_m \dagger T_t$.  Hence $T_{s+1} \pitchfork T_t$ as well.  In particular, $z_{s+1},z_m \in T_t$.  On the other hand, $(v_1,\dots,v_{s-1},v_{s+1},v_m)$ induces a sun, with $|v_{s+1}|,|v_m| \geq 3$.  It follows that $T_1$ is contained in the open interval with endpoints $z_{s+1}$ and $z_m$, so that $T_1$ and $T_t$ do not abut, in contradiction to $v_1 \sim v_t$.

\noindent (b) $B_m.$  In this case, $T_2$ and $T_m$ both abut $T_t$, and at the opposite end as $T_1$.  As $|v_2|, |v_m| \geq 3$, it follows that both $T_2, T_m \prec T_t$.  Hence one of $T_2$, $T_m$ contains the other, in contradiction to their unbreakability.

\noindent (c) $C_{s,m}.$  Now $T_s \prec T_t$.  If $T_m \pitchfork T_t$, then $T_m$ and $T_t$ abut $T_s$ at opposite ends.  However, $T_{s+1}$ abuts $T_s$ as well, but $v_s \not\sim v_t, v_{s+1}$.  It follows that $m = s+2$ and $T_m \dagger T_t$.  But then $(v_s,v_m,v_t)$ forms a negative triangle.

It follows that $S_{t-1}$ is built from $\varnothing$ by a sequence of expansions, as desired.

\end{proof}

\begin{prop}\label{p: berge viii}

Suppose that $|v_i| \ne 2$ for some $i < t$.  Then $S = S_t$.

\end{prop}

\begin{proof}

We proceed by way of contradiction.  Thus, suppose that $S \ne S_t$, and consider $v_{t+1}$.  

Since $G(S_t)$ is a path, Lemma \ref{l: one gappy} implies that $v_{t+1}$ is not gappy.  Set $k := \min(\supp(v_{t+1}))$.  By Lemma \ref{l: tight}(3), it follows that $k > 0$.  Hence $\L v_t, v_{t+1} \R \leq |v_{t+1}| - 3$, so $|v_{t+1}| \in \{2,3,4\}$.  If $|v_{t+1}| = 2$, then $(v_t;v_1,v_s,v_{t+1})$ induces a claw.  Similarly, if $|v_{t+1}| = 4$, then $(v_t;v_1,v_s,v_{t+1})$ induces a claw unless $k \in \{1,s\}$; but if $k \in \{1,s\}$, then $(v_k,v_t,v_{t+1})$ forms a negative triangle.

It remains to consider the case that $|v_{t+1}|=3$.  In this case, $v_{t-1}$ is the unique smaller neighbor of $v_{t+1}$, and $v_t \not\sim v_{t+1}$, so $z_{t+1} \notin T_t$. Let $P \subset G(S)$ denote the induced path with consecutive vertices $(v_{i_1},\dots, v_{i_l})$, where $i_1 = t$ and $i_l =t+1$.  Thus, $i_2 \in \{1,s\}$ and $i_{l-1} = t-1$.  Observe that if $m < t$ is maximal with the property that $|v_m| \geq 3$, then $v_m \in V(P)$; in fact, $i_j = m$, where $j+m = t+1$.  Note that $z_{i_j} \notin T_{i_h}$ for all $h \ne 1,j$.  Let $x$ denote the endpoint of $T_t$ at which $T_{i_1} = T_t$ and $T_{i_2}$ abut.  Without loss of generality, suppose that $y$ is the left endpoint of $T_t$.  Thus, $T_{i_{j-1}}$ abuts the left endpoint of $T_m$.  It follows that $T_{i_{j+1}}$ abuts the right endpoint of $T_{i_j}$, since otherwise $(v_{i_{j-1}},v_{i_j},v_{i_{j+1}})$ induces a triangle, while $P$ is a path.  Hence $z_m$ separates $x$ from all $T_{i_h}$ with $h > j$.  In particular, $z_m$ separates $x$ from $z_{t+1} \in T_{t+1} = T_{i_l}$.  As $z_{t+1} \notin T_t$, it follows that $z_{t+1}$ lies to the right of $T_t$.  Hence $T_t \subset T := \bigcup_{h=2}^l T_{i_h}$.  However, $d(T_t) = t+4 > t+1 \geq d(T)$, a contradiction.

It follows that $v_{t+1}$ cannot exist, so $S = S_t$, as desired.

\end{proof}

Henceforth we assume that $|v_i|=2$ for all $i < t$.

\begin{prop}\label{p: tight unbreakable}

Suppose that $z_i \notin T_t$ for all $i \leq n' \leq n$ with $|v_i| \geq 3$.  Then $S_{n'}$ takes one of the following forms:

\begin{enumerate}

\item $t=1$, $|v_s| = s+1$ for some $s > 1$, $|v_i| = 2$ for all $1 < i < s$, and $S_{n'}$ is built from $S_s$ by a sequence of expansions;

\item $t=1$, $|v_s| = s$ for some $s > 1$, $|v_i| = 2$ for all $1 < i < s$, and $S_{n'}$ is built from $S_s$ by a sequence of expansions; or

\item $t > 1$, $|v_i| = 2$ for all $i < t$, and $S_{n'}$ is built from $S_t$ by a sequence of expansions.

\end{enumerate}

\end{prop}

\noindent Notice that Proposition \ref{p: tight unbreakable}(1) allows the possibility that $s=2$, a slight divergence from our convention on the use of $s$ stated at the outset of Section \ref{s: just right}.  Under the assumption that $n = n'$, Proposition \ref{p: tight unbreakable} produces three broad families of examples.  Assuming instead that $n > n'$, Propositions \ref{p: breakable 1}, \ref{p: breakable 2}, and \ref{p: breakable 3} utilize this result to produce even more.

\begin{proof}

By Lemma \ref{l: one gappy}, $S_{n'}$ does not contain a gappy vector.  Choose any $m > t$, and suppose by way of contradiction that $v_m$ had multiple smaller neighbors.  Since $G(S_{m-1})$ is connected and $G(S_m)$ does not contain a cycle of length $>3$, it follows that $v_m$ has exactly two smaller neighbors $v_k$ and $v_j$, $k < j$, and $v_k \sim v_j$.  Therefore, $|v_i| = 2$ for all $k+1 < i < m$, $i \ne j$, and since $G(S_m)$ does not contain a heavy triple, it follows that $|v_i| = 2$ for all $i \leq k$.  Hence $t \in \{k+1, j\}$.  However, if $t = k+1$, then $(v_t,v_j,v_m)$ forms a heavy triple, while if $t = j$, then $k = 1, t=3$, and $(v_k,v_t,v_m)$ forms a negative triangle.
Therefore, $v_m$ has exactly one smaller neighbor.

Set $k: = \min(\supp(v_m))$ and suppose that $k = 0$.  Then $\L v_t, v_m \R = t$, so it follows that $t = 1$.  Since $v_m$ has no other smaller neighbor, it follows that $|v_i| = 2$ for all $1 < i < m$.  Thus, $S_m$ takes the form stated in (2) with $m = s$.  Suppose instead that $|v_{k+1}| = 2$.  Then $v_k$ has no smaller neighbor $v_i$, since then $(v_k;v_i,v_{k+1},v_m)$ induces a claw.  As $t \leq k$, $G(S_k)$ is connected, so $k = t = 1$.  Thus, $S_m$ takes the form stated in (3) with $m = s$.  If neither $k = 0$ nor $|v_{k+1}|=2$, then it follows that $S_m$ is an expansion on $S_{m-1}$.  By induction, it follows that $S$ takes one of the forms stated in the Lemma.

\end{proof}


\subsection{A tight, breakable vector.}\label{ss: tight breakable}

Now we treat the case that $v_t$ is breakable.  This is the final and most arduous step in the case analysis, resulting in Propositions \ref{p: breakable 1}, \ref{p: breakable 2}, and \ref{p: breakable 3}.

\begin{lem}\label{l: breakable gappy}

Suppose that $v_t$ is breakable, $g \ne t$, $|v_g| \geq 3$, and $z_g \in T_t$.  Then $g > t+1$ and either $t > 1$, $v_g = -e_g+e_{g-1}+e_{t-1}$, and $T_g \pitchfork T_t$, or else $v_g = -e_g+e_{g-1}+e_{t-1}+\cdots+e_0$ and $T_g \prec T_t$.

\end{lem}

\noindent Note that we do not assume {\em a priori} that $v_g$ is gappy.

\begin{proof}

\noindent (a) $g > t+1$.

\noindent Otherwise, $\L v_t, v_g \R \in \{ |v_g| - 3, |v_g|-2 \}$, with the second possibility iff $\min(\supp(v_g)) = 0$.  Lemma \ref{l: tight 2} rules out the first possibility and Lemma \ref{l: tight}(3) the second.

It follows that $\supp(v_g)$ contains at least two values $> t$.  

\noindent (b) $v_{gt}=0$.

\noindent Otherwise, $1 \leq \L v_t, v_g \R \leq |v_g|-3$.  By Lemma \ref{l: tight 2}, we must have $\L v_t, v_g \R = 1$ and $|v_g| = 3$, so $v_g = -e_g + e_{g-1} + e_t$.  Now Lemma \ref{l: tight}(2) implies that $z_g \notin T_t$, a contradiction.

As $z_g \in T_t$, it follows that $\L v_t, v_g \R > 0$, so $v_g$ is gappy and there exists a gappy index $k < t$.  Since $|v_{k+1}| \geq 3$, it follows that $k = t-1$, and $\supp(v_g) \cap \{0,\dots,t-1\}$ consists of consecutive integers.

\noindent (c) $\supp(v_g)$ contains exactly two values $>t$.

\noindent Otherwise, $0 \leq \L v_t, v_g \R \leq |v_g| -2$, where the latter inequality is attained precisely when $v_g = -e_g+e_{g-1}+e_m+e_{t-1}+\cdots+e_0$ for some $t < m < g-1$.  Thus, $T_g \pitchfork T_t$, $\epsilon_g = \epsilon_t$, and $\epsilon_g([T_g - T_t] - [T_t - T_g])$ is reducible.  However, this equals  $v_g - v_t = -e_g + e_{g-1} + e_m + e_t -e_0$.  Since every non-zero entry in this vector is $\pm 1$, a decomposition $v_g - v_t = x + y$ with $\L x,y \R = 0$ satisfies $x_i y_i = 0$ for all $i$, and both $x$ and $y$ have a negative coordinate.  Without loss of generality, $x_g = -1$ and $y_0 = -1$.  Then $0 = \L y, \sigma \R \geq -1 + \sigma_i$ for some $i \in \{ t, m, g-1 \}$; but $\sigma_i \geq \sigma_t = t+1 > 1$, a contradiction.

It follows that $v_g = -e_g + e_{g-1} + e_{t-1} + \cdots + e_l$ for some $0 \leq l \leq t-1$.  Suppose by way of contradiction that $0 < l < t-1$.  Then $(v_l;v_i,v_{l+1},v_g)$ induces a claw in $G(S)$, where $i = l-1$ if $l > 1$, and $i = t$ if $l = 1$.  Therefore, $l \in \{0, t-1\}$, and the statement of the Lemma follows on consideration of $\L v_t, v_g \R$.

\end{proof}

Observe that if $v_t$ is breakable, $z_i \notin T_t$ for all $i < t$, and $g$ is chosen minimally as in Lemma \ref{l: breakable gappy}, then $S_{g-1}$ takes one of the forms stated in Proposition \ref{p: tight unbreakable}.  We assume henceforth that this is the case, and $g > t+1$ is chosen minimally with $z_g \in T_t$.

\begin{lem}\label{l: engulf}

Suppose that $T_t$ is breakable, $T_i \prec T_t$, and let $C = \{ v_t \} \; \cup \; \{ v_j \; | \; T_j \dagger T_i,T_t \}$.  Then $C$ separates $v_i$ in $G(S)$ from every other $v_l$ of norm $\geq 3$ for which $z_l \notin T_t$.

\end{lem}

\begin{proof}

For suppose the contrary, and choose an induced path $P$ in $G(S) - C$ with distinct endpoints $v_i, v_l$ such that $|v_l| \geq 3$ and every vector interior to $P$ has norm $2$.  Set $T = \bigcup_{v_k \in V(P)} T_k$.  Since $V(P) \cap C = \varnothing$ and $z_l \notin T_t$, it follows that $T_t \subset T$ and $T_t - T_j$ contains no vertex of degree $\geq 3$.  But then $T_t$ is  unbreakable, a contradiction.  

\end{proof}

\begin{prop}\label{p: breakable 1}

Suppose that $S_{g-1}$ is as in Proposition \ref{p: tight unbreakable}(1).  Then $s=2$, $n \geq g = 3$, and $S$ takes one of the following forms (up to truncation):

\begin{enumerate}

\item $|v_m| = m-1$ for some $m \geq 4$;

\item $v_4 = -e_4+e_3+e_0$ and $|v_m| = m-1$ for some $m \geq 5$; or

\end{enumerate}

\end{prop}

\begin{proof}

Lemma \ref{l: breakable gappy} implies that $v_g = -e_g+e_{g-1}+e_0$, $T_g \prec T_t$, and furthermore that $G(S) = \widehat{G}(S)$ in this case (see the end of the paragraph following Definition \ref{d: int graph}).

\noindent (a) {\em $g = s+1$, and $s=2$.}

\noindent If $g > s+1$, then $G(S_{g-1})$ is a path, and $v_g$ neighbors $v_1,v_s,$ and $v_{g-1}$, so $(v_1,\dots,v_g)$ spans a cycle in $G(S)$ missing the edge $(v_1,v_{g-1})$, in contradiction to Lemma \ref{l: triangle}.  If $s > 2$, then $(v_1;v_2,v_s,v_g)$ induces a claw.  

Let $h$ denote the maximum index of a vector $v_h$ for which $|v_h| \geq 3$ and $z_h \in T_1$.  

\noindent (b) {\em $v_h = -e_h + e_{h-1}+e_0$ and $h \in \{3,4\}$.}

\noindent The first statement follows from Lemma \ref{l: breakable gappy}, which also implies that $\epsilon_h = \epsilon_1 = \epsilon_g$.  If $h > 4$, then $\L v_g, v_h \R = 1$.  But both $v_g$ and $v_h$ are unbreakable, so $\L v_g, v_h \R = \epsilon_g \epsilon_h \L [T_g], [T_h] \R = \L [T_g], [T_h] \R \leq 0$, a contradiction.

\noindent (c) {\em $v_{m0}=1$ for all $m>h$.}

\noindent For suppose that $v_{m0} = 1$ for some $m > h$.  Then $v_{m1} = 1$ by Lemma \ref{l: breakable gappy} and the definition of $h$, which implies that $\L v_m, v_2 \R \geq 1$.  It follows that $T_1,T_2,T_m$ abut in pairs.  But this cannot occur, since $T_m \not\prec T_1$ by assumption, and $T_2$ and $T_m$ are both unbreakable.

\noindent (d) {\em $v_{m1}=0$ for all $m > h$.}

\noindent For suppose that $v_{m1} = 1$ for some $m > h$.  Thus, $T_m \dagger T_1$.  If $v_{m2} = 0$, then $(v_1,v_2,v_m)$ forms a negative triangle. If $v_{m2}=1$, then $T_m$ abuts $T_1$ at the same end as $T_3$, so $v_{m3} = 0$, and then $(v_1,v_3,v_m)$ forms a negative triangle.

It follows that $\min(\supp(v_m)) \geq 2$ for all $m > h$.  In particular, $\L v_t, v_m \R = 0$.  

\noindent (e) {\em There is no $m > h$ for which $v_m$ is gappy.}

\noindent Suppose by way of contradiction that $v_m$ is gappy for some smallest $m > h$, and choose a minimal gappy index $k$ for $v_m$.  Take $i = 3$ in Lemma \ref{l: engulf}.  Then $C = \{v_1\}$, $v_k \not\sim v_3$, and so $k > 2$.  If $h = 4$, then take $i=4$ in Lemma \ref{l: engulf}.  Then $C = \{v_1,v_2\}$, $v_k \not\sim v_h$, and so $k > 3$.  In any event, it follows that $k > h-1$, so $v_{k+1}$ is not gappy.  It follows as in the proof of Lemma \ref{l: one gappy} that $v_m \sim v_k, v_{k+1}$.  Now either $v_k \sim v_{k+1}$, in which case $(v_k,v_{k+1},v_m)$ forms a heavy triangle, or else $v_k \not\sim v_{k+1}$, and then the connectivity of $G(S_{m-1})$ implies that $G(S_m)$ contains an induced cycle of length $> 3$.  Either case results in a contradiction.

Thus, if $|v_m| \geq 3$ for some $m > h$, then Lemma \ref{l: engulf} implies that $v_m$ does not lie in the same component of $G(S) - \{v_1,v_2\}$ as $v_g$ or $v_h$.  It quickly follows that there is at most one index $m > h$ for which $|v_m| \geq 3$, and if so, then $v_m \sim v_2$.   It then follows that $S$ takes one of the forms stated in the Proposition.

\end{proof}

\begin{prop}\label{p: breakable 2}

Suppose that $S_{g-1}$ is as in Proposition \ref{p: tight unbreakable}(2).  Then $n \geq g = s+1$, and $S$ takes one of the following forms (up to truncation):

\begin{enumerate}

\item $s=2$, $v_4 = -e_4+e_3+e_0$, and $|v_m| = m-1$ for some $m \geq 5$;

\item $|v_m| = m-g+2$ for some $m \geq g$; or

\item $|v_m| = m-g+3$ for some $m \geq g$.

\end{enumerate}

\end{prop}

\begin{proof}

As in Proposition \ref{p: breakable 1}, Lemma \ref{l: breakable gappy} implies that $G(S) = \widehat{G}(S)$.  In particular, if $\L v_j,v_t \R = \pm 1$, then $T_j$ abuts $T_t$.  Observe that $v_1 \sim v_2, v_s$, but $v_2 \not\sim v_s$.  It follows that $T_1 \dagger T_2, T_s$, and $T_2$ and $T_s$ are distant.  If $g > s+1$, then $(v_1;v_2,v_s,v_g)$ induces a claw.  Hence $g = s+1$, $T_g \prec T_1$, and $T_g$ abuts $T_1$ at the same end as $T_s$.  

\noindent (I)  {\em Suppose that $v_{j0} = 1$ for some $j > g$.}

\noindent (a) {\em $v_{j1} = 0$.}

\noindent  If $v_{j1}=1$, then $T_j \dagger T_1$.  As $T_s$ and $T_g$ abut $T_1$ at the same end, it follows that either $v_j \not\sim v_s, v_g$, or else $v_j \sim v_s, v_g$ and $s = |v_s| = 2$.  Furthermore, $v_{ji} = 1$ for all $1 \leq i \leq s-1$, since otherwise $v_j \ll v_j - v_i$ for any such $i$ with $v_{ji} = 0$.  Now, if $v_j \not\sim v_g$, then $v_{js} = 0$, which implies that $\L v_s, v_j \R = s-1 \geq 1$ and $v_j \sim v_s$, a contradiction.    If instead $v_j \sim v_s$ and $s =2$, then $(v_1,v_2,v_j)$ forms a negative triangle.  Therefore, $v_{j1} = 0$.  

It follows that $T_j \prec T_1$, so $v_j = -e_j + e_{j-1} + e_0$ according to Lemma \ref{l: breakable gappy}.  Now, $T_j$ and $T_g$ abut $T_1$ at opposite ends, so $v_j \not\sim v_g$.  It follows that $j = g+1$.  Furthermore, $s = 2$, since otherwise $T_2$ abuts $T_1$ at the same end as $T_j$, but $v_j \not\sim v_2$.  In summary, $s =2$, $g=3$, $j = 4$, $v_4 = -e_4+e_3 + e_0$, $T_4 \prec T_1$, and $T_4$ abuts $T_1$ at the opposite end as do $T_2$ and $T_3$.  In particular, $v_{m0} = 0$ for all $m > 4$.

Now suppose that there exists $m > 4$ with $|v_m| \geq 3$.

\noindent (b) {\em $v_{m1}=0$, and $v_{m2} = v_{m3} = v_{m4}$.}

\noindent For suppose that $v_{m1}=1$.  Then $v_{m2}=1$ since $|v_2|=2$.  Thus, $v_m \sim v_1$ and $v_m \not\sim v_2$.  It follows that $T_m$ abuts $T_1$ at the same end as $T_4$.  Thus, $v_m \not\sim v_3$, so $v_{m3}=1$, and $v_m \sim v_4$, so $v_{m4}=0$.  But now $(v_1,v_4,v_m)$ forms a negative triangle, a contradiction.  Thus, $v_{m1} = 0$.  It follows that $v_m \not\sim v_1$, so $v_m \not\sim v_3, v_4$.  Thus, $v_{m4} = v_{m3} = v_{m2}$.

Let us further suppose that $m > 4$ is minimal subject to $|v_m| \geq 3$.   If $k := \min(\supp(v_m)) > 4$, then $(v_k;v_{k-1},v_{k+1},v_m)$ induces a claw.  Hence $k = 2$.  Since $|v_i|=2$ for $4 < i < m$, it follows that $v_m$ is not gappy, and $v_m = -e_m + e_{m-1} + \cdots + e_2$.

\noindent (c) {\em There does not exist  $m' > m$ such that $|v_{m'}| \geq 3$.}

For suppose otherwise, and choose $m'$ minimal with this property.  Thus, (b) implies that $v_{m'1} = 0$ and $v_{m'2} = v_{m'3} = v_{m'4}$.  If these values all equal $1$, then $\L v_m, v_{m'} \R \geq 2$, a contradiction.  Hence $k' := \min(\supp(v_{m'})) > 4$.  Now, Lemma \ref{l: engulf} implies that $k' \geq m$, taking $i = 4$ and $l = m'$ therein.  But now $|v_{k'+1}|=2$ and $v_{k'}$ has a smaller neighbor $v_i$, so $(v_{k'};v_i,v_{k'+1},v_{m'})$ induces a claw.  This is a contradiction.

In summary, (I) leads to case (1) of the Proposition.

\noindent (II)  {\em Suppose that $v_{m0} = 0$ for all $m > g$.}

\noindent (d) {\em If $s > 2$, then $v_{m1} = 0$ for all $m > g$.}

\noindent Assume the contrary, and choose $m$ accordingly.  It follows that $v_m \sim v_1$, and moreover that $T_m \dagger T_1$.  Since $v_s$ is unbreakable, it follows that $T_m$ abuts $T_1$ at the same end as $T_2$.  Thus, $v_m \not\sim v_s, v_g$.  From $v_m \not\sim v_s$ it follows that $v_{m2}= \cdots = v_{m,s-1} = 0$ and $v_{ms}=1$, and from $v_m \not\sim v_g$ it subsequently follows that $v_{m0}=0$ and $v_{mg} = 1$.  But then $(v_1,v_2,v_m)$ forms a negative triangle.

It follows that if $v_{m1} = 1$, then $s = 2$ and $T_m \dagger T_1$; otherwise $v_m \not\sim v_1$.  We henceforth drop any assumption about $s$.

\noindent (e) {\em $v_m$ is not gappy for any $m > g$.}

\noindent  For suppose some $v_m$ were, choose $m$ minimal with this property, and choose a minimal gappy index $k$ for $v_m$.   If $g = k+1$, then $v_m \sim v_g$.  If $v_{m1}=0$, then Lemma \ref{l: engulf} implies a contradiction with $i = g$ and $l = m$, while if $v_{m1}=1$, then $s=2$, $\L v_3, v_m \R \geq 0$ and $T_m \dagger T_1$ implies that $(v_1,v_3,v_m)$ is a negative triangle.  It follows in either case that $g \ne k+1$, and since $m$ is chosen minimal, it follows that $v_{k+1}$ is not gappy.  It follows at once that $v_{m,k-1} = 0$, whence $v_m \sim v_k, v_{k+1}$.  Furthermore, $k \ne 1$ since $|v_{k+1}| \geq 3$.  Since $k \geq 2$, both $v_k$ and $v_{k+1}$ are unbreakable.  Now, if $v_k \sim v_{k+1}$, then $|v_k| \geq 3$, and $(v_k,v_{k+1},v_m)$ forms a heavy triangle.  Hence $v_k \not\sim v_{k+1}$; but then a shortest path between them in $G(S_{k+1})$, together with $v_m$, results in an induced cycle of length $> 3$, a contradiction.  It follows that $v_m$ is not gappy.

\noindent (f) {\em $v_m$ does not have multiple smaller neighbors for any $m > g$.}

\noindent For suppose that $v_m \sim v_j$ for some $j > k := \min(\supp(v_m)) \geq 1$.  Note that $j \ne g$ because of the form $v_g$ takes, so $v_j$ is not gappy, and it follows that $j = k+2$ is uniquely determined.  In particular, it follows that $k > 1$. Hence $v_j \sim v_k$, since otherwise $G(S)$ contains an induced cycle of length $> 3$.  As $|v_j| \geq 3$, it follows that $|v_i| = 2$ for all $1 < i \leq k$, since otherwise $(v_i,v_j,v_m)$ forms a heavy triple for some such $i$.  Moreover, $|v_{k+1}| \geq 3$, since otherwise $(v_k;v_{k-1},v_{k+1},v_m)$ induces a claw.  It follows that $k = s-1$, but then $j = g$ and $v_j \not\sim v_m$.  Therefore, $v_m$ does not have multiple smaller neighbors.

Thus, $v_m \sim v_k$ and $v_m$ has no other smaller neighbor.  Furthermore, $|v_{k+1}| \geq 3$, as argued in the last paragraph.  If $|v_m| \geq 3$ for some smallest $m > g$, then $k \in \{s-1, s\}$, and $v_{m-1}$ lies in the same component of $G(S) - \{v_1,v_s\}$ as $v_g$.  Suppose by way of contradiction that there exists some smallest $m' > m$ for which $|v_{m'}| \geq 3$.  It follows from the foregoing that $\min(\supp(v_{m'})) + 1 = m$.  But then $v_{m'}$ lies in the same component of $G(S)-\{v_1,v_s\}$ as $v_g$, in contradiction to Lemma \ref{l: engulf}.

Therefore, $|v_m| \geq 3$ for at most one value $m > g$, and in this case, $v_m$ is not gappy, and $\min(\supp(v_m)) \in \{s-1,s\}$.  The two possibilities lead to cases (2) and (3), respectively.

\end{proof}

\begin{prop}\label{p: breakable 3}

Suppose that $S_{g-1}$ is as in Proposition \ref{p: tight unbreakable}(3).  Then $n \geq g = t+2$, $v_{t+2} = -e_{t+2}+e_{t+1}+e_{t-1}+\cdots+e_0$, and $S$ takes one of the following forms (up to truncation):

\begin{enumerate}

\item $|v_{t+1}| = 2$, $v_{t+3} = -e_{t+3}+e_{t+2}+e_{t-1}$, and $|v_m| = m-t$ for some $m \geq t+4$;

\item $|v_{t+1}| = 3$, $v_{t+3} = -e_{t+3}+e_{t+2}+e_{t-1}$, and $|v_m| = m-t$ for some $m \geq t+4$;

\item $|v_{t+1}|=2$ and $|v_m| = m-t$ for some $m \geq t+3$.

\item $|v_{t+1}|=3$ and $|v_m| = m-t$ for some $m \geq t+3$.

\end{enumerate}

\end{prop}

\begin{proof}  \noindent (a) {\em $v_g = -e_g + e_{g-1} + e_{t-1} + \cdots + e_0$ for some $g \geq t+2$.} 

\noindent By Lemma \ref{l: breakable gappy}, it follows that $v_g = -e_g+e_{g-1} + e_{t-1}$ or $-e_g + e_{g-1} + e_{t-1} + \cdots + e_0$, and $g \geq t+2$.  Let us rule out the first possibility.  Thus, assume by way of contradiction that this is the case.  It follows that $v_g \not\sim v_t$ in $G(S)$.  Note that $G(S_{g-1})$ is a path, and that $v_g \sim v_{t-1}$.  Suppose that $v_g \sim v_{g-1}$.  It follows that $v_{g-1} \sim v_{t-1}$, since otherwise $G(S)$ contains an induced cycle of length $> 3$.  However, since $S_{g-1}$ is built from $S_t$ by a sequence of expansions, it follows that $t-1 = \min(\supp(v_{g-1}))$.  However, this implies that $v_g \not\sim v_{g-1}$, a contradiction.  Hence $v_g \not\sim v_{g-1}$.  But then $(v_{t-1};v_i,v_{g-1},v_g)$ induces a claw, where $i = t $ if $t = 2$, and $i = t-2$ if $t > 2$.  This contradiction shows that $v_g = -e_g + e_{g-1} + e_{t-1} + \cdots + e_0$, as desired.

\noindent (b) {\em $g = t+2$.}

\noindent Observe that $|v_{t+1}| \in \{2,3\}$ since $S_{t+1}$ is an expansion on $S_t$.  If $|v_{t+1}| = 2$, then $v_{t+1} \sim v_g$ since otherwise $(v_t;v_1,v_{n+1},v_g)$ induces a claw.  If $|v_{t+1}|=3$, then $v_{t+1} \not\sim v_g$, since otherwise $(v_g,v_t,v_1,\cdots,v_{t-1})$ induces a cycle of length $>3$ in $G(S)$.  It follows in either case that $g = t+2$, as desired.

\noindent (c) {\em If $v_h = -e_h + e_{h-1} + e_{t-1} + \cdots + e_0$, then $h = g$.}

\noindent For if $h \ne g$, then $T_g, T_h \prec T_t$ and $T_g$ and $T_h$ are distant, and $(v_t;v_1,v_g,v_h)$ induces a claw.

\noindent (d) {\em If $v_h = -e_h + e_{h-1} + e_{t-1}$, then $h = g+1 = t+3$.}

\noindent For if $h > g+1$, then $(v_h,v_g,v_t,v_1,\dots,v_{t-1})$ spans a cycle in $G(S)$ that is missing the edge $(v_h,v_t)$.

Henceforth we write $h = g+1$ if $v_{g+1}$ takes the form in article (d), and $h = g$ otherwise.
It follows from Lemma \ref{l: breakable gappy} and articles (c) and (d) that if $|v_m| \geq 3$ for some $m > h$, then $z_m \notin T_t$.  In particular, $v_m \sim v_i$ iff $T_m$ and $T_i$ abut for all $m > h$.

\noindent (e) {\em If $m > h$, then $v_m$ is not gappy.}

\noindent For suppose that $v_m$ were gappy for some minimal $m > h$, let $k = \min(\supp(v_m))$, and let $j$ denote a minimal gappy index for $v_m$.  Then $k > 0$, since otherwise $|v_{j+1}| \geq 3$ implies that $\L v_{j+1},v_g \R \geq 2$, and then $t = j+1$ and $z_g \in T_t$, a contradiction.  Now $\L v_m,v_k \R = -1$ and $\L v_m, v_{j+1} \R =1$, so $v_m \sim v_k, v_{j+1}$, and since $G(S_{m-1})$ is connected, it follows that $v_k \sim v_{j+1}$.  Furthermore, $T_g \dagger T_{j+1}$ since $|v_g|, |v_{j+1}| \geq 3$, and since $(v_k,v_{l+1},v_g)$ is a positive triangle, it follows that $\L v_k,v_{l+1} \R = -1$.  Thus, $v_{l+1,k} = 1$, and since $\L v_{l+1},v_g \R \leq 1$, it follows that $l = k$.  Now, $v_{k,k-1} \ne 0$, so it follows that $v_{k+1,k-1} = 0$.  Consequently, $v_{k+1}$ is gappy.  Since $m > h$ was chosen minimal, it follows that $k+1 \in \{g,h\}$.  However, the only way that this can occur and satisfy $\L v_k, v_{k+1} \R = -1$ is if $k+1 = g = t+2$ and $|v_{t+1}|=2$.  However, in this case, $(v_t,v_{t+1},v_m,v_{t+2})$ spans a cycle that is missing the edge $(v_t,v_m)$, a contradiction.  It follows that no such $m$ exists, as desired.

Thus, $z_m \notin T_t$.

\noindent (f) {\em $\min(\supp(v_m)) = t+1$ or $\geq t+3$ for all $m > h$.}

\noindent Let $k = \min(\supp(v_m))$.  Since $\L v_g, v_m \R \leq 1$, it follows that $k \geq t-1$.  Lemma \ref{l: engulf} with $i = g$ and $l = m$ implies that $k \notin \{t-1,t+2\}$.  Finally, $k \ne t$, since otherwise $(v_t;v_1,v_g,v_m)$ induces a claw.

Suppose that there exists a minimal $m > h$ such that $|v_m| \geq 3$.

\noindent (g) {\em $k =t+1$.}

\noindent For if $k \geq t+3$, then $v_k$ has a smaller neighbor $v_i$, and $(v_k;v_i,v_{k+1},v_m)$ induces a claw.

\noindent (h) {\em There does not exist $m' > m$ for which $|v_{m'}| \geq 3$.}

\noindent Suppose otherwise, and let $k' = \min(\supp(v_{m'}))$.  If $k' = t+1$, then either $(v_{t+1},v_m,v_{m'})$ or $(v_g,v_m,v_{m'})$ forms a heavy triple, depending on $|v_{t+1}| \in \{2,3\}$.  From (f) it follows that $k' \geq t+3$.  Thus, $k'+1 = m$, since otherwise $|v_{k'+1}|=2$, $v_{k'}$ has a smaller neighbor $v_i$, and $(v_{k'},v_i,v_{k'+1},v_{m'})$ induces a claw.  Thus, $(v_h,\dots,v_{m-1},v_{m'})$ induces a path.  If $h = t+2$, then we obtain a contradiction to Lemma \ref{l: engulf} with $i = g$ and $l = m'$.  If $h = t+3$, then we obtain a similar contradiction with a bit more work.  Specifically, considering the path $(v_t,v_1,\dots,v_{t-1},v_{t+3},\dots,v_{m-1},v_{m'})$, it follows that the interval $T_{t+3} \cup \bigcup_{i=1}^{t-1}T_i$ contains one endpoint of $T_t$ and the interval $T_{m'} \cup \bigcup_{i=t+3}^{m-1} T_i$ contains the other.  As $z_{m'} \notin T_t$, it follows that $z_{t+3}$ is the unique vertex of degree $\geq 3$ in $T_t$, a contradiction, since $z_{t+2} \in T_t$ as well.

The four cases stated in the Proposition now follow from the possibilities $|v_{t+1}| \in \{2,3\}$ and $h \in \{t+2,t+3\}$.

\end{proof}


\section{Producing the Berge types}\label{s: cont fracs}

The goal of this section is to show how the families of linear lattices enumerated in the structural Propositions of Sections \ref{s: just right}, \ref{s: gappy, no tight}, and \ref{s: tight} give rise to the homology classes of Berge knots tabulated in Subection \ref{ss: list}. 
Subsection \ref{ss: methodology} describes the methodology and Table \ref{table: A} collects the results.  Subsection \ref{ss: cont fracs} contains the necessary background material about continued fractions, and Subsections \ref{ss: large} and \ref{ss: small} carry out the details.

\subsection{Methodology}\label{ss: methodology}

Given a standard basis $S$ expressed in one of the structural Propositions, we show that the changemaker lattice it spans is isomorphic to a linear lattice $\Lambda(p,q)$ by converting $S$ into a vertex basis $B = \{x_1,\dots,x_n\}$ for it.  Letting $\nu$ denote the sequence of norms $(|x_1|,\dots,|x_n|)$, we recover $p$ as the numerator $N[\nu]^- = N[|x_1|,\dots,|x_n|]^-$ of the continued fraction.  We recover the value $k$ of Proposition \ref{p: homology}, and hence $q \equiv - k^2 \pmod p$, in the following way.  Let $B^\star$ denote the elements in $B$ that pair non-trivially with $e_0$, let $\nu_i = (|x_1|,\dots,|x_i|)$, and let $p_i = N[\nu_i]^-$.  Then 
\begin{equation}\label{e: k}
k = \sum_{x_i \in B^\star} p_{i-1} \L x_i, e_0 \R
\end{equation}
according to \eqref{e: x}, Proposition \ref{p: homology}, and Lemma \ref{l: cont frac basics}(1).  In practice, $B^\star$ contains at most three elements, and each value $\L x_i, e_0 \R$ is typically $\pm 1$ (in case of Proposition \ref{p: tight unbreakable}, it can equal $\pm 2$).

\medskip

\noindent {\em Example.}  As an illustrative example, consider a standard basis $S$ as in Proposition \ref{p: just right 1}(1).  By inspection, $\widehat{G}(S)$ is nearly a path, which suggests that $S$ is not far off from a vertex basis.  Indeed, a little manipulation shows that 
\[ B = \{-v_s^\star,-v_{s+2},v_{s+3},v_{s-1},\dots,v_1^\star, -(v_{s+1}+v_{s-1}+\cdots+v_1)^\star\} \]
is a vertex basis for the lattice spanned by $S$.  The elements denoted by a star ($\star$) belong to $B^\star$.  From $B$ we obtain the sequence of norms
\[ \nu = (s+1,3,5,2^{[s-1]},3), \]
using $2^{[t]}$ as a shorthand for a sequence of $t$ 2's. In order to determine $p$, we calculate
\[p = N[s+1,3,5,2^{[s-1]},3]^- = N[s+1,3,4,-s,2]^- = N[s+1,-3,4,s,-2]^+ = 22s^2+31s+11,\]
using Lemma \ref{l: cont frac basics 2}(1) for the second equality and Mathematica \cite{mathematica} for the last one.  In order to determine $k$, we consider the substrings
\[ \nu_0 = \varnothing, \quad \nu_{s+1} = (s+1,3,5,2^{[s-2]}), \quad \nu_{s+2} = (s+1,3,5,2^{[s-1]}). \]
Weight the numerator of each $[\nu_{i-1}]^-$ by $\L x_i, e_0 \R$, which equals the sign $\pm 1$ appearing on the leading term in the starred expression, and add them up to obtain the value $k$.  Thus,
\begin{eqnarray*}
k &=& -N[\varnothing]^- + N[s+1,3,5,2^{[s-2]}]^- - N[s+1,3,5,2^{[s-1]}]^- \\
&=& -1 + N[s+1,3,4,-(s-1)]^- - N[s+1,3,4,-s]^- \\
&=& -1 + N[s+1,-3,4,s-1]^+ - N[s+1,-3,4,s]^+ \\
&=& -1 + (11s^2-s-5) - (11s^2+10s+2) \\
&=& 11(-s-1)+3.
\end{eqnarray*}
Since $s \geq 2$ and $p = (2k^2+k+1)/11$, it follows that the standard bases of \ref{p: just right 1}(1) correspond to Berge type X with $k \leq 11(-3)+3$.  The result of this example appears in Table \ref{table: A} and as the first entry in Table \ref{table: 1}.

\medskip

In this manner we extract a linear lattice, described by the pair $(p,k)$, from each standard basis expressed in the structural Propositions.  In the process, we show that these values account for precisely the pairs $(p,k)$ tabulated in Subsection \ref{ss: list}.  Table \ref{table: A} displays the results.  Note that Proposition \ref{p: gappy structure}(2) gets reported in terms of its constituents,  Lemmas \ref{l: S_g nothing}(1,3,4,5) and \ref{l: S_g not nothing}.

\begin{table}[h]
\caption{Structural Propositions sorted by Berge type.}
\label{table: A}

\begin{center}

{\small

\begin{tabular}{l l l l}

{\bf I$_+$} \ref{p: just right 3}(1,3) &
{\bf I$_-$} \ref{p: tight unbreakable}(1,3) &
{\bf II$_+$}  \ref{p: tight unbreakable}(2) &
{\bf II$_-$} \ref{p: just right 1}(3) \\

\cr

{\bf III}(a)$_+$ \ref{p: breakable 2}(1), \ref{p: breakable 3}(1) &
(a)$_-$ \ref{p: just right 3}(2), \ref{l: S_g not nothing}(1,2), \ref{p: gappy structure}(3) &
(b)$_+$ \ref{p: breakable 1}(1), \ref{p: breakable 3}(4) &
(b)$_-$ \ref{l: S_g not nothing}(3) \\

\cr

{\bf IV}(a)$_+$ \ref{p: breakable 1}(2), \ref{p: breakable 3}(2) &
(a)$_-$ \ref{l: S_g nothing}(3) &
(b)$_+$ \ref{p: breakable 3}(3) &
(b)$_-$ \ref{p: just right 2}(5), \ref{l: S_g nothing}(1), \ref{p: gappy structure}(3) \\

\cr

{\bf V}(a)$_+$ \ref{p: breakable 2}(2) &
(a)$_-$ \ref{p: just right 2}(5), \ref{p: just right 3}(2), \ref{l: S_g nothing}(4) &
(b)$_+$ \ref{p: breakable 2}(3) &
(b)$_-$ \ref{l: S_g nothing}(5) \\

\cr

{\bf VII} \ref{p: gappy structure}(1) &
{\bf VIII} \ref{p: berge viii} &
{\bf IX} \ref{p: just right 1}(2), \ref{p: just right 2}(2,3) &
{\bf X} \ref{p: just right 1}(1), \ref{p: just right 2}(1,4) 

\end{tabular}

} 

\end{center}

\end{table}

We mention one caveat, which amounts to the overlap between different Berge types.  For example, in type III(a)$_+$, \ref{p: breakable 2}(1) and \ref{p: breakable 3}(1) account for the cases that $d =2$, $(k+1)/d \geq 5$ and $d=3$, $(k+1)/d \geq 3$, respectively (Table \ref{table: 2}).  What happens when $(d, (k+1)/d) \in \{ (2,3), (1, *), (*,1) \}$?  For $(2,3)$, notice that we obtain the same family of examples by setting $d = 3, (k+1)/d = 2$ in type V(a)$_+$.  This is covered by \ref{p: breakable 2}(2).  Moreover, \ref{p: breakable 2}(2) fills out most of V(a)$_+$, while only this sliver of it applies to III(a)$_+$.  For that reason, we only report \ref{p: breakable 2}(2) next to V(a)$_+$ in Table \ref{table: A}.  Similarly, the cases of $(1,*)$ and $(*,1)$ correspond to II$_-$ with $i=2$ and I$_-$ with $i=1$, respectively.  In general, it is not difficult to identify the overlaps of this sort and use Table \ref{table: A} to obtain a complete correspondence between structural Propositions and Berge types.  In a few places the overlap is explicit: \ref{p: just right 2}(5), \ref{p: just right 3}(2), and \ref{p: gappy structure}(3) each appear twice in Table \ref{table: A}.

The correspondence between structural Propositions and Berge types exhibits some interesting features.  For example, amongst the ``small" families (defined just below), and excluding the special cases of \ref{p: just right 2}(5) and \ref{p: just right 3}(1),
\begin{itemize}

\item all elements of $S$ are just right iff $L$ is an exceptional type (IX or X);

\item $S$ has a gappy vector but no tight one iff $L$ is of $-$ type;

\item $S$ has a tight vector iff $L$ is of $+$ type.

\end{itemize}
It would be interesting to examine the geometric significance of this correspondence.

In determining the values $(p,k)$ from the structural Propositions, it is useful to partition these families into two broad classes: {\em large families}, those that involve a sequence of expansions, and {\em small families}, those that do not.  The large families (along with \ref{p: just right 3}(1)) correspond to Berge types I, II, VII, and VIII in Table \ref{table: A}. Determining the relevant values $(p,k)$ for these families occupies Subsection \ref{ss: large}.  The small families, while more numerous, are considerably simpler to address. We take them up in Subsection \ref{ss: small}.  Excluding \ref{p: just right 3}(1), they correspond to Berge types III, IV, V, IX, and X in Table \ref{table: A}.

Lastly, we remark that the determination of the isomorphism types of the sums of linear lattices enumerated in Proposition \ref{p: decomposable structure} follows as well, and involves far fewer cases.  As it turns out, they correspond precisely to the sums of lens spaces that arise by surgery along a torus knot or a cable thereof.  For example, Proposition \ref{p: expansion} enumerates the sums of linear lattices spanned by standard bases built from $\varnothing$ by a sequence of expansions.  They correspond to the connected sums $- (L(p,q) \# L(q,p))$ that result from $pq$-surgery along the positive $(p,q)$-torus knots.  In fact, \cite[Theorem 1.5]{greene:cabling} asserts a much stronger conclusion: if surgery along a knot produces a connected sum of lens spaces, then it is either a torus knot or a cable thereof.  We refer to \cite{greene:cabling} for further details.


\subsection{Minding $p$'s and $q$'s.}\label{ss: cont fracs}

Given a basis $C = \{v_1,\dots,v_n\}$ built from $\varnothing$ by a sequence of expansions, augment $C$ by a vector $v'_{n+1} := \sum_{i=k}^n e_i$, where $k = 0$ if $|v_i| = 2$ for all $v_i \in C$; $k = n-1$ if $|v_n| \geq 3$; and $k$ is the maximum index of a vector in $C$ with norm $\geq 3$ otherwise.  Observe that $C':= C \cup \{v'_{n+1}\}$ spans a lattice isomorphic to a sum of two non-zero linear lattices for which $C'$ is a vertex basis.  More precisely, partition $C' = \{v_{i_1},\dots,v_{i_l}\} \cup \{v_{j_1},\dots,v_{j_m}\}$ into vertex bases for the two summands, where $i_1 > \cdots > i_l$ and $n+1=j_1 > \cdots > j_m$, and write $(a_1,\dots,a_l) = (|v_{i_1}|,\dots,|v_{i_l}|)$ and $(b_1,\dots,b_m) = (|v_{j_1}|,\dots,|v_{j_m}|)$.  Then $\L C' \R \cong \Lambda(p,q) \oplus \Lambda(p',q')$, where $p/q = [a_1,\dots,a_l]^-$ and $p'/q' = [b_1,\dots,b_m]^-$.  The following result sharpens this statement.

\begin{lem}\label{l: aug exp}

The lattice spanned by $C$ is isomorphic to $\Lambda(p,q) \oplus \Lambda(p,p-q)$ for some $p > q > 0$.

\end{lem}

Note that Lemma \ref{l: aug exp} implies a relationship between the Hirzebruch-Jung continued fraction expansions of $p/q$ and $p/(p-q)$. This is nicely expressed by the Riemenschneider point rule (see the German original, \cite[pp. 222-223]{riemenschneider}; or \cite[pp. 2158-2159]{lisca:lens2}).

\begin{proof}

We proceed by induction on $n=|C|$.  When $n=1$, we have $C' = \{e_0-e_1,e_0+e_1\}$, and $C'$ spans a lattice isomorphic to $\Lambda(2,1) \oplus \Lambda(2,1)$, from which the Lemma follows with $p = 2$ and $q = 1$.

For $n > 1$, observe that $C'$ is constructed from $C'_{n-1}$ by either setting $v_n = v'_n-e_n$ and $v'_{n+1} = e_{n-1}+e_n$, or else $v_n = e_{n-1}-e_n$ and $v'_{n+1} = v'_n + e_n$. 
By induction, $C'_{n-1}$ determines two strings of integers $(a_1,\dots,a_l)$ and $(b_1,\dots,b_m)$, and $\L C'_{n-1} \R \cong \Lambda(p,q) \oplus \Lambda(p,p-q)$, where $p/q = [a_1,\dots,a_l]^-$ and $p/(p-q) = [b_1,\dots,b_m]^-$.  Swapping the roles of $q$ and $p-q$ if necessary, $C'$ determines the strings $(2,a_1,\dots,a_l)$ and $(b_1+1,\dots,b_m)$, for which we calculate 
$[2,a_1,\dots,a_l]^- = 2 - 1/(p/(p-q)) = (p+q)/p$ and $[b_1+1,\dots,b_m]^- = 1 + p/(p-(p-q)) = (p+q)/p$.  Therefore, $\L C' \R = \Lambda(p+q,q) \oplus \Lambda(p+q,p)$, which takes the desired form and completes the induction step. 

\end{proof}

\begin{prop}\label{p: expansion}

Suppose that $C$ is built from $\varnothing$ by a sequence of expansions.  If $|v_i|=2$ for all $v_i \in C$, then $L \cong \Lambda(n+1,n)$.  Otherwise, $L \cong \Lambda(p,q) \oplus \Lambda(r,s)$ for some $p > q > 0$, where $r = p-q$ and $s$ denotes the least positive residue of $-p \pmod r$.

\end{prop}

\begin{proof}

Augment $C$ to $C'$ as above and write $\L C' \R \cong \Lambda(p,q) \oplus \Lambda(p,p-q)$ according to Lemma \ref{l: aug exp}.  Then $C$ determines the strings $(a_1,\dots,a_l)$ and $(b_2,\dots,b_m)$, where the second string is empty in case $m = 1$.  We have $b_1 = \lceil p/(p-q) \rceil$, from which it easily follows that $[b_2,\dots,b_m]^- =r/s$ when $m > 1$, with the values $r$ and $s$ as above.  Thus, $L$ takes the desired form in this case.  Furthermore, when $m=1$, it follows easily that $L \cong \Lambda(n+1,n)$.  This establishes the Proposition.

\end{proof}

\begin{defin}\label{d: cont frac}

Given integers $a_1,\dots,a_l \geq 2$, write $p_j / q_j = [a_1,\dots,a_j]^-$, $r_j = p_j - q_j$, and $p_0 = 1$, and define integers $b_1,\dots,b_m \geq 2$ by $[b_1,\dots,b_m]^- = p_l / r_l$.

\end{defin}

\noindent Thus, Definition \ref{d: cont frac} relates the strings $(a_1,\dots,a_l)$ and $(b_1,\dots,b_m)$ preceding Lemma \ref{l: aug exp}.

\begin{lem}\label{l: cont frac basics}

Given integers $a_1,\dots,a_n \geq 2$ and an indeterminate $x$, the following hold:

\begin{enumerate}

\item $p_j = p_{j-1} a_j - p_{j-2}$ and $q_j = q_{j-1} a_j - q_{j-2}$;

\item $[a_1,\dots,a_n,x]^- = (p_n x-p_{n-1})/(q_n x - q_{n-1})$;

\item $[a_j,\dots,a_1]^- = p_j / p_{j-1}$;

\item $p_{j-1}$ is the least positive residue of $q_j^{-1} \pmod{p_j}$;

\item $q_{j-1}$ is the least positive residue of $-p_j^{-1} \pmod{q_j}$;

\item $r_{j-1}$ is the least positive residue of $p_j^{-1} \equiv q_j^{-1} \pmod{r_j}$.

\end{enumerate}

\end{lem}

\begin{proof}[Proof sketch.]

Item (1) follows by induction on $k$, using the identity \[[a_1,\dots,a_j,a_{j+1}]^- = [a_1,\dots,a_j - 1/a_{j+1}]^-.\]  Item (2) follows at once from (1).  Item (3) follows from $[a_{j+1},\dots,a_1]^- = a_{j+1}-q_j/p_j$ and (1).  The identity \[p_{j-1} q_j - p_j q_{j-1} = 1\] follows from (1) and induction; the inequalities $0 < p_{j-1} < p_j$ and $0 < q_{j-1} < q_j$ follow from the fact that $a_j \geq 2$; and items (4) and (5) follow from these observations.  From the preceding identity we obtain \[ p_j r_{j-1} - p_{j-1} r_j  = 1 \quad \text{and} \quad q_j r_{j-1}-  q_{j-1} r_j = 1,\] and (1) implies that $0 < r_{j-1} < r_j$.  Item (6) now follows as well.

\end{proof}

We collect a few more useful facts whose routine proofs follow from Lemma \ref{l: cont frac basics}.  Following Lisca, we use the shorthand \[(\dots,2^{[t]},\dots) := (\dots,\underbrace{2,\dots,2}_t,\dots).\] 

\begin{lem}\label{l: cont frac basics 2} The following identities hold:

\begin{enumerate}

\item $[\dots,b+1,2^{[a-1]},c+1,\dots]^- = [\dots,b,-a,c,\dots]^-$;

\item $[2^{[a-1]},b+1,\dots]^- = p/q \implies [-a,b,\dots]^- = -p/(p-q)$;

\item $[\dots,b+1,2^{[a-1]}]^- = [\dots,b,-a]^-$;

\item $[b_m,\dots,b_2]^- = r_l/(r_l -  r_{l-1})$;

\item $[a_1,\dots,a_l,t+1,b_m,\dots,b_2]^- = (p_l r_l t+1)/(q_l r_l t + 1)$;

\item $[a_1,\dots,a_l+1,2^{[t-2]},b_m+1,\dots,b_2]^- = (p_l r_l t - 1)/(q_l r_l t -1)$;

\item $[a_1,\dots,a_l,b_m,\dots,b_1]^- = (p_l^2 -p_l p_{l-1} + p_{l-1}^2)/(p_l q_l -p_l q_{l-1} + p_{l-1} q_{l-1})$;

\item $[a_1,\dots,a_l,b_m,\dots,b_2]^- = (p_l r_l - p_{l-1} r_l + p_{l-1} r_{l-1})/(q_l r_l -q_{l-1} r_l + q_{l-1} r_{l-1})$;

\item $[a_1,\dots,a_l+b_m+1,\dots,b_1]^- = (p_l^2 + p_l p_{l-1} - p_{l-1}^2)/(q_l p_l + q_{l-1} p_l - q_{l-1} p_{l-1} - 1)$;

\item $[a_1,\dots,a_l+b_m+1,\dots,b_2]^- = (p_l r_l + p_{l-1} r_l - p_{l-1} r_{l-1} -1)/(q_l r_l + q_{l-1} r_l - q_{l-1} r_{l-1} - 1)$.

\qed

\end{enumerate}

\end{lem}


\subsection{Large families}\label{ss: large}  

For each standard basis $S$ occurring in a large family, we alter at most one $v_i \in S$ to another $\overline{v}_i$ such that $S - \{v_i\} \cup \{ \overline{v}_i \}$ is a vertex basis, up to reordering and negating some elements.  In each case, there exists a unique partition $\{1,\dots,n\} = \{i_1,\dots,i_\lambda \} \cup \{j_1,\dots,j_\mu \}$ with the following properties:

\begin{itemize}

\item $i_1 = n$, and in the case of Propositions \ref{l: S_g nothing}(2) and \ref{p: berge viii}, $j_1 = n-1$;

\item $i_1 > \cdots > i_\lambda$ and $j_1 > \cdots > j_\mu$;

\item $\{i_1,j_1\} = \left\{1, \min\{ j > 1 \; | \; |v_j| > 2 \} \right\}$;

\item the subgraphs of $\widehat{G}(S - \{v_i\} \cup \{ \overline{v}_i \})$ induced on $\{v_{i_1},\dots,v_{i_\lambda }\}$ and $\{v_{j_1},\dots,v_{j_\mu}\}$ are paths with vertices appearing in consecutive order (replacing $v_i$ by $\overline{v}_i$).

\end{itemize}

For the first five families below, we modify $S$ to a related subset $C$ built from $\varnothing$ by a sequence of expansions.  We obtain a pair of strings $(a_1,\dots,a_l)$, $(b_1,\dots,b_m)$ from $C'$, and we express the sequence of norms $\nu$ in terms of them.  The values $l$ and $m$ are related to the values $\lambda$ and $\mu$. To determine $p$ and $k$ from $\nu$, we apply Lemmas \ref{l: cont frac basics} and \ref{l: cont frac basics 2}.  Frequently it is easier to recover $k'$ instead of $k$ by reversing the order of the basis.

\smallskip

{\bf Proposition \ref{p: just right 3}(3).}  

\noindent $B = \{v_{i_1},\dots,v_{i_\lambda},v_{j_\mu},\dots,v_{j_1} \}$ and $B^\star = \{ v_1 \}$;

\noindent $C = S - \{v_1\} \subset \text{span}\L e_1,\dots,e_n \R$;

\noindent $\nu = (a_1,\dots,a_l,2,b_m,\dots,b_2)$;

\noindent $p = p_l r_l+1$ by \ref{l: cont frac basics 2}(5) with $t=1$;

\noindent $k = p_l$ if $j_\mu=1$; $k' = r_l$ if $i_\lambda = 1$.

\noindent Note that $p_l \geq 3$ and $r_l \geq 2$.  With $\{i,k\} = \{p_l,r_l\}$, it follows that $p = ik +1$ with $i,k \geq 2$ and $\gcd(i,k)=1$.   The values $i$ and $k$ are unconstrained besides these conditions.  In summary, \ref{p: just right 3}(3) accounts for Berge type I$_+$ with $i,k \geq 2$.

{\bf Proposition \ref{p: tight unbreakable}(1).}

\noindent $B = \{ -v_{i_1},\dots,-v_{i_\lambda}^\star,v_{j_\mu}^\star,\dots,v_{j_1} \}$;

\noindent $C = S \cup \{v',v_t'\} - \{v_t\} \subset \text{span} \L e', e_0, \dots, e_n \R$, where $v' = -e_0+e'$ and $v_t' = v_t + v'$;

\noindent $\nu = (a_1,\dots,a_l+b_m,\dots,b_2)$;

\noindent $p = p_l r_l - 1$ by \ref{l: cont frac basics 2}(5) with $t=-1$;

\noindent $k = p_l$ if $j_\mu = 1$, using $a_l = 2$ and \ref{l: cont frac basics}(1); $k' = r_l$ if $i_\lambda = 1$ in the same way.

\noindent With $\{i,k\} = \{p_l,r_l\}$, it follows that $p = ik -1$ with $i,k \geq 3$ and $\gcd(i,k) = 1$.  The values $i$ and $k$ obey two further constraints coming from $\max\{a_l, b_m\} = 3$ and $\max\{a_{l-1}, \dots, a_1, b_{m-1}, \dots, b_2\} \geq s+1 \geq 3$.  In summary, \ref{p: tight unbreakable}(1) accounts for part of Berge type I$_-$ with $i,k \geq 3$.

{\bf Proposition \ref{p: tight unbreakable}(3).}

\noindent  The argument is identical to the case of \ref{p: tight unbreakable}(1), switching the conclusions in the case of $i_\lambda = 1$ and $j_\mu = 1$.  Now we have the constraint that $\max\{a_l, b_m\} = t+2 \geq 4$.  In summary, \ref{p: tight unbreakable}(3) accounts for another part of Berge type I$_-$ with $i,k \geq 3$.

{\bf Proposition \ref{p: tight unbreakable}(2).}

\noindent $B = \{ v_{i_1},\dots,v_{i_\lambda},v_{j_\mu},\dots,v_{j_1} \}$ and $B^\star = \{ v_1 \}$;

\noindent $C = S - \{v_1\} \subset \text{span} \L e_1, \dots, e_n \R$;

\noindent $\nu = (a_1,\dots,a_l,5,b_m,\dots,b_2)$;

\noindent $p = 4p_l r_l +1$ by \ref{l: cont frac basics 2}(5) with $t=4$;

\noindent $k = 2p_l$ if $j_\mu = 1$; $k' = 2r_l$ if $i_\lambda = 1$.

\noindent With $\{i,k\} = \{2p_l,2r_l\}$, it follows that $p = ik+1$ with $i,k \geq 2$ and $\gcd(i,k) = 2$.  
The values $i,k$ are unconstrained besides these conditions.  However, the case $\min\{i,k\}=2$ (which occurs when $s=2$) accounts for Berge I$_-$ with $\min\{i,k\}=2$.  In summary, \ref{p: tight unbreakable}(2) accounts for Berge type II$_+$ and this special case of Berge I$_-$.

{\bf Proposition \ref{p: just right 1}(3).}

\noindent $B = \{ -v_{i_1},\dots,-\overline{v}_{i_{\lambda-2}}^{\; \star},v_{i_{\lambda - 1}},v_{i_\lambda}^\star,v_{j_\mu},\dots,v_{j_1} \}$ if $i_\lambda = 1$, in which case $i_{\lambda-2} = m$, 
$\overline{v}_{i_{\lambda-2}} = v_m + v_1 + v_2$, $i_{\lambda-1} = 2$, and $j_\mu = 3$;

\noindent $C = S \cup \{v_3',v_m'\} - \{v_1,v_2,v_3,v_m\} \subset \text{span} \L e_2,\dots,e_n \R$, where $v_3' = v_3 - e_1$ and $v_m' = v_m -e_1$;

\noindent $\nu = (a_1,\dots,a_l+1,2,2,b_m+1,\dots,b_2)$;

\noindent $p = 4p_l r_l -1$ by \ref{l: cont frac basics 2}(6) with $t=4$;

\noindent $k = 2p_l$ if $i_\lambda = 1$; $k' = 2r_l$ if $j_\mu = 1$.

\noindent With $\{i,k\} = \{2p_l,2r_l\}$, it follows that $p = ik - 1$ where $i,k \geq 4$ and $\gcd(i,k) = 2$.  A similar argument applies in case $j_\mu = 1$. In summary, \ref{p: just right 1}(3) accounts for Berge type II$_-$.

\smallskip

For the two remaining large families, we modify $S$ directly into a subset $C'$ as in Subsection \ref{ss: cont fracs}.  Let $(a_1',\dots,a_l')$ and $(b_1',\dots,b_m')$ denote its corresponding strings, and let $(a_1,\dots,a_l)$ and $(b_1,\dots,b_m)$ denote their reversals ($a_i = a_{l+1-i}'$ and $b_j = b_{m+1-j}'$).  This notational hiccup results in cleaner expressions for $p$ and $k$.  Note that these values are still related in the manner of Definition \ref{d: cont frac}, so Lemma \ref{l: cont frac basics 2} applies.

\smallskip

{\bf Proposition \ref{p: gappy structure}(1).}  Here $n = g$.

\noindent $B = \{ -v_{i_\lambda}^\star,\dots,-v_{i_1},\overline{v}_{j_1},\dots,v_{j_\mu}^\star \}$ if $i_\lambda = 1$, where $j_1 = n$ and $\overline{v}_n = \overline{v}_g = v_g + v_{g-1} + \cdots + v_{k+2}$;\footnote{Apology: this is the $k$ of Lemma \ref{l: S_g nothing}, not the homology class of $K$.}

\noindent $C' = S \cup \{v_n'\} - \{v_n\} \subset \text{span} \L e_0,\dots,e_{n-1} \R$, where $v_n' = \overline{v}_n + e_n + e_{n-1}$;

\noindent $\nu = (a_1,\dots,a_l,b_m,\dots,b_1)$;

\noindent $p = p_l^2 -p_l p_{l-1} + p_{l-1}^2$ by \ref{l: cont frac basics 2}(7);

\noindent $k = p_l r_l - p_{l-1} r_l + p_{l-1} r_{l-1} - 1$ by \ref{l: cont frac basics 2}(8); also, observe that the difference between the numerator and denominator in $[\nu]^-$ is $D := r_l^2 - r_l r_{l-1} + r_{l-1}^2$.

\noindent Now we use the identity $(a^2-ab+b^2)(c^2-cd+d^2) = (e^2-ef+f^2)$, where $e= ac-bc+bd$ and $f = ad-bc$.  We apply this identity with $a = p_l, b = p_{l-1}, c = r_l$, and $d = r_{l-1}$, noting that $f = 1$.  It follows that $p \cdot D = (k+1)^2 - (k+1) + 1 = k^2+k+1$.  Up to renaming variables, the same argument applies in case $j_\mu = 1$.  In summary, \ref{l: S_g nothing}(2) accounts for Berge type VII.

{\bf Proposition \ref{p: berge viii}.}  Here $n = t$.

\noindent $B' = \{ v_{i_\lambda}^\star,\dots, v_{i_1},\overline{v}_{j_1},\dots,v_{j_\mu}^\star \}$ if $i_\lambda = 1$, where $j_1 = n$ and $\overline{v}_n = \overline{v}_t = v_t -(v_{t-1} + \cdots + v_1)$;

\noindent $C' = S \cup \{v_n', v_{n+1}' \} - \{ v_n \}$, where $v_n' = \overline{v}_n - e_{n-1}$ and $v_{n+1}' = e_{n-1}+e_n$;

\noindent $\nu = (a_1,\dots,a_l+b_m + 1, \dots, b_1)$;

\noindent $p = p_l^2 + p_l p_{l-1} - p_{l-1}^2$ by \ref{l: cont frac basics 2}(9);

\noindent $k = p_l r_l +p_{l-1} r_l - p_{l-1} r_{l-1}$ by \ref{l: cont frac basics 2}(10); and the difference between the numerator and denominator in $[\nu]^-$ is $D = r_l^2 + r_l r_{l-1} - r_{l-1}^2$.

\noindent Now we use the identity $(a^2+ab-b^2)(c^2+cd-d^2) = (e^2+ef-f^2)$, where $e = ac-bd+bc$ and $f = ad-bc$.   As before, we apply it with $a = p_l, b = p_{l-1}, c = r_l$, and $d = r_{l-1}$, noting that $f = 1$.  It follows that $p \cdot D = k^2 + k - 1$.  Again, the same conclusion holds if instead $j_\mu = 1$.  Replacing $k$ by $-k$, it follows in summary that \ref{p: berge viii} accounts for Berge type VIII.


\subsection{Small families}\label{ss: small}  

{\scriptsize

\begin{table}
\caption{Small families and their Berge types}
\label{table: 1}
\begin{center}

\makebox[0pt]{

\begin{tabular}{c | l l}

Prop$^{\text{n}}$. & Berge type & $B$ and $B^\star$; $\nu$; $k$, $p$ \\[.1cm]
\hline \hline \\[-.3cm]
\ref{p: just right 1}(1)
& X, $k \leq 11(-3)+3$ &
$\{-v_s^\star,-v_{s+2},v_{s+3},v_{s-1},\dots,v_1^\star, -(v_{s+1}+v_{s-1}+\cdots+v_1)^\star$\} \\[.1cm]
&& $(a+1,3,5,2^{[a-1]},3)$, $a=s \geq 2$ \\[.1cm]
&& $k = 11(-a-1)+3$, $p=(2k^2+k+1)/11$ \\[.1cm]
\cline{2-3} \\[-.3cm]
\ref{p: just right 1}(2)
& IX, $k \leq 11(-3)+2$ &      
$\{-v_s^\star, -v_{s+3}, v_{s+2}, v_{s-1}, \dots, v_1^\star, -(v_{s+1}+v_{s-1}+\cdots+v_1)^\star\}$ \\[.1cm]
&& $(a+1,4,4,2^{[a-1]},3)$, $a=s \geq 2$ \\[.1cm]
&& $k=11(-a-1)+2$, $p= (2k^2+k+1)/11$ \\[.1cm]
\cline{1-3} \\[-.3cm]

\ref{p: just right 2}(1)
& X, $k = 11(-2)+3$ & $\{-v_1^\star,-v_3,v_4^\star,-v_2^\star\}$ \\[.1cm]
&& $(2,3,5,3)$ \\[.1cm]
&& $k = -19$, $p=(2k^2+k+1)/11=64$ \\[.1cm]
\cline{2-3} \\[-.3cm]
\ref{p: just right 2}(2)
& IX, $k = 11(-2)+2$ & $\{-v_1^\star,-v_4,v_3^\star, -v_2^\star\}$ \\[.1cm]
&& $(2,4,4,3)$ \\[.1cm]
&& $k=-20$, $p=(2k^2+k+1)/11=71$ \\[.1cm]
\cline{2-3} \\[-.3cm] 
\ref{p: just right 2}(3)
& IX, $k \geq 11(2)+2$ &
$\{-v_1^\star, \dots, -v_{s-1}, -v_{s+3}, v_{s+2}, v_s^\star, v_{s+1}\}$ \\[.1cm]
&& $(2^{[a-1]},5,3,a+1,2), a=s \geq 2$ \\[.1cm]
&& $k=11a+2$, $p= (2k^2+k+1)/11$ \\[.1cm]
\cline{2-3} \\[-.3cm] 
\ref{p: just right 2}(4)
& X, $k \geq 11(2)+3$ &
$\{-v_1^\star,\dots,-v_{s-1},-v_{s+2},  v_{s+3},  v_s^\star,v_{s+1}\}$ \\[.1cm]
&& $(2^{[a-1]},4,4,a+1,2)$, $a=s \geq 2$ \\[.1cm]
&& $k=11a+3$, $p=(2k^2+k+1)/11$ \\[.1cm]
\cline{2-3} \\[-.3cm] 
\ref{p: just right 2}(5)
& IV(b)$_-, d = 3, {2k-1 \over d} \geq 5$ 
& $\{v_2, v_1^\star, v_m, -v_3^\star, \dots, -v_{m-1},-v_{m+1},  \dots,-v_n\}$ \\[.1cm]
& {\em and} & $(2,2,a+3, 4, 2^{[a-1]}, 3,  2^{[b-1]})$,  $a = m-3 \geq 1, b = n-m \geq 0$ \\[.1cm]
& V(a)$_-, d = 3, {k+1 \over d} \geq 3$ & $k=3a+5$, $p=(b+1)k^2-3(k+1)$ \\[.1cm]
\cline{1-3} \\[-.3cm]

\ref{p: just right 3}(1)
& any type with $k=1$ & $\{v_1^\star, \dots,v_n\}$ \\[.1cm]
& & $(2^{[n]})$ \\[.1cm]
&& $k=1$, $p=n+1$ \\[.1cm]
\cline{2-3} \\[-.3cm]
\ref{p: just right 3}(2)
& III(a)$_-$, $d=2, {k+1 \over d}=3$ & $\{-v_2,-v_1^\star,v_3,v_4^\star,\dots,v_n\}$ \\[.1cm]
& {\em and} & $(2,2,3,5,2^{[a-1]})$, $a \geq 1$ \\[.1cm]
& V(a)$_-$, $d=3, {k+1 \over d} = 2$& $k = 5$, $p = 25(a+1)-18$ \\[.1cm]
\cline{1-3} \\[-.3cm] 
\ref{l: S_g nothing}(1)
& IV(b)$_-, d \geq 5, {2k-1 \over d} \geq 5$ &
$\{v_s^\star,v_{s+1},-(v_g+v_1+\cdots+v_{s-1}+v_{s+1})^\star, v_1^\star,\dots,v_{s-1},v_{s+2},\dots,v_{g-1},  v_{g+1},\dots,v_n\}$ \\[.1cm]
&& $(a+1,2,b+3, 2^{[a-1]}, 4,2^{[b-1]},3,2^{[c-1]})$,  $a=s \geq 2, b=g-s-2 \geq 1, c=n-g \geq 0$  \\[.1cm]
&& $k=2ab+3a+b+2$, $p=(c+1)k^2-  (2a+1)(k+1)$ \\[.1cm]
\cline{2-3} \\[-.3cm]
\ref{l: S_g nothing}(3)
& IV(a)$_-, d \geq 5, {2k+1 \over d} \geq 5$ &
$\{-v_1^\star,\dots,-v_{s-1},-v_{s+1},v_g,v_s^\star, v_{s+2},  \dots,v_{g-1},v_{g+1}, \dots, v_n\}$ \\[.1cm]
&& $(2^{[a-1]},3,b+2,a+1, 3,2^{[b-1]},3,2^{[c-1]})$,  $a=s \geq 2, b=g-s-2 \geq 1, c=n-g \geq 0$ \\[.1cm]
&& $k=2ab+3a+b+1$, $p=(c+1)k^2 -  (2a+1)(k-1)$ \\[.1cm]
\cline{2-3} \\[-.3cm]
\ref{l: S_g nothing}(4)
& V(a)$_-, d \geq 5, {k+1 \over d} \geq 3$ &
$\{-v_j,-v_1^\star,-v_g,v_2^\star,\dots,v_{j-1}, v_{j+1}, \dots, v_{g-1},v_{g+1},\dots,v_n\}$ \\[.1cm]
&& $(a+2,2,b+3,3,2^{[a-1]}, 3,2^{[b-1]},3,2^{[c-1]})$,  $a = j-2 \geq 1, b= g-j-1 \geq 1,  c=n-g \geq 0$ \\[.1cm]
&& $k=2ab+4a+3b+5$, $p=(c+1)k^2-  (2a+3)(k+1)$ \\[.1cm]
\cline{2-3} \\[-.3cm]
\ref{l: S_g nothing}(5)
& V(b)$_-, d \geq 3, {k-1 \over d} \geq 2$ &
$\{-v_{j-1},\dots,-v_2^\star,v_g,v_1^\star,v_j,  \dots,v_{g-1},v_{g+1},\dots,v_n\}$ \\[.1cm]
&& $(2^{[a-1]},3,b+2,2,a+2,2^{[b-1]},3,2^{[c-1]})$,  $a=j-2 \geq 1, b=g-j \geq 1, c=n-g \geq 0$ \\[.1cm]
&& $k=2ab+2a+b+2$, $p=(c+1)k^2-  (2a+1)(k-1)$

\end{tabular}

} 

\end{center}

\end{table}

} 


{\scriptsize

\begin{table}
\caption{Small families (cont$^\text{d}$)}
\label{table: 2}
\begin{center}

\makebox[0pt]{

\begin{tabular}{c | l l}

\hline \hline \\[-.3cm]

\ref{l: S_g not nothing}(1)
& III(a)$_-, d \geq 3, {k+1 \over d} \geq 5$ &
$\{v_s^\star,v_g,-v_{s+1},-v_{s-1},\dots,-v_1^\star, (v_{s+2}+v_1+\cdots+v_{s-1})^\star, \dots,  v_{g-1}, v_{g+1},\dots,v_n\}$ \\[.1cm]
&& $(a+1,b+2,3,2^{[a-1]}, 4,2^{[b-1]},3,2^{[c-1]})$,  $a=s \geq 2, b=g-s-2 \geq 1, c=n-g \geq 0$ \\[.1cm]
&& $k=2ab+3a+2b+2$, $p=(c+1)k^2 -  (a+1)(2k-1)$ \\[.1cm]
\cline{2-3} \\[-.3cm]
\ref{l: S_g not nothing}(2)
& III(a)$_-, d=2, {k+1 \over d} \geq 5$ &
$\{v_1^\star,v_g,-v_2^\star,v_3^\star,\dots,v_{g-1},v_{g+1},\dots,v_n\}$ \\[.1cm]
&& $(2,a+2,3,4,2^{[a-1]}, 3,2^{[b-1]})$,  $a = g-3\geq 1, b = n-g\geq 0$ \\[.1cm]
&& $k=4a+5$, $p=(b+1)k^2 -  2(2k-1)$ \\[.1cm]
\cline{2-3} \\[-.3cm]
\ref{l: S_g not nothing}(3)
& III(b)$_-, d \geq 2, {k-1 \over d} \geq 5$ &
$\{-v_1^\star,\dots,-v_{s-1},-v_g,-v_{s+1}, (v_s+v_{s+1})^\star, v_{s+2},\dots,v_{g-1}, v_{g+1},\dots, -v_n\}$ \\[.1cm]
&& $(2^{[a-1]},b+3,2,a+1,3, 2^{[b-1]},3,2^{[c-1]})$,  $a=s \geq 2, b=g-s-2 \geq 1, c=n-g \geq 0$ \\[.1cm]
&& $k=2ab+3a+1$, $p=(c+1)k^2 -  a(2k+1)$ \\[.1cm]
\cline{1-3} \\[-.3cm]

\ref{p: gappy structure}(3)
& III(a)$_-, d \geq 2, {k+1 \over d} = 3$ &
$\{-v_s^\star,-v_{s+1},-(v_{s+2}-v_1-\cdots-v_{s-1})^\star, v_1^\star, \dots, v_{s-1}, v_{s+3}, \dots, v_n\}$ \\[.1cm]
& {\em and} & $(a+1,2,3,2^{[a-1]},5,2^{[b-1]})$,  $a=s \geq 2, b=n-s-2 \geq 0$ \\[.1cm]
& IV(b)$_-, d \geq 3, {2k-1 \over d} =3$ & $k=3a+2$, $p=(b+1)k^2-  (k+1)(2k-1)/3$ \\[.1cm]
\cline{1-3} \\[-.3cm]

\ref{p: breakable 1}(1)
& III(b)$_+$, $d = 2$, ${k-1 \over d} \geq 3$ &
$\{-v_3^\star,\dots,-v_{m-1},-(v_1-v_3-\cdots-v_{m-1})^\star, v_2^\star,v_m,\dots,v_n\}$ \\[.1cm]
&  &  $(3,2^{[a-1]},4,3,a+2,2^{[b-1]})$,  $a=m-3 \geq 1, b=n-m+1 \geq 0$ \\[.1cm]
&  & $k=4a+3$, $p=bk^2 + 2(2k+1)$ \\[.1cm]
\cline{2-3} \\[-.3cm]
\ref{p: breakable 1}(2)
&  IV(a)$_+$, $d =5$, ${2k+1 \over d} \geq 5$ &
$\{-v_3^\star,-(v_1-v_3-v_4-\cdots-v_{m-1})^\star,-v_{m-1},\dots,-v_4^\star,v_2^\star,v_m,\dots,v_n\}$ \\[.1cm]
&  & $(3,3,2^{[a-1]},3,3,a+3, 2^{[b-1]})$,  $a=m-4 \geq 1, b=n-m+1 \geq 0$ \\[.1cm]
&  & $k=5a+7$, $p=bk^2 +  5(k-1)$ \\[.1cm]
\cline{1-3} \\[-.3cm]

\ref{p: breakable 2}(1)
& III(a)$_+$, $d = 2$, ${k+1 \over d} \geq 5$ &
$\{v_4,\dots,v_{m-1},(v_1-v_3 -v_4-\cdots - v_{m-1}), (v_3+v_2)^\star,-v_2,v_m,\dots,v_n\}$ \\[.1cm] 
&  & $(3,2^{[a-1]},3,3,2,a+3,2^{[b-1]})$,  $a = m-4 \geq 1, b=n-m+1 \geq 0$ \\[.1cm]
&  & $k=4a+5$, $p=bk^2 +   2(2k-1)$ \\[.1cm]
\cline{2-3} \\[-.3cm] 
\ref{p: breakable 2}(2)
& V(a)$_+$, $d \geq 3$, ${k+1 \over d} \geq 2$ &
$\{v_s,\dots,v_2,(v_1-v_2-\cdots-v_s-v_{s+2})^\star, v_{m-1},\dots,v_{s+2}^\star,v_{s+1},v_m,\dots, v_n\}$ \\[.1cm]
&  & $(2^{[a-1]},4,2^{[b-1]},3,a+1, b+2,2^{[c-1]})$,  $a=s \geq 1, b=m-s-2 \geq 1,  c=n-m+1 \geq 0$ \\[.1cm]
&  & $k=2ab+2a+b$, $p=ck^2 +   (2a+1)(k+1)$ \\[.1cm]
\cline{2-3} \\[-.3cm]
\ref{p: breakable 2}(3)
& V(b)$_+$, $d \geq 3$, ${k-1 \over d} \geq 2$ &
$\{v_{s+1},v_{s+2}^\star,\dots,v_{m-1}, (v_1-v_{s+2}-\cdots-v_{m-1})^\star, \dots, v_s,v_m,\dots,v_n\}$ \\[.1cm]
&  & $(a+1,3,2^{[b-1]}, 4, 2^{[a-1]},b+3, 2^{[c-1]})$,  $a=s \geq 1, b=m-s-2 \geq 1,  c=n-m+1 \geq 0$ \\[.1cm]
&  & $k=2ab+2a+b+2$, $p=ck^2 +   (2a+1)(k-1)$ \\[.1cm]
\cline{1-3} \\[-.3cm]

\ref{p: breakable 3}(1)
& III(a)$_+$, $d \geq 3$, ${k+1 \over d} \geq 3$ &
$\{v_1^\star,\dots,v_{t-1},v_{t+3},\dots,v_{m-1},  v_t-v_1-\cdots-v_{t-1}-v_{t+2}-\cdots-v_{m-1},$ \\[.1cm]
&& $v_{t+2}^\star,v_{t+1},v_m,\dots,v_n\}$ \\[.1cm]
&   & $(2^{[a-1]},3,2^{[b-1]},3,a+2,2,b+3,2^{[c-1]})$,  $a = t \geq 2, b = m-t-3 \geq 0,  c=n-m+1 \geq 0$ \\[.1cm]
&  & $k=2ab+3a+2b+2$, $p=ck^2 +   (a+1)(2k-1)$ \\[.1cm]
\cline{2-3} \\[-.3cm]
\ref{p: breakable 3}(2)
& IV(a)$_+$, $d \geq 7$, ${2k+1 \over d} \geq 3$ &
$\{-v_{t+2}^\star,-v_t + v_1 + \cdots +v_{t-1}+v_{t+2}+\cdots+v_{m-1}, -v_{m-1},\dots,$ \\[.1cm]
&& $(-v_{t+3}+v_1+\cdots+v_{t-1})^\star,v_1^\star,\dots,v_{t-1}, v_{t+1},v_m,\dots,v_n\}$ \\[.1cm]
&  & $(a+2,3,2^{[b-1]},3,2^{[a-1]},3,b+3,2^{[c-1]})$,  $a=t \geq 2, b=m-t-3 \geq 0,  c = n-m+1\geq 0$\\[.1cm]
&  & $k=2ab+3a+3b+4$, $p=ck^2 + (2a+3)(k-1)$ \\[.1cm]
\cline{2-3} \\[-.3cm]
\ref{p: breakable 3}(3)
& IV(b)$_+$, $d \geq 5$, ${2k-1 \over d} \geq 3$ &
$\{-v_{t-1},\dots,-v_1^\star, (-v_t+v_{t+2}+ \cdots+v_{m-1})^\star, v_{m-1}, \cdots, v_{t+2}^\star,v_{t+1}, v_m,\dots, v_n\}$ \\[.1cm]
&  & $(2^{[a-1]},4,2^{[b-1]},a+2,2,b+2,2^{[c-1]})$,  $a=t \geq 2, b = m-t-2 \geq 1,  c=n-m+1 \geq 0$ \\[.1cm]
&  & $k=2ab+a+b+1$, $p=ck^2 + (2a+1)(k+1)$ \\[.1cm]
\cline{2-3} \\[-.3cm]
\ref{p: breakable 3}(4)
& III(b)$_+$, $d \geq 3$, ${k-1 \over d} \geq 3$ &
$\{-v_{t+2}^\star,\dots,-v_{m-1}, (-v_t+v_{t+2}+\cdots+v_{m-1})^\star, v_1^\star, \dots, v_{t-1},v_{t+1},v_m, \dots, v_n\}$ \\[.1cm]
&   & $(a+2,2^{[b-1]},4,2^{[a-1]},3,b+2,2^{[c-1]})$,  $a=t \geq 2, b=m-t-2 \geq 1,  c = n-m+1 \geq 0$ \\[.1cm]
&  & $k=2ab+2a+b+2$, $p=ck^2+(a+1)(2k+1)$

\end{tabular}

} 

\end{center}

\end{table}

} 

The 26 small families fall to a straightforward, though somewhat lengthy, analysis.  In each case, converting the standard basis $S$ into a vertex basis $B$ usually involves altering just one element from $S$ into a sum of several such, and then permuting these elements and replacing some of them by their negatives.  In two cases (\ref{p: breakable 2}(1) and \ref{p: breakable 3}(2)) there are two such alterations, and in a handful there are none.

From the sequence of norms $\nu$, it is straightforward to obtain the values $p$ and $k$ as in the example of Subsection \ref{ss: methodology}.  Lemmas \ref{l: cont frac basics 2}(1,2,3) help reduce the number of terms appearing in the continued fraction expansions under consideration; note that although Lemma \ref{l: cont frac basics 2}(3) relates two different fractions, their numerators are opposite one another.  In this way, we reduce each string to one with at most three variables ($a,b,c$) and eight entries, which a computer algebra package or a tenacious person can evaluate.  We used Mathematica \cite{mathematica} to perform these evaluations, relying on the command {\tt FromContinuedFraction} and the conversion $[\dots,a_i, \dots]^- = [\dots,(-1)^{i+1}a_i, \dots]^+$.  Tables \ref{table: 1} and \ref{table: 2} report the results.  We use variables $a,b,c$ (instead of $g,m,n,s,t$) to keep notation uniform across different families.  As in Table \ref{table: A}, we report Lemmas \ref{l: S_g nothing}(1,3,4,5) and \ref{l: S_g not nothing} in place of  Proposition \ref{p: gappy structure}(2).

Certain degenerations in our notation deserve mention.  The string $(\dots,x,y,2^{[-1]})$ should be understood as $(\dots,x)$.  Thus, in \ref{p: just right 2}(5), taking $b=0$, we obtain the string $(2,2,a+3,4,2^{[a-1]})$.  Furthermore, it follows that $n = m-1$ in this case, so the vertex basis truncates to $\{v_2,v_1,v_m,-v_3,\dots,-v_{m-1}\}$.  In \ref{p: breakable 3}(1) and (2), there are two degenerations that can occur.  In these cases, the degeneration $b=0$ can occur only if $c=0$.
If $b = c = 0$, then we obtain the strings $(2^{[a-1]},4,a+2,2)$ and $(a+2,4,2^{[a-1]},3)$, respectively.


\section{Proofs of the main results}\label{s: completion}

Recall that we established Theorem \ref{t: main technical} in Subsection \ref{ss: changemaker} using \cite[Theorem 3.3]{greene:cabling}.  

\begin{proof}[Proof of Theorem \ref{t: linear changemakers}]

Suppose that $(p,q)$ appears on Berge's list.  Then $\Lambda(p,q)$ embeds as the orthogonal complement to a changemaker by Theorem \ref{t: main technical}.  On the other hand, suppose that $\Lambda(p,q)$ is isomorphic to a changemaker lattice $L =(\sigma)^\perp \subset \Z^{n+1}$.  Then $L$ has a standard basis $S$ appearing in one of the structural Propositions of Sections \ref{s: just right} - \ref{s: tight}.  Section \ref{s: cont fracs} in turn exhibits an isomorphism $L \cong \Lambda(p',q')$, where the pair $(p',q')$ appears on Berge's list.  By Proposition \ref{p: gerstein}, $p' = p$, and either $q' = q$ or $q q' \equiv 1 \pmod p$.  Hence at least one of the pairs $(p,q)$, $(p,q')$ appears on Berge's list.

\end{proof}

\begin{proof}[Proof of Theorem \ref{t: main}]

This follows from Theorems \ref{t: main technical} and \ref{t: linear changemakers}, using Proposition \ref{p: homology} and the analysis of Section \ref{s: cont fracs} to pin down the homology class of the knot $K$.

\end{proof}

We note that the statement of Theorem \ref{t: main} holds with $S^3$ replaced by an arbitrary L-space homology sphere $Y$ with $d$-invariant $0$.  The only modification in the set-up is to use the 2-handle cobordism $W$ from $L(p,q)$ to $Y$.  The space $X(p,q) \cup W$ is negative-definite and has boundary $Y$.  By \cite[Corollary 9.7]{os:absgr} and Elkies' Theorem \cite{elkies}, its intersection pairing is diagonalizable, so \cite[Theorem 3.3]{greene:cabling} and the remainder of the proof go through unchanged.

\begin{proof}[Proof of Theorem \ref{c: main}]

Suppose that $K_p = L(p,q)$, and let $K' \subset L(p,q)$ denote the induced knot following the surgery.  By Theorem \ref{t: main}, $[K'] = [B'] \in H_1(L(p,q);\Z)$ for some Berge knot $B \subset S^3$.  By \cite[Theorem 2]{r:Lspace}, it follows that $\widehat{HFK}(K') \cong \widehat{HFK}(B')$.  By \cite[Proposition 3.1 and the remark thereafter]{r:Lspace}, it follows that $\Delta_{K'} = \Delta_{B'}$, where $\Delta$ denotes the Alexander polynomial.  Since $\Delta$ depends only on the knot complement, it follows that $\Delta_K= \Delta_B$.  By \cite[Theorem 1.2]{os:lens}, $\Delta_K$ and $\Delta_B$ determine $\widehat{HFK}(K)$ and $\widehat{HFK}(B)$; therefore, these groups are isomorphic.

Next, suppose that $K$ is doubly primitive.  As remarked in the introduction, both $K'$ and $B'$ are simple knots, and since they are homologous, they are isotopic.  Thus, the same follows for $K$ and $B$, whence every doubly primitive knot is a Berge knot.

\end{proof}

\begin{proof}[Proof sketch of Theorem \ref{t: berge bound}]  The main idea is to analyze the changemakers implicit in the structural Propositions and apply Proposition \ref{p: genus}, which restates the essential content of \cite[Proposition 3.1]{greene:cabling}.  The use of weight expansions draws inspiration from \cite{ms:ellipsoids}.

\begin{prop}\label{p: genus}

Suppose that $K_p =L(p,q)$, and let $\sigma$ denote the corresponding changemaker.  Then
\[ 2g(K) = p - |\sigma|_1. \] \qed

\end{prop}

\begin{defin}\label{d: wt expansion}

A {\em weight expansion} is a vector of the form 

\[w = (\underbrace{a_0,\dots,a_0}_{m_0}, \underbrace{a_1,\dots,a_1}_{m_1}, \dots, \underbrace{a_j,\dots,a_j}_{m_j}) \] where each $m_i \geq 1$, $a_{-1}:=0, a_0=1$, and $a_i = m_{i-1} a_{i-1} + a_{i-2}$ for $i = 1,\dots,j$.

\end{defin}

It is an amusing exercise to show that the entries of $w$ form the sequence of side lengths of squares that tile an $a_j \times a_{j+1}$ rectangle.
\begin{figure}[h]
\centering
\includegraphics[width=2.8in]{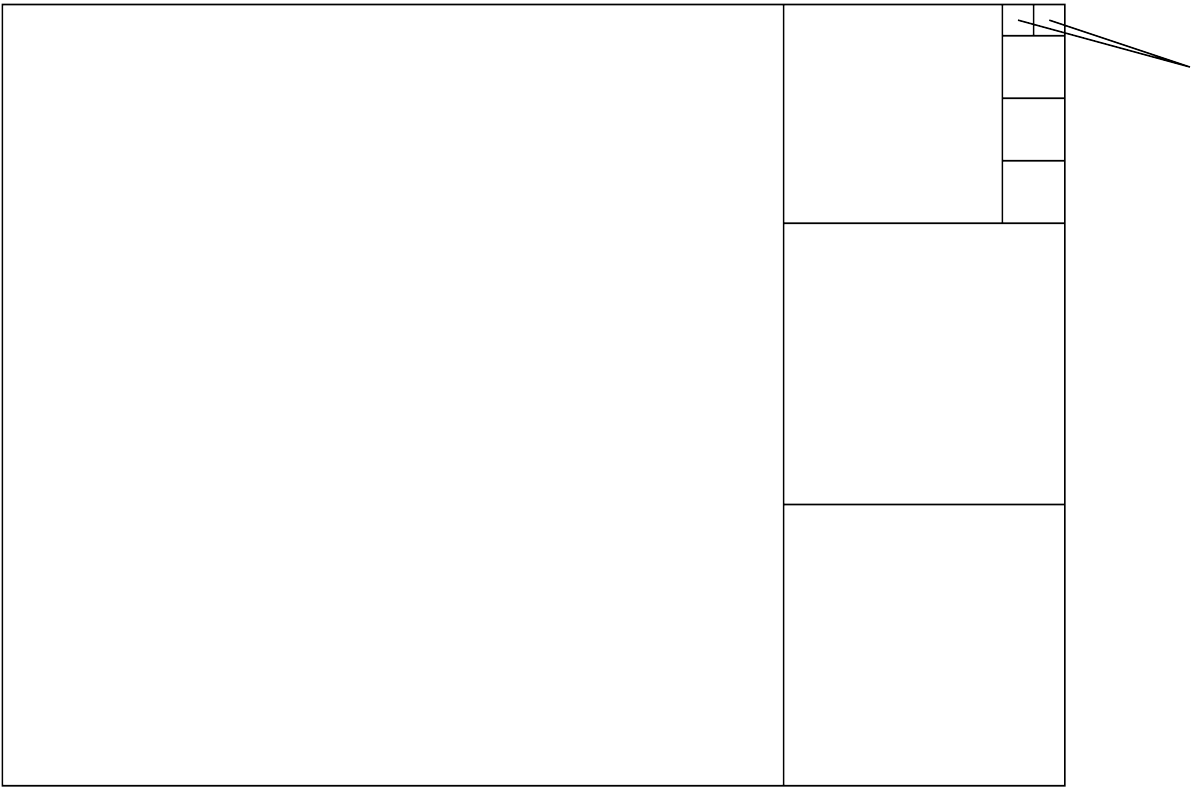}
\put(-135,65){$25$}
\put(-48,70){$9$}
\put(-48,20){$9$}
\put(-53,110){$7$}
\put(-30,118){$2$}
\put(-30,107.5){$2$}
\put(-30,97){$2$}
\put(0,118){$1$}
\caption{The tiling specified by the weight expansion $w = (1,1,2,2,2,7,9,9,25)$.} \label{f: rectangle}
\end{figure}
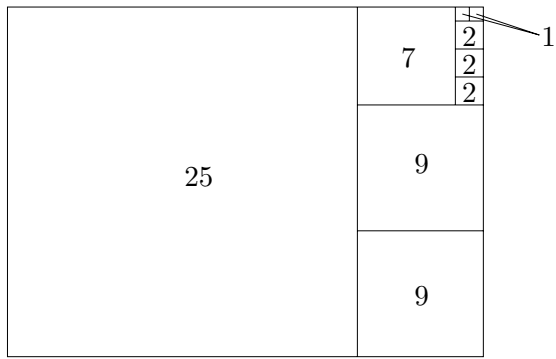
%


\noindent Another useful observation is $a_t = 1 + \sum_{i=0}^{t-1} m_i a_i - a_{t-1}$, for all $t \geq 1$.  Thus, 
\begin{equation}\label{e: wt exp}
|w| = a_j \cdot a_{j+1} \quad \text{and} \quad |w|_1 = a_j + a_{j+1} - 1,
\end{equation}
which together with $a_{j+1} \geq a_j+1$ leads to the bound
\begin{equation}\label{e: wt bound}
(|w|_1 + 1)^2 \geq 4 |w| + 1.
\end{equation}

Observe that a weight expansion is a special kind of changemaker.  In fact, it is easy to check that 
a changemaker $\sigma$ is a weight expansion iff the changemaker lattice $L = (\sigma)^\perp \subset \Z^{n+1}$ is built from $\varnothing$ by a sequence of expansions.  Such lattices occur as one case of Proposition \ref{p: decomposable structure}.  Indeed, by inspection, for each changemaker lattice that appears in one of the structural Propositions of Sections \ref{s: decomposable} - \ref{s: tight}, the changemaker $\sigma$ is just a slight variation on a weight expansion.  For example, the changemakers implicit in Proposition \ref{p: berge viii} are obtained by augmenting a weight expansion by $a_j + a_{j+1}$, while those in Proposition \ref{p: tight unbreakable}(1,3) are obtained by deleting the first entry in a weight expansion with $m_0 \geq 2$.

Using Proposition \ref{p: genus}, we obtain estimates on the genera of knots appearing in these families.  For the changemakers $\sigma$ specified by Proposition \ref{p: berge viii}, \eqref{e: wt exp} easily leads to the inequality
\begin{equation}\label{e: jimmy}
(|\sigma|_1 + 1)^2 \geq (4/5) \cdot (4|\sigma|+1)
\end{equation}
in the same way as \eqref{e: wt bound}.  Furthermore, equality in \eqref{e: jimmy} occurs precisely for changemakers $\sigma$ of the form $(1,\dots,1,n,2n+1)$, with $1$ repeated $n$ times.  By Proposition \ref{p: genus}, it follows that the bound \eqref{e: berge bound} stated in Theorem \ref{t: berge bound} holds for type VIII knots, with equality attained precisely by knots $K$ specified by the pairs $(p,k) = (5n^2+5n+1,5n^2-1)$.

Similarly, for the changemakers $\sigma$ specified by Proposition \ref{p: tight unbreakable}(1,3), \eqref{e: wt exp} easily leads to the bound $(|\sigma_1| + 2)^2 \geq 4|\sigma|+5$.  Furthermore, equality occurs precisely for changemakers $\sigma$ of the form $(1,\dots,1,n+1)$, with $1$ repeated $n$ times.  By Proposition \ref{p: genus}, it follows that the bound
\begin{equation}\label{e: I- bound}
2g(K)-1 \leq p + 1 - \sqrt{4p+5}
\end{equation}
holds for type I$_-$ knots, with equality attained precisely by knots $K$ specified by the pairs $(p,k) = (n^2 + 3n + 1, n+1)$.  In fact, \eqref{e: jimmy} holds for all the changemakers of Proposition \ref{p: tight unbreakable}(1,3) with the single exception of $(1,2)$.  This corresponds to $5$-surgery along a genus one L-space knot, which must be the right-hand trefoil by a theorem of Ghiggini \cite{ghiggini}.  Thus, the bound \eqref{e: berge bound} holds for all the type I$_-$ knots with the sole exception of $5$-surgery along the right-hand trefoil.

The changemakers in the other structural Propositions fall to the same basic analysis.  Due to the abundance of cases, we omit the details, and instead happily report that the bound \eqref{e: berge bound} is strict for the remaining lens space knots.  This completes the proof sketch of Theorem \ref{t: berge bound}.

\end{proof}


\nocite{ni:erratum}

\bibliographystyle{plain}
\bibliography{/Users/Josh/Desktop/Papers/References}

\end{document}